\documentclass[11pt, reqno]{amsart}
\usepackage{amsthm, amsmath}
\usepackage{amssymb} 
\usepackage{graphicx}
\usepackage{epsf}

\numberwithin{equation}{section}

\usepackage[colorlinks=true,pagebackref,hyperindex]{hyperref}
\usepackage{mathrsfs} 
\usepackage[all]{xy}
\usepackage{color}
\usepackage{amsfonts}
\usepackage{amscd}

\definecolor{darkgreen}{rgb}{0.0, 0.7, 0.0}
\newenvironment{TO}{\noindent \color{darkgreen}{\bf TOM:} \footnotesize}{}
\definecolor{cyan}{cmyk}{1,0,0,0}
\newenvironment{MD}{\noindent \color{blue}{\bf MARK:} \footnotesize}{}

\newcommand{\cb}{\color{blue}}
\newcommand{\cm}{\color{magenta}}

\newcommand{\bdg}{\begin{dg}}
\newcommand{\edg}{\end{dg}}

\newtheorem{tm}{Theorem}[subsection]
\newtheorem{lm}[tm]{Lemma}
\newtheorem{pr}[tm]{Proposition}
\newtheorem{rmk}[tm]{Remark}
\newtheorem{cor}[tm]{Corollary}
\newtheorem{ex}[tm]{Example}

\newtheorem{fact}[tm]{Fact}
\newtheorem{??}[tm]{Question}
\newtheorem{???}[tm]{Questions}
\newtheorem{defi}[tm]{Definition}

\newtheorem{conj}[tm]{Conjecture}

\newcommand{\ben}{\begin{enumerate}}
\newcommand{\een}{\end{enumerate}}
\newcommand{\bit}{\begin{itemize}}
\newcommand{\eit}{\end{itemize}}
\newcommand{\beq}{\begin{equation}}
\newcommand{\eeq}{\end{equation}}
\newcommand{\la}{\label}
\newcommand{\n}{\noindent}
\newcommand\ci{\cite}

\font\tenmsb=msbm10
\font\sevenmsb=msbm7
\font\fivemsb=msbm5
\newfam\msbfam
\textfont\msbfam=\tenmsb
\scriptfont\msbfam=\sevenmsb
\scriptscriptfont\msbfam=\fivemsb
\def\Bbb#1{{\fam\msbfam #1}}
\font\teneufm=eufm10
\font\seveneufm=eufm7
\font\fiveeufm=eufm5
\newfam\eufmfam
\textfont\eufmfam=\teneufm
\scriptfont\eufmfam=\seveneufm
\scriptscriptfont\eufmfam=\fiveeufm

\newcommand{\im}{ \hbox{\rm Im} }

\newcommand{\lorw}{\longrightarrow}

\newcommand\oql{\overline{\mathbb Q}_\ell}

\newcommand\zed{{\mathbb Z}}

\newcommand\pn[1]{{\mathbb P}^{#1}}

\newcommand{\w}[1]{\widetilde{#1}}

\newcommand{\ov}[1]{\overline{#1}}

\newcommand{\m}[1]{\mathcal{#1}}

\newcommand{\bb}[1]{\Bbb{#1}}

\newcommand{\tu}[1]{ \tau_{ \geq {#1} } }

\newcommand{\G}{\mathcal{G}}
\newcommand{\PP}{\mathcal{P}}
\newcommand{\Q}{\mathcal{Q}}
\newcommand{\B}{\mathcal{B}}
\newcommand{\W}{\mathcal{W}}

\newcommand{\pwp}{^{\PP}\!\W^{\PP}}

\newcommand{\xp}[1]{X_\PP ({#1})}
\newcommand{\db}[1]{D^b_m({#1}, \oql)}
\newcommand\ic[1]{\m{IC}_{#1}}

\newcommand{\bwp}[2]{_{{#1}}\m{W}_{{#2}}}

\newcommand{\xbp}[2]{X_{#1}(#2)}

\newcommand{\good}{good}

\title{Frobenius semisimplicity for convolution morphisms}
\author{Mark Andrea de Cataldo}
\address{Department of Mathematics\\ 
Stony Brook University\\ 
Stony Brook, NY 11794, U.S.A.}
\email{mark.decataldo@stonybrook.edu}
\author{Thomas J. Haines}
\address{University of Maryland\\
Department of Mathematics\\
College Park, MD 20742-4015 U.S.A.}
\email{tjh@math.umd.edu}
\author{Li Li}
\address{Department of Mathematics and Statistics\\ 
Oakland University, Rochester\\
MI 48309 U.S.A.}
\email{li2345@oakland.edu}

\thanks{The research of M.A. de Cataldo was partially supported by NSF grant DMS-1301761 and by a grant from the Simons Foundation
($\#$296737 to Mark Andrea de Cataldo). The research of T.\,Haines was partially supported by NSF grant DMS-1406787. The research of L.\,Li was partially supported by the Oakland
University URC Faculty Research Fellowship Award.}

\begin{document}

\maketitle

\begin{abstract}
This article concerns properties of mixed $\ell$-adic complexes on
varieties over finite fields, related to the action of the Frobenius
automorphism. We establish a fiberwise criterion for the
semisimplicity and Frobenius semisimplicity of the direct image
complex under a proper morphism of varieties over a finite field. We
conjecture that the direct image of the intersection complex on the
domain is always semisimple and Frobenius semisimple; this conjecture
would imply that a strong form of the decomposition theorem of
Beilinson-Bernstein-Deligne-Gabber is valid over finite fields. We
prove our conjecture for (generalized) convolution morphisms
associated with partial affine flag varieties for split connected
reductive groups over finite fields. As a crucial tool, we develop a new schematic theory of big cells for loop groups. With suitable reformulations, the main results are valid
over any algebraically closed ground field. 
\smallskip

\end{abstract}

\tableofcontents


\markboth{Mark Andrea de Cataldo, \,Thomas J.\,Haines,\, and Li Li}{Frobenius semisimplicity for convolution morphisms}

\section{Introduction and terminology}\la{intro}


\subsection{Introduction}\la{intro}$\;$

Let $k$ be a finite field with a fixed algebraic closure $\ov{k}$, let  $f:X \to Y$ be a proper $k$-morphism of $k$-varieties, and let $P$ be a mixed and simple, hence pure, perverse sheaf on $X$;  we denote the situation
after passage to $\ov{k}$ by $\ov{f}: \ov{X} \to \ov{Y}$, $\ov{P}$. The decomposition theorem
\ci{bbd} holds over $\ov{k}$, i.e.,\,the direct image complex
$R\ov{f}_* \ov{P}$ on $\ov{Y}$ splits as a finite direct sum of shifted intersection cohomology complexes
$\m{IC}_{\ov{Z'}}(\ov{L'})$ associated with pairs $(Z',L')$, where, after having passed to a finite extension
$k'$ of $k$ if necessary, $Z'$ is a geometrically integral subvariety of $Y'=Y\otimes_k k'$, and $L'$ is a pure and  simple sheaf defined on a suitable Zariski-dense smooth open subvariety of $Z'$. We abbreviate the above as follows:
after passage to $\ov{k}$, the complex $R\ov{f}_* \ov{P}$ on $\ov{Y}$
is {\em semisimple}. 

 It is not known  whether  $Rf_* P$  is already semisimple
over $k$, i.e.,\,whether $Rf_*P$ splits  into a finite direct sum of shifted terms of the form 
$\m{IC}_Z(L)$ with $Z$ being $k$-integral and $L$  pure and simple. As pointed out in \ci[Prop.\,2.1]{dC},  this is true 
if we only ask that $L$ is indecomposable, rather than simple; the only obstruction to the  simplicity of an indecomposable $L$
is the a priori possible presence of Jordan-type sheaves; see Fact \ref{factjod}.  

A different, yet intimately  related question is: 
is the action of Frobenius on the stalks of the direct image sheaves  $R^if_* P$ semisimple?  In this case, we say that the complex $Rf_*P$ on $Y$  is {\em Frobenius 
semisimple}.  

General considerations related to the Tate conjecture over finite fields lead us to conjecture (see Conjecture \ref{conjic}) that the direct image complex $Rf_* \m{IC}_X$ on $Y$  is semisimple  and  Frobenius semisimple, where $\m{IC}_X$ is the  intersection complex of $X$.\footnote{We have normalized the intersection complex $\m{IC}_X$ of an integral variety $X$
so that if $X$ is smooth, then $\m{IC}_X \cong {\oql}_X$; this is not a perverse sheaf; the perverse sheaf counterpart
is $IC_X=\m{IC}_X [\dim X]$.} Note that this is not known even for $f= {\rm Id}_X.$ Moreover, a proof of our conjecture would imply 
the semisimplicity of the action of Frobenius on the cohomology of a smooth projective variety, which is also unknown in general. (It is known in some important special cases: Weil's proof of the Riemann Hypothesis for abelian varieties implies Frobenius semisimplicity for their cohomology groups, cf.\,\cite[p.\,203]{Mum}; Deligne proved the corresponding result for K3 surfaces, using a reduction to abelian varieties, cf.\,\cite[6.6]{Del}.)

In this paper, we  establish the validity of Conjecture \ref{conjic} in the case
of  Lusztig-type convolution morphisms associated with twisted products of  Schubert varieties
in partial (affine) flag varieties. The  validity of the conjecture in the case of proper toric morphisms of toric varieties
has already been established in \ci{dC}. 

Along the way, we prove other results, some of which are valid for any proper morphism,
and some of which  are specific to the context of twisted product varieties. 

Let us summarize the main
results of this paper. 

Theorem \ref{tma}:  the direct image $Rf_*\m{IC}_X$ is semisimple and Frobenius semisimple
if and only if Frobenius acts semisimply on the cohomology groups
of all closed fibers with coefficients in $\m{IC}_X.$

Theorem \ref{tmb}:  the intersection complex $\m{IC}_{f(X)}$ is a direct summand of $Rf_* \m{IC}_X$. 

Corollary \ref{conv}: the convolution complex $\m{IC}_{X_{\PP}(w_1)} * \cdots * \m{IC}_{X_{\PP}(w_r)}$ associated with a twisted product variety is semisimple and Frobenius semisimple. In fact, viewing this result as the (Frobenius) semisimplicity of a direct image complex of a convolution morphism, it holds for a larger class of 
convolution-type morphism, which we introduce and name generalized convolution morphisms; see Theorem \ref{dtm}. Note that we prove something stronger than semisimplicity and Frobenius
semisimplicity, namely evenness (no odd cohomology) and Tateness (the lisse  and pure coefficients are constant, up to  a 
precise Tate-twist).

The proof of Theorem \ref{dtm}, which deals with the direct image of the intersection complex by a generalized convolution morphism, is intertwined with the proof of analogous statements concerning the intersection cohomology groups of twisted product varieties;
see Theorem \ref{ctm}.

One of the key ingredients is the semisimplicity of the action of Frobenius
on the cohomology of the fibers. This is achieved in two very distinct ways.
The former is by means of affine paving results for the fibers of certain convolution morphisms; see Theorem
\ref{pavingtm}. The latter is by means of the  surjectivity for fibers Theorem \ref{tmff}.

The proof of the surjectivity Theorem \ref{tmff}, which is a geometric statement, is 
arithmetic in nature (it uses the yoga of weights) and it  is inspired by the Kazhdan-Lusztig observation linking
contracting $\mathbb G_m$-actions and purity. This idea has been exploited also
in the toric case in \ci{dC}. The necessary preparation, i.e.,\,the local product structure Lemma
\ref{ulemma}, relies on a new schematic theory of ``big cells'' adapted to partial affine flag varieties, which generalizes to partial affine flag varieties results of Beauville-Laszlo \cite{BLa} and Faltings \cite{F} for affine Grassmannians and affine flag varieties, respectively. In particular, we define the ``negative'' parahoric loop group $L^{--}P_{\bf f}$ (Definition \ref{neg_parahoric_defn}) and prove

Theorem \ref{big_cell_thm}: The map $L^{--}P_{\bf f} \times L^+P_{\bf f} \rightarrow LG$ is an open immersion.

Let us remark that the big cell in a Kac-Moody full flag variety has been constructed by completely different methods (for example, see the remarks after \cite[Lem.\,8]{Mat89}). It is not clear at all that the Kac-Moody construction could be used to define big cells in our context.  Indeed, we are working with the partial affine flag varieties $LG/L^+P_{\bf f}$, and $LG$ is not a Kac-Moody group unless $G$ is a simply-connected semisimple group. Of course, if $G_{\rm sc}$ is the simply-connected cover of the derived group $G_{\rm der}$, then $LG$ is closely related to the Kac-Moody group $LG_{\rm sc}$, and one might expect one could exploit this relationship to construct the big cells for $LG$. In fact it is even true that the Kac-Moody full flag variety constructed in \cite{Mat89} for $LG_{\rm sc}$ coincides as an ind-$k$-scheme with the object $LG_{\rm sc}/L^+P_{\bf a}$ we consider (although this is not obvious; see \cite[$\S9.h$]{PR}).  However, we found no way to reduce the construction of the {\em schematic} big cell for $LG$ to that for $LG_{\rm sc}$: just one issue is that the notion of parahoric subgroup in $LG$ is much more subtle than in $LG_{\rm sc}$, where there are no issues of disconnectedness of Bruhat-Tits group schemes (such issues are the subject of \cite{HRa}).  In this article we propose a self-contained construction of the schematic big cell in $LG$, using the key group ind-scheme $L^{--}P_{\bf f}$. Most of the geometric results about convolution morphisms hinge on properties of $L^{--}P_{\bf f}$ (such as the Iwahori-type decompositions $\S\ref{Iwah_type_sec}$).  These foundations for loop groups form a substantial part of this article. They made possible our rather efficient affine paving, contraction and surjectivity techniques. We also expect these foundations to be useful apart from Frobenius semisimplicity questions.

Finally, we mention:

Theorem \ref{beqdt}: ``explicit" form of the decomposition theorem
for generalized convolution morphisms. 

Some special and important cases of our Corollary \ref{conv} have already been proved
in works by Beilinson-Ginzburg-Soergel \cite{BGS}, Bezrukavnikov-Yun \cite{BY}, and Achar-Riche \cite{AR}. The relation to these papers is discussed in Remark \ref{inpiu} and in  \S\ref{comp_lit_sec}.

The paper is organized as follows.
In \S\ref{nota} and \S\ref{cmbtpv} we introduce (a minimal amount of) terminology and notation which will be used to state the main results of this paper in \S\ref{tempmr}. We review the background and establish  preliminary results on: affine groups and affine partial flag varieties (\S\ref{affgrps}), twisted product varieties (\S\ref{twcon_sec}),  geometric $\PP$-Demazure product (\S\ref{gpdem}) and its comparison with the standard Demazure product defined using the $0$-Hecke algebra (\S\ref{Dem_comp_subsec}),  connectivity of fibers of convolution morphisms (\S\ref{fibco}), and generalized convolution morphisms (\S\ref{098}). We develop our theory of big cells in partial affine flag varieties in $\S\ref{affgrps}$; in particular some new Iwahori-type decompositions are presented in $\S\ref{Iwah_type_sec}$. The proofs of our main results are then presented in \S\ref{abssq}, \S\ref{tmctmffpfs}, and \S\ref{pf??}. 
In \S\ref{abssq}, we prove two results which hold for any proper morphism, namely the (Frobenius) semisimplicity of the proper direct image  criterion (Theorem \ref{tma}) and that the intersection complex splits off (Theorem \ref{tmb}).
In \S\ref{tmctmffpfs}, we prove our surjectivity for fibers criterion (Theorem \ref{tmff}) and apply it to prove  Theorems \ref{ctm} and \ref{dtm}. 
In \S\ref{pf??}, we prove the affine paving of Demazure-type maps (Theorem \ref{pavingtm}), and use it to give a second proof of Corollary \ref{conv} which asserts, among other things, that the convolution product is even and  Tate.
In \S\ref{KM_rmks}, we make brief remarks about the Kac-Moody setting and explain the relation of our results with other works in the published literature.
In \S\ref{otherk}, we discuss how to view our results over other fields $k$.

\smallskip

{\bf Acknowledgments.}  We gratefully acknowledge discussions with Patrick Brosnan, Pierre Deligne, Xuhua He, Robert Kottwitz, Mircea Musta\c{t}\u{a}, George Pappas, Timo Richarz, Jason Starr, Geordie Williamson, and Zhiwei Yun.

\subsection{Frobenius semisimplicity and the notion of  \underline{\good}}\la{nota} $\,$

Unless stated otherwise, we work with separated schemes of finite type over a finite field 
$k$ (varieties, for short)  and with a fixed algebraic closure $k \subseteq \ov{k}$. We fix a prime  number $\ell \neq{\rm char} \, k$,  and we work with the associated
``bounded-derived-constructible-mixed" categories with the middle perversity
$t$-structure $\db{-} \subseteq D^b_c (-, \oql)$ in \ci{bbd}, whose objects we call complexes. 
Complexes, maps, etc. defined over $k$, can be pulled-back to $\ov{k}$, in which case
they are branded with a bar over them, e.g. a map of $k$-varieties $f: X \to Y$ pulls-back to $\ov{f}: \ov{X} \to \ov{Y},$ 
and  a complex $\m{F} \in \db{X}$ on $X$  pulls-back to the complex $\ov{\m F}$
on $\ov{X}.$  The stalks $\m{H}^*(\m{F})_{\ov x}$ of a complex $\m{F} \in  \db{x}$ at a point $\ov{x} \in 
X(\ov{k})$ are finite dimensional graded Galois $\oql$-modules endowed with a weight filtration.
In particular, so are the cohomology groups $H^*(\ov{X}, \ov{\m F})$. 
Unless otherwise stated, the direct image functor $Rf_*$ will be denoted simply by $f_*$.

We are especially interested in:
the intersection cohomology groups $I\!H^*(\ov{X}, \oql):=H^*(\ov{X}, \ov{\m{IC}_X})$, where $\m{IC}_X$ is the intersection complex of $X$, normalized so that, if $X$ is smooth and connected, then $\m{IC}_X= {\oql}_X$;
the cohomology groups $H^*(\ov{f^{-1} (y)}, \ov{\m{IC}_X})$, where $y$ is a closed point in $Y$.

Let $X$ be a $k$-variety.
We consider the following properties of complexes  $\m{F} \in \db{X}$: 
\begin{itemize}
\item[-] {\em semisimplicity}: $\m{F}$  is isomorphic to  the direct sum
of  shifted simple perverse sheaves (necessarily supported on integral closed subvarieties of $X$);
\item[-] {\em Frobenius semisimplicity}: the graded Galois modules $\m{H}^*(\m{F})_{\ov{x}}$ are semisimple
for every $\ov{x} \in X(\ov{k})$; 
\item[-] {\em evenness}: the $\m{H}^*(\m{F})_{\ov{x}}$ are even, i.e.,~trivial in odd cohomological degrees 
\item[-] {\em purity with weight $w$}:   $\m{H}^*(\m{F})_{\ov{x}}$ has weights  $\leq w+i$
in degree $i$ 
and  $\m{H}^*(\m{F}^\vee)_{\ov{x}}$ has weights  $\leq -w+i$ in degree $i$ ($\m{F}^\vee$ the Verdier dual);
\item[-] {\em very pure with weight $w$} \cite[$\S4$]{KL}: $\m{F}$ is pure with weight $w$ and the mixed
graded Galois module  $\m{H}^*(\m{F})_{\ov{x}}$  is  pure 
with weight $w$, i.e.,~it has  weight $w+i$ in degree $i$.\footnote{Equivalently, $\m{F}$ is pure of weight $w$ and each $\m{H}^i(\mathcal F)$ is {\em pointwise pure} of weight $w+i$ in the sense of \,\cite[p.\,126]{bbd}.} 
\item[-] {\em Tateness}: each $\m{H}^i(\m{F})_{\ov{x}}$ is isomorphic to a direct sum of Tate modules
 $\oql (-k)$ of possibly varying weights $2k.$ 
\end{itemize}
We also have the notions of 
Frobenius semisimple/even/pure/Tate finite dimensional Galois graded modules.  
According to our definition, a Tate Galois module is  automatically semisimple.

Next, we introduce a piece of terminology that makes some of the statements we prove less lengthy.

\begin{defi} \la{even_Tate_defn} 
We say that $\m{F} \in \db{X}$ is {\bf \good} if it is semisimple, Frobenius semisimple, very pure with weight zero,
even and Tate.
We say that   a graded Galois module is {\bf \good} if it is  Frobenius semisimple, very pure
with weight zero,
even and Tate. 
\end{defi}


\subsection{Convolution morphisms between twisted product varieties}\la{cmbtpv} $\;$

What follows is a brief summary of the notions surrounding twisted product varieties and convolution maps
that are more thoroughly discussed in $\S\ref{affgrps}, \ref{twcon_sec}$ and that are needed to state some of our main results in 
\S\ref{tempmr}.

Let $G$ be a split connected  reductive group over the finite field $k$. 
Let $\G\supsetneq \m{Q} \supset \B \subset \PP$ be the associated  loop group   together with a nested sequence
of parahoric subgroups, with $\B$ being the Iwahori associated with a  $k$-rational Borel  on $G$.
Let $\m{W}$ be the extended affine Weyl group associated with $\m{G}$ and  let $\m{W}_\PP \subseteq \m{W}$
be the finite subgroup associated with $\PP$ (see $\S\ref{affgrps}$).

 The twisted product varieties $X_\PP(w_\bullet)= X_\PP (w_1, \ldots w_r)$ (see Definition \ref{zdeftp}),
with $w_i \in  \m{W}_\PP \backslash \m{W} / \m{W}_\PP$,  are closed subvarieties  in the product  $(\G/\PP)^r$. 
We denote by $w_i''$ the image of $w_i$ under the natural surjection 
 $\m{W}_\PP \backslash \m{W} / \m{W}_\PP \to \m{W}_\m{Q} \backslash \m{W} / \m{W}_\m{Q}$;
 see \S\ref{orbtz}, especially (\ref{e3.5}) and (\ref{e3.6}).
 Given $1\leq r' \leq r$ and $1 \leq i_1 < \ldots  <i_{m}=r'$,  by consideration of
the natural product of projection maps $(\G/\PP)^r \to (\G/\m{Q})^m$ onto the $i_k$-th components, 
in \S\ref{098}, we introduce the generalized convolution maps
$p:X_\PP (w_\bullet) \to X_\m{Q} (w''_{I,\bullet})$ between twisted product varieties;
they generalize the standard convolution map $X_\B (w_1,  \ldots , w_r) \to X_\B (w_1* \cdots * w_r)$
(\ref{e56})
which is the special case when $\B = \PP = \m{Q}$, $r=r'$, $m=1$ and $i_1=r.$ 

Here, $*$ is the  Demazure product
on $\W$ (see \S\ref{Dem_comp_subsec}).
In this paper, we use  an equivalent version of this product operation on $\m{W}_\PP \backslash \m{W} / \m{W}_\PP,$ which  we call {\em geometric
Demazure product}, and we denote by $\star_\PP$ (see \ref{gpdem}).

We work with  the  convolution maps for which the $w_i$, which in general correspond to $\PP$-orbit closures
in $\G/\PP$, correspond to $\m{Q}$-orbit closures in $\G/\PP.$ Such $w$'s are said to be of $\m{Q}$-type
(see Definition \ref{qpma}).
These include the $w$'s that correspond to those $\m{Q}$-orbit closures $X_\PP (w)$  that 
are the full-pre-image of their image $X_\m{Q} (w'') \subseteq \G/\m{Q}$, which we name of $\m{Q}$-maximal type.
Note that both conditions are automatic when $\PP = \m{Q}$, so that the case of classical convolution maps 
is covered.

Our results hold also in the  ``finite" (vs.\,affine) context of  partial flag varieties $G/P$, with the same, or simpler, proofs. The choice of the notation 
$\G, \B$, etc., reflects our unified   treatment of the finite and of   affine cases; see \S\ref{notsy}.

\section{The main results}\la{tempmr}$\,$

\subsection{Proper maps over finite fields}\la{pmoff}$\,$

The {\em decomposition theorem} in \ci{bbd} states that if $f:X \to Y$ is a proper $k$-morphism and $\m{F}$
is a simple perverse sheaf on $X$, then 
$\ov{f_*\m{F}}$  is semisimple. See (cf.\,\S\ref{intro}). It is not known whether $f_* \m{F}$ is semisimple.
The issue is  whether the indecomposable lisse local systems appearing in the a-priori weaker  decomposition
over the finite field  are, in fact, already simple (absence of Frobenius Jordan blocks on the stalks); see \S\ref{dtoff}.
Moreover, it is   not known whether Frobenius acts semisimply on the stalks of a simple perverse sheaf,
not even in the case of the intersection complex of  the affine cone over a smooth projective 
variety. In fact, that would imply that Frobenius acts semisimply on the cohomology of smooth projective varieties.

Let us emphasize that  in Theorems \ref{tma} and \ref{tmb}, we do not need to pass to the algebraic closure, i.e. the indicated splittings
{\em already}  hold
over the finite field. Moreover, in Theorem \ref{tmb}, we do not assume, as one usually finds in the literature, that the proper map $f$ is birational, nor generically finite.

\begin{tm}\la{tma} {\bf (Semisimiplicity criterion for direct images)} 
Let $f: X \to Y$ be a proper map of varieties over the finite field $k$  and let $\m{F} \in \db{X}$
be semisimple. The direct image complex $f_* \m{F} \in \db{Y}$  is semisimple
and  Frobenius
semisimple if and only if the graded Galois modules
$H^*(\ov{f^{-1} (y)}, \ov{\m F})$
 are Frobenius semisimple for every closed point $y$ in $Y$.
\end{tm}

A rather different statement, which gives a sufficient condition to guarantee semisimplicity and Frobenius semisimplicity, can be found in \cite[Prop.\,9.15]{AR}.

\begin{tm}\la{tmb} {\bf (The intersection complex splits off)}
Let $f: X \to Y$ be a proper map of varieties over the finite field $k$.
The intersection complex $\ic{f(X)}$ is a direct summand of 
$f_* \ic{X}$ in $D^b_m(Y, \oql)$. In particular, the graded Galois module
$I\!H^*(\ov{f(X)}, \oql)$ is a direct summand of the graded Galois module
$I\!H^*(\ov{X}, \oql)$.
\end{tm}

Recall our Definition \ref{even_Tate_defn} of a good complex.

\begin{cor}\la{corbb} {\bf (Goodness for intersection cohomology)}
Let  $f: X \to Y$ be a proper  and surjective map of varieties over the finite field $k$.
 If the Galois module $I\!H^*(\ov{X}, \oql )$ is  pure of weight  zero \textup{(}resp.  with weights $\leq w$, resp. with weights
 $\geq w$, resp. Frobenius semisimple, resp. even, resp. Tate, 
 resp. \good\textup{)} then so is $I\!H^*(\ov{Y}, \oql)$.
\end{cor}

\subsection{Generalized convolution morphisms}\la{r42}$\;$

\begin{tm}\la{ctm} {\bf (Goodness for twisted product varieties)}
Let $X=X_\PP(w_\bullet)$ be a twisted product variety.
Then $I\!H^*(\ov{X}, \oql)$ and $\m{IC}_X$ are \good.  In particular, $\m{IC}_X$ is Frobenius semisimple and   Frobenius acts semisimply on $I\!H^*(\ov{X}, \oql)$.
\end{tm}

\begin{tm}\la{dtm} {\bf (Goodness for generalized convolution morphisms)}
Consider a generalized convolution morphism $p: Z:=X_\PP(w_\bullet) \to X:=X_\Q (w''_{I,\bullet})$ with $w_i \in \, _\PP W_\PP$ 
of $\m{Q}$-type. Then  $p_* \m{IC}_Z$ is \good.

\n
Moreover, for every closed point $x\in X$ and every open $U \supseteq p^{-1}(x)$, the natural
restriction map $I\!H^*(\ov{U}, \oql) \to H^*(\ov{p}^{-1} (\ov{x}), \m{IC}_{\ov{Z}})$ is surjective,
and the target is \good.

\n
In particular, the fibers of $p$ are geometrically connected.
\end{tm}

 We remind the reader of Remark
\ref{conv_rmk}, which clarifies the relation between a convolution product of certain perverse sheaves and a
convolution morphism, i.e., the former is a complex that coincides with a direct image by the latter: see (\ref{luc}). The following corollary  is a special case of Theorem \ref{dtm} and a direct consequence of it.

\begin{cor}\la{conv} {\bf (Goodness for  convolution products)}

\n
The convolution product   $\m{IC}_{X_\PP (w_1)} * \cdots * \m{IC}_{X_\PP (w_r)}= p_* \m{IC}_{\xbp{\PP}{w_\bullet}}$ is  \good.
\end{cor}

\begin{rmk}\la{anche}
We also give a different proof of Corollary \ref{conv} using the paving Theorem \ref{pavingtm}.\textup{(}2\textup{)}, in \S\ref{anches}.
\end{rmk}

\begin{rmk}\la{inpiu}   In the  case of Schubert varieties in  the finite \textup{(}i.e.\,``ordinary''\textup{)} flag variety $G/B$,
the fact that $\m{IC}_{X_B (w)}$ is \good \,    has been proved in
\ci[Corollary 4.4.3]{BGS}. The semisimplicity and Frobenius semisimplicity aspect of Corollary \ref{conv} has been addressed in \cite{AR} and \cite{BY},  and  in the cases of full \textup{(}affine\textup{)} flag varieties, one can also deduce semisimplicity and Frobenius semisimplicity using their methods.  In \S\ref{comp_lit_sec}, we shall make a few more remarks about the relation of our results with theirs. 
\end{rmk}


\begin{rmk}\la{rfz}
Recall that a \good \, graded Galois module is, in particular, Frobenius semi\-simple.
Theorem \ref{dtm} gives a proof, in  our set-up, of the general  Conjecture \ref{conjic}. 
For a proof  of this conjecture in the context of proper toric maps, see \ci{dC}.
\end{rmk}

Given Theorem \ref{dtm}, the proof of the following theorem concerning generalized convolution morphisms proceeds almost exactly as in the case \ci[Theorem 1.4.1]{dC} of proper
toric maps over finite fields and, as such, it is omitted; it is not used in the remainder of the paper.
The only issue that is not identical with respect to the proof in \ci{dC}, is the one of the geometric integrality of the varieties
$\m{O}$ below; in the case where $r'=1$ (see \S\ref{cmbtpv}), these varieties are $\Q$-orbit closures, hence they are geometrically integral
(e.g.\,by Proposition \ref{fqco});
 the  case where $r'>1$ follows from the $r'=1$ case, coupled with the local product structure Proposition \ref{gt50}. 

Following the statement is a short discussion relating the theorem to the positivity of  certain polynomials.

\begin{tm}\la{beqdt}{\rm ({\bf $\Q$-equivariant decomposition theorem for  generalized convolution
morphisms})}
Let  $p: Z:=X_\PP(w_\bullet) \to X:=X_\Q (w''_{I,\bullet})$ be a generalized convolution morphism with $w_i \in \, _\PP W_\PP$ 
of $\m{Q}$-type.
 There is an isomorphism in $\db{X}$ of \good \, complexes
\begin{align*}
p_* \ic{Z}  \cong & \;\bigoplus_{\mathcal O}  \ic{\mathcal O} \otimes \bb{M}_{p;\mathcal O},
\\
 {\bb{M}}_{p;\mathcal O}  =&\; \bigoplus_{j=0}^{\dim{Z} - \dim{\mathcal O}} {\oql}^{m_{p;\mathcal O,2j}} (-j)[-2j],
 \end{align*}
 where $\mathcal O$ is a finite collection of  geometrically integral  $\Q$-invariant closed subvarieties  in $X$,  the multiplicities 
 $m_{p; \mathcal O,2j}$ are subject to the following constraints:
 \ben
 \item
 Poincar\'e-Verdier duality:
 $m_{p;\mathcal O,2j} = m_{p;\mathcal O, 2\dim{Z} - 2\dim{\mathcal O} -2j};$
 \item
 relative hard Lefschetz:
 $m_{p;\mathcal O,2j} \geq m_{p;\mathcal O,2j-2},$ for every $2j \leq \dim{Z} - \dim{\mathcal O}.$
 \een
\end{tm}

In what follows, we are going to use freely the notion of incidence algebra of the poset associated with the 
$\B$-orbits in $\G/\B$ as summarized in \ci{dMM}, \S6. In particular, (\ref{anz}) is the analogue of \ci[Theorem 7.3]{dMM}
 in the context of the map $p$ below. The reader is warned that in \ci{dMM}, the poset of orbits in the toric variety has the order
 opposite to the one employed below in $\G/\B$, i.e.  here we have $v \leq w$ iff $X_\B (v) \subseteq X_\B (w)$. One can also use Hecke algebras in place of incidence algebras.  

Let $p: X_\B (w_\bullet) \to \xbp{\B}{w}$ be a  convolution morphism ($w$ is the Demazure product of the $w_i$'s).
For every $u\leq v  \leq w \in \W$, 
let: $F_{p;v}(q) \in \zed_{\geq 0}[q]$ be the Poincar\'e polynomial of the geometric  fiber $\ov{p^{-1} ({v\B})}$;
$P_{uv}(q) \in \zed_{\geq 0}[q]$ be the Kazhdan-Lusztig polynomial (which we view, thanks to the Kazhdan-Lusztig theorem
\ci{KL}, as the graded dimension of graded stalk 
of intersection complex of $\xbp{\B}{v}$ at the geometric point  at $u\B$);   $\w{P}_{uv}(q) \in \zed [q]$ be the 
function  inverse  to the 
Kazhdan-Lusztig polynomials in the incidence $\zed [q]$-algebra associated
with the poset of $\B$-orbits in $\G/\B$, i.e. we have  $\sum_{u \leq x\leq v}P_{ux}  \w{P}_{xv}= \delta_{uv}$;
 $M_{p;v}(q) \in \zed_{\geq 0}[q]$ be the graded dimension of
$\mathbb M_{p;v}$. 

Of course, $v\leq w$ appears non trivially in the decomposition Theorem  for $p$ iff $\mathbb M_{p;v}\neq 0;$
in this case, we say that {\em $v$ is a support of the map $p.$} 
By virtue of the precise form of the decomposition  Theorem \ref{beqdt}  for $p$, and by  using incidence algebras (or Hecke algebras)
exactly as in \ci[Theorem 7.3]{dMM},
one gets the following identities
\beq\la{anz}
F_{p;v} = \sum_{v \leq x \leq w} P_{vx} \, M_{p;x} \;\; \mbox{(in $\zed_{\geq 0} [q]$)},
\qquad M_{p;x} = \sum_{x \leq z \leq w} \w{P}_{xz} \, F_{p;z}
 \;\; \mbox{(in $\zed [q]$)},
\eeq
where the first one stems from Theorem \ref{beqdt}, and the  second one is obtained by inverting the first one by means of the identity $\sum_{u \leq x\leq v}P_{ux}  \w{P}_{xv}= \delta_{uv}$.

The polynomials $\w{P}$ satisfy the identity $\w{P}_{uv} =(-1)^{\ell(u) + \ell(v)} Q_{uv}$, where the
polynomials $Q_{uv} \in \zed_{\geq 0} [q]$ are the inverse Kazhdan-Lusztig polynomials; see \cite[Prop.\,5.7]{KL} and \cite[Thm.\,3.7]{GH}. In particular, 
it is not a priori clear that the r.h.s.   of the second identity in (\ref{anz})  should be a polynomial with non-negative coefficients:
here, we see this as a consequence of the decomposition theorem.

Further, let us specialize to $p: X_\B (s_\bullet) \to X_\B (s)$ being a Demazure map, where $s_\bullet \in \m{S}^r$
and $s$ is the Demazure product $s_1 * \cdots *s_r.$
The polynomial $F_{p;v}[q]$ counts  the number of affine cells in each dimension
in the fiber of $p$ over the point $v \B$ (cfr. Theorem \ref{pavingtm}.(1)). The total number of these,
i.e. the Euler number of the fiber, is of course $F_{p;v}(1)$ and it is also the cardinality of the  set
$
\{(t_1, \ldots, t_{r}) \, | \; t_i= s_i \, \mbox{\rm or}\, 1;\,\Pi_i t_i = v \}.
$

The identity
$M=\sum \w{P} F$ in (\ref{anz}) expresses  
a non-trivial  relation between the  supports of the Demazure map $p$, the topology of its fibers,
and  the inverse Kazhdan-Lusztig polynomials.

\begin{??}\la{domanda!} {\rm ({\bf  Supports for Demazure maps})}
Which Schubert subvarieties $\xbp{\B}{v}$ of a Schubert variety $\xbp{\B}{w}$ in $\G/\B$ appear as supports of a given  
Demazure map?
This boils down to determining when the non-negative  $M_{p;v}(1)$ is in fact positive.  This seems to be 
a difficult problem  --even in the finite case!--, in part because of the presence of the inverse Kazhdan-Lusztig polynomials.
\end{??}

\subsection{The negative parahoric loop group and big cells} \la{big_cell_intro_sec}$\;$

In $\S\ref{affgrps}$ we introduce a ``negative'' loop group $L^{--}P_{\bf f}$ associated to a parahoric loop group $L^+P_{\bf f}$, for any facet ${\bf f}$ of the Bruhat-Tits building for $G(k(\!(t)\!))$. More precisely, assuming ${\bf f}$ is in the closure of an alcove ${\bf a}$ corresponding to a Borel subgroup $B = TU$ of $G$, we have the standard definition $L^{--}P_{\bf a} := L^{--}G \cdot \bar{U}$, and then we define 
$$
L^{--}P_{\bf f} \, := \, \bigcap_{w \in \w{W}_{\bf f}} \, ^wL^{--}P_{\bf a}.
$$
Our results on the geometry of twisted products of Schubert varieties rest on the following theorem:

\begin{tm} \la{big_cell_thm}
The multiplication map gives an open immersion
$$
L^{--}P_{\bf f} \times L^+P_{\bf f} \longrightarrow LG.
$$
\end{tm}
This allows us to define the Zariski-open {\em big cell} $\mathcal C_{\bf f} := L^{--}P_{\bf f} \cdot x_e$ in the partial affine flag variety $\mathcal F_{P_{\bf f}}$. The proof rests on new Iwahori-type decompositions, most importantly Proposition \ref{proto_decomp2}.

\smallskip

\subsection{A surjectivity criterion}\la{asurjctz}$\;$

The part of Theorem \ref{dtm} concerning fibers  is a consequence   of the following more technical statement.
Recall the Virasoro action via $L_0$ of $\mathbb G_m$ on $\G$ (\ci[p.\,48]{F}) (defined and named ``dilation" action in  \S\ref{pftmff} and denoted by $c$).

\begin{tm}\la{tmff} {\bf (Surjectivity for fibers criterion)} 
Let $X:=\xbp{\B\PP}{w} \subseteq \mathcal G / \mathcal P$ be the closure of a $\mathcal B$-orbit,
 let $g: Z \to X$ be  a proper and $\B$-equivariant map of $\B$-varieties. 
In the partial affine flag  case, we further assume that 
there is a $\mathbb{G}_m$-action $c_Z$ on $Z$ such that $c_Z$ commutes with the $T(k)$-action, and such  that $g$ is equivariant with respect to the action $c_Z$ on $Z$ and  with the action  $c$ on $X$. 

\n
For every closed point $x \in X$, for every Zariski open subset $U \subseteq Z$ such that
$g^{-1}(x) \subseteq U$,  
 the natural restriction map of graded Galois modules
\beq\la{vobn}
I\!H^*(\ov{U}) \to H^*(\ov{g}^{-1} (\ov{x}), \ic{\ov{Z}})
\eeq
 is surjective and  the target is pure of weight zero.

\n
If, in addition, $I\!H^*(\ov{Z})$ is Frobenius semisimple \textup{(}resp.\,\,even, Tate, \good\textup{)}, then so is the target.
\end{tm}

\begin{rmk}\la{fg56}
The hypothesis on the existence of the action $c_Z$ in Theorem \ref{tmff} is not restrictive
in the context of this paper,  as it is automatically satisfied in the situations we meet when 
proving  Theorem \ref{dtm}. Let us emphasize that it is also automatically satisfied in the finite case. Similarly, the  hypotheses
 at the end of the statement of 
  Theorem \ref{tmff}  are also not  too restrictive, since, as it turns out, they are  automatically satisfied in the context of Theorem \ref{dtm},
  by virtue of Theorem \ref{ctm}.
\end{rmk}

\subsection{Affine paving of fibers of Demazure-type maps}\la{affdm}$\;$

\begin{defi}\la{affpav}
{\rm 
({\bf Affine paving}) A $k$-variety $X$ is {\em paved by affine spaces} if there exists a sequence of closed subschemes $\emptyset = X_0 \subset X_1 \subset \cdots \subset X_n =: X_{\rm red}$ such that, for every $1 \leq i \leq n$, the difference $X_i - X_{i-1}$ is the (topologically) disjoint union of finitely many affine spaces ${\mathbb A}^{n_{i_j}}$. 
}
\end{defi}

We refer to $\S\ref{cmbtpv}, \ref{notsy}, \ref{orbtz}$ for the notation used in the following result.

\begin{tm}\la{pavingtm} Given sequences  \textup{}$s_\bullet \in \m{S}^r$, $w_\bullet \in (\m{W}_\PP \backslash \m{W} /
\m{W}_{\PP})^r$\textup{},
the following can be paved by affine spaces
\ben
\item
The fibers of the convolution map $p : X_\B (s_\bullet) \to X_\B (s_\star)$, and $p^{-1}(Y_\B (v))$, $\forall v  \leq s_\star.$
\item
The fibers of  the convolution map obtained as the composition  $X_\B (s_{\bullet \bullet}) \to X_\B (s_\star)\to X_\PP (u_\star)$,
where, for every $1\leq i \leq r$,  $s_{i\bullet}$ is a reduced word for an element $u_i \in \m{W}$ which is 
$\PP$-maximal, and $s_\star$ is the Demazure product of the $s_{\bullet \bullet}$ which, by associativity,
coincides with the one, $u_\star$, for the $u_i$. 

\item The twisted product varieties $X_\PP(w_\bullet)$.
\een
\end{tm}

\begin{rmk}\la{re45}
In  Theorem \ref{pavingtm}.\textup{(}1\textup{)} it is important that we {\em do not require} $s_1\cdots s_r$ to be a reduced expression.
As we shall show, associated with a convolution map with
$\PP=\m{Q}$, there is a commutative diagram \textup{(}\ref{mmh}\textup{)} of maps
of twisted product varieties. Theorem \ref{pavingtm}.\textup{(}2\textup{)}  is a paving result for the fibers of the morphisms $q'p'\pi$ appearing in that diagram.
\end{rmk}

\begin{rmk} \label{gen_paving_rem}
We do not know whether every fiber of a convolution morphism $p : X_\PP(w_\bullet) \rightarrow X_\PP(w_\star)$ is paved by affine spaces, even in the context of affine Grassmannians \textup{(}cf.\,\cite[Cor.\,1.2 and Question 3.9]{H06}\textup{)}. However, it is not difficult to show that, in general, each fiber can be paved by varieties which are iterated bundles $B_{N} \rightarrow B_{N-1} \rightarrow \cdots B_1 \rightarrow B_0 = \mathbb A^0 $ where each $B_{i+1} \rightarrow B_i$ is a locally trivial ${\mathbb A}^1$ or ${\mathbb A}^1 - {\mathbb A}^0$ fibration. We shall not use this result.
\end{rmk}


\section{Loop groups and partial affine flag varieties}\la{affgrps}$\,$

In $\S \ref{Borel_pair_sec}- \ref{rt_grp_sec}$ we review some standard background material; our main references are \cite{BLa, F, PR, HRa}. In $\S\ref{neg_loop_grp_sec}-\ref{orbtz}$ we develop new material, including our definition of the negative parahoric loop group, various Iwahori-type decompositions, and our theory of the big cell.

\subsection{Reductive groups and Borel pairs} \la{Borel_pair_sec}

Throughout the paper $G$ will denote a split connected reductive group over a finite field $k$, and $\bar{k}$ will denote an algebraic closure of $k$. Fix once and for all a $k$-split maximal torus $T$ and a $k$-rational Borel subgroup $B \supset T$.  Let $U$ be the unipotent radical of $B$, so that $B = TU$. Let $\bar{B} = T\bar{U}$ be the opposite Borel subgroup; among the Borel subgroups containing $T$, it is characterized by the equality $T = B \cap \bar{B}$. Let $\Phi(G, T) \subset X^*(T)$ denote the set of roots associated to $T \subset G$; write $\alpha^\vee \in X_*(T)$ for the coroot corresponding to $\alpha \in \Phi(G,T)$. Write $U_\alpha \subset G$ for the root subgroup corresponding to $\alpha \in \Phi(G, T)$; we say $\alpha$ is {\em positive} (and  write $\alpha > 0$) if $U_\alpha \subset U$; we have $B = T \displaystyle\prod_{\alpha > 0} U_\alpha$.

The following remark will be used a few times in the paper.

\begin{rmk} \label{Borel_pair_rem} We work over any perfect field $k$. Fix a faithful finite-dimensional representation of $G$, i.e.,\,a closed immersion of group $k$-schemes $G \hookrightarrow {\rm GL}_N$. Let $B_N$ be a $k$-rational Borel subgroup of ${\rm  GL}_N$ containing $B$ \textup{(}cf.\,\cite[15.2.5]{Spr}\textup{)}, and choose inside $B_N$ a $k$-split maximal torus $T_N$ containing $T$. Let $\bar{B}_N$ be the $k$-rational Borel subgroup  opposite to $B_N$ with respect to $T_N$.  Let $U_N \subset B_N$ \textup{(}resp.,\, $\bar{U}_N \subset \bar{B}_N$\textup{)} be the unipotent radical. Since $T$ is its own centralizer in $G$, we have $T = G \cap T_N$. Also, we clearly have $B = G \cap B_N$ \textup{(}cf.\,\cite[$\S23.1$, Cor.\,A]{Hum1}\textup{)}, and therefore $U = G \cap U_N$.  Note that $(G \cap \bar{B}_N)^\circ$ is a Borel subgroup $B'$ containing $T$, since it is connected and solvable and $G/(G \cap \bar{B}_N)^\circ$ is complete. As $T \subseteq B \cap B' \subseteq G \cap B_N \cap \bar{B}_N = G \cap T_N = T$, we have $T = B \cap B'$ and so $B' = \bar{B}$. Hence we have $\bar{B} = G \cap \bar{B}_N$ and $\bar{U} = G \cap \bar{U}_N$ as well. The content of this set-up is that the standard Iwahori loop-group $L^+P_{\bf a}$ \textup{(}and its ``negative'' analogues $L^{--}P_{\bf a}(m)$, etc.\,, defined  later in $\S\ref{affgrps}$\textup{)} in a loop group $LG$ can be realized by intersecting $LG$ with the corresponding object in $L{\rm GL}_N$. 
\end{rmk}

\subsection{Affine roots, affine Weyl groups, and parahoric group schemes}\la{affterm}$\,$

A convenient reference for the material recalled here is \cite{HRa}.

Let $N_G(T)$ be the normalizer of $T$ in $G$. Let $W = N_G(T)/T$ be the finite Weyl group, and let $\w{W}:= X_*(T) \rtimes W $ be the {\em extended affine Weyl group}. The groups $W$ and $\w{W}$ act by affine-linear automorphisms on the Euclidean space $\mathbb X_* = X_*(T)\otimes \mathbb R$; in the case of $\w{W}$ this is defined by setting,  for every  $ \lambda \in  X_*(T)$, $\ov{w} \in W$  and $x \in \mathbb X_*$
\beq\la{mnbv}
t_\lambda \bar{w} (x)  := (\lambda, \ov{w}) (x): = \lambda + \bar{w}(x).
\eeq

The set $\Phi_{\rm aff}$ of {\em affine roots} consists of the affine-linear functionals on $\mathbb X_*$ of the form $\alpha + n$, where $\alpha
\in \Phi(G,T)$ and $n \in \mathbb Z$. The {\em affine root hyperplanes}
are the zero   loci  $H_{\alpha + n} \subseteq \mathbb X_*$
of  the affine roots. The {\em alcoves} are the connected components of $\mathbb X_* 
- \bigcup_{\alpha + n} H_{\alpha +k}$. The hyperplanes give the structure of a polysimplicial complex to the Euclidean space $\mathbb X_*$, and the {\em facets} are the simplices (thus alcoves are facets).

The action of $\widetilde{W} = X_*(T) \rtimes W$ on $\mathbb X_*$ induces an action on the set of affine roots, by $w(\alpha + n)(\cdot) := (\alpha + n)(w^{-1}(\cdot))$. Let $\mathbb X^+_* \subset \mathbb X_*$ be the {\em dominant chamber} consisting of the $x \in \mathbb X_*$ with $\alpha(x) > 0$ for all $\alpha >0$. Let ${\bf a}$ be the unique alcove in $\mathbb X^+_*$ whose closure contains the origin of $\mathbb X_*$. 

We say an affine root $\alpha + n$ is {\em positive} if either $n \geq 1$, or $n = 0$ and $\alpha > 0$.  Equivalently, $\alpha + n$ takes positive values on ${\bf a}$. We write $\alpha + n > 0$ in this case. We write $\alpha + n < 0$ if $-\alpha - n > 0$.  Let $S_{\rm aff}$ be the set of {\em simple affine roots}, namely those positive affine roots of the form $\alpha_i$ (where $\alpha_i$ is a simple positive root in $\Phi(G,T)$), or $-\tilde{\alpha} + 1$ (where $\tilde{\alpha} \in \Phi(G, T)$ is a highest positive root).

Let $s_{\alpha+n}$ be the {\em affine reflection} on $\mathbb X_*$ corresponding to 
the affine root $\alpha + n$; this is the map sending $x \in \mathbb X_*$ to $x - (\alpha + n)(x)\alpha^\vee = x - \alpha(x) \alpha^\vee - n\alpha^\vee$.  We have   $s_{\alpha + n} = t_{-n\alpha^\vee} s_\alpha \in \widetilde{W}$.  We can think of $S_{\rm aff}$ as the set of reflections on $\mathbb X_*$ through the walls of ${\bf a}$. Let $Q^\vee : = \langle t_{\beta^\vee}  \in \w{W} ~ | ~ \beta \in \Phi(G,T) \rangle$, and let $W_{\rm aff} : = Q^\vee \rtimes W$.  Then $(W_{\rm aff}, S_{\rm aff})$ is a Coxeter system, and hence there is a length function $\ell: W_{\rm aff} \rightarrow {\mathbb Z}_{\geq 0}$ and a Bruhat order $\leq$ on $W_{\rm aff}$.

The action of $\w{W}$ on $\mathbb X_*$ permutes the affine hyperplanes, and hence the alcoves in $\mathbb X_*$; let $\Omega_{\bf a} \subset \w{W}$ denote the stabilizer of ${\bf a}$.  Then we have a semi-direct product
\begin{equation} \la{q.cox}
\w{W} = W_{\rm aff} \rtimes \Omega_{\bf a}.
\end{equation}
This gives $\w{W}$ the structure of a {\em quasi-Coxeter group}: a semi-direct product of a Coxeter group with an abelian group.  We can extend $\ell$ and $\leq$ to $\w{W}$: for $w_1, w_2 \in W_{\rm aff}$ and $\tau_1, \tau_2 \in \Omega_{\bf a}$, we set  $\ell(w_1 \tau_1) = \ell(w_1)$, and $w_1 \tau_1 \leq w_2 \tau_2$ iff $\tau_1 = \tau_2$ and $w_1 \leq w_2$ in $W_{\rm aff}$.

Let $t$ be an indeterminate. We now fix, once and for all, a set-theoretic embedding $\w{W} \hookrightarrow N_G(T)(k(\!(t)\!))$ as follows: send $\bar{w} \in W$ to any lift in $N_G(T)(k)$, chosen arbitrarily; send $\lambda \in X_*(T)$ to $\lambda(t^{-1}) \in T(k(\!(t)\!))$.\footnote{The reason for using $t^{-1}$ instead of $t$ here is to make (\ref{f'f_Bruhat_closure}) and (\ref{U_conj}) true.}  Henceforth, any $w \in \widetilde{W}$ will be viewed as an element of $G(k(\!(t)\!))$ using this convention.

There is an isomorphism
$$
N_G(T)(k(\!(t)\!))/T(k[\![t]\!]) \, \overset{\sim}{\rightarrow} \, \w{W},
$$
which sends $\lambda(t^{-1}) \bar{w}$ to $t_\lambda \bar{w}$. Via this isomorphism, $N_G(T(k(\!(t)\!))$ acts on $\mathbb X_*$. The Bruhat-Tits building $\mathfrak B(G, k(\!(t)\!))$ is a polysimplicial complex containing $\mathbb X_*$ as an ``apartment''. It is possible to extend the action of $N_G(T(k(\!(t)\!))$ on $\mathbb X_*$ to an action of $G(k(\!(t)\!))$ on $\mathfrak B(G, k(\!(t)\!))$. 

Let ${\bf f}$ be a facet in $\mathbb X_*$. In \cite[5.2.6]{BTII}, Bruhat and Tits construct the {\em parahoric group scheme} $P_{\bf f}$ over ${\rm Spec}(k[\![t]\!])$.  In our present setting, $P_{\bf f}$ can be characterized as the unique (up to isomorphism) smooth affine group scheme over ${\rm Spec}(k[\![t]\!])$ with connected geometric fibers, with generic fiber isomorphic to $G_{k(\!(t)\!)}$, and with $P_{\bf f}(k[\![t]\!])$ identified via that isomorphism with the subgroup of $G(k(\!(t)\!))$ which fixes ${\bf f}$ pointwise. 

We call $P_{\bf a}$ ``the'' {\em Iwahori group scheme}.  Its $k[\![t]\!]$-points can also be characterized as the preimage of $B(k)$ under the reduction map $G(k[\![t]\!]) \rightarrow G(k), \,\,\, t \mapsto 0$. If ${\bf 0} \subset \mathbb X_*$ is the facet containing the origin, then $P_{\bf 0}$ is a  (hyper)special maximal parahoric group scheme and its $k[\![t]\!]$-points can be identified with $G(k[\![t]\!])$. For any $k$-algebra $R$, we can similarly characterize $P_{\bf a}(R[\![t]\!])$ and $P_{\bf 0}(R[\![t]\!])$ (using the Iwahori decomposition in the case of $P_{\bf a}$).

Assume from now on that ${\bf f}$ is contained in the closure of ${\bf a}$. Set $$\w{W}_{\bf f} := [N_G(T)(k(\!(t)\!)) \cap P_{\bf f}(k[\![t]\!])]/T(k[\![t]\!]).$$ 
Then $\w{W}_{\bf f}$ can be identified with the subgroup of $W_{\rm aff}$ which fixes ${\bf f}$ pointwise. Let $S_{\rm aff, \bf f} \subset S_{\rm aff}$ be the simple affine reflections through the walls containing ${\bf f}$.  Then it is known that $(\widetilde{W}_{\bf f}, S_{\rm aff, \bf f})$ is a sub-Coxeter system of $(W_{\rm aff}, S_{\rm aff})$. Note that $\w{W}_{\bf f}$ is a {\em finite} group.

It is well-known that for any two facets ${\bf f}_1, {\bf f}_2$ in $\mathbb X_*$, the embedding $\w{W} \hookrightarrow G(k(\!(t)\!))$ induces a bijection (the Bruhat-Tits decomposition)
\begin{equation} \la{proto_BT_decomp}
\w{W}_{{\bf f}_1} \backslash \w{W}/ \w{W}_{{\bf f}_2} \, \overset{\sim}{\rightarrow} P_{{\bf f}_1}(k[\![t]\!]) \backslash G(k(\!(t)\!)) / P_{{\bf f}_2}(k[\![t]\!]).
\end{equation}

\subsection{Loop groups, parahoric loop groups, and partial affine flag varieties}

A convenient reference for the material recalled in this subsection is \cite{PR}.

The loop group $LG$ of $G$ is the  ind-affine group ind-scheme over $k$ that represents the functor $R \mapsto G(R(\!(t)\!))$ on $k$-algebras $R$. The positive loop group $L^+G$ is the affine group scheme over $k$ that represents $R \mapsto G (R[\![t]\!])$. The negative loop group  $L^- G$ is the  ind-affine group ind-scheme 
over  $k$ that represents the  functor $R \mapsto G(R[t^{-1}])$.

We have the natural inclusion maps  $L^\pm G \to LG$  and the natural reduction maps $L^\pm G \to G$ (sending $t^{\pm 1} \mapsto 0$). The kernels of the reduction maps are denoted by $L^{++} G \subset L^+G$ resp.\,$L^{--}G \subset L^-G$.

We may also define $L^+P_{\bf f}$ to be the group scheme representing the functor
$$
R \mapsto P_{\bf f}(R[\![t]\!]).
$$
This makes sense as $R[\![t]\!]$ is a $k[\![t]\!]$-algebra. Also, $R(\!(t)\!) = R[\![t]\!][\frac{1}{t}]$, and we define $LP_{\bf f}$ as the group ind-scheme representing $R \mapsto P_{\bf f}(R(\!(t)\!))$. Since $P_{\bf f} \otimes_{k[\![t]\!]} k(\!(t)\!) \cong G_{k(\!(t)\!)}$, we have $LP_{\bf f} \cong LG$.

It is not hard to show that $L^+P_{\bf f}$ is formally smooth, pro-smooth, and integral as a $k$-scheme. We omit the proofs.

\begin{defi} We define the partial affine flag variety $\mathcal F_{P_{\bf f}}$ to be the {\em fpqc}-sheaf associated to the presheaf on the category of $k$-algebras $R$ 
$$
R \longmapsto LG(R)/L^+P_{\bf f}(R).
$$
\end{defi}
It is well-known that $\mathcal F_{P_{\bf f}}$ is represented by an ind-$k$-scheme which is ind-projective over $k$; see e.g.\,\cite[Thm.\,1.4]{PR}. We denote this ind-scheme also by $\mathcal F_{P_{\bf f}}$.  Note that $\mathcal F_{P_{\bf f}}$ is usually not reduced (see \cite[$\S6$]{PR}). Therefore, from $\S\ref{notsy}$ onward, we will always consider $\mathcal F_{P_{\bf f}}$ {\em with its reduced structure}.  The same goes for the Schubert varieties (and the twisted products of Schubert varieties) which are defined in what follows.

\subsection{Schubert varieties and closure relations} 

Fix facets ${\bf f}'$ and ${\bf f}$ in the closure of ${\bf a}$. Given $v \in \w{W}$ (viewed in $N_G(T)(k[t,t^{-1}])$ according to our convention in $\S\ref{affterm}$), we write $Y_{\bf f', \bf f}(v)$ for the reduced $L^+P_{{\bf f}'}$-orbit of 
$$
x_v := v L^+P_{\bf f}/L^+P_{\bf f}
$$
in the ind-scheme $\mathcal F_{P_{\bf f}}$. Then $Y_{\bf f', \bf f}(v)$ is an integral smooth $k$-variety.  Let $X_{\bf f', \bf f}(v)$ denote its (automatically reduced) Zariski closure in $\mathcal F_{P_{\bf f}}$. Then $X_{\bf f', \bf f}(v)$ is a possibly singular integral $k$-variety.

A fundamental fact is that the Bruhat order describes closure relations:
\begin{equation} \la{f'f_Bruhat_closure}
X_{\bf f', \bf f}(w) = \coprod_{v \leq w} Y_{\bf f', \bf f}(v)
\end{equation}
where $v, w \in \w{W}_{\bf f'} \backslash \w{W} / \w{W}_{\bf f}$ and $v \leq w$ in the Bruhat order on $\w{W}_{\bf f'} \backslash \w{W} / \w{W}_{\bf f}$ induced by the Bruhat order $\leq$ on $\w{W}$. The closure relations can be proved using Demazure resolutions and thus, ultimately, the BN-pair relations; 
 see e.g.,\,\cite[Prop.\,0.1]{Ric}.

In what follows, we will often write $Y_{\bf f}$ (resp.\,$X_{\bf f}$) for $Y_{\bf f, \bf f}$ (resp.\,$X_{\bf f, \bf f}$).

\subsection{Affine root groups} \la{rt_grp_sec}

Given $\alpha \in \Phi(G, T)$, let $u_\alpha: {\mathbb G}_a \rightarrow U_\alpha$ be the associated root homomorphism.  We can (and do) normalize the $u_\alpha$ such that for $w \in W$, $wu_{\alpha}(x)w^{-1} = u_{w\alpha}(\pm x)$. Given an affine root $\alpha + n$, we define the {\em affine root group} as the $k$-subgroup $U_{\alpha + n } \subset LU_\alpha$ which is the image of the homomorphism 
\begin{align*}
{\mathbb G}_a &\rightarrow LU_\alpha \\
x &\mapsto u_\alpha(x\,t^n).
\end{align*}

Representing $w \in \widetilde{W}$ by an element $w \in N_G(T))(k(\!(t)\!))$ according to our conventions, we have
\begin{equation} \label{U_conj}
w\, U_{\alpha + n} \, w^{-1} = U_{w(\alpha + n)}.
\end{equation}
Associated to a root $\alpha > 0$ we have a homomorphism $\varphi_\alpha: {\rm SL}_2 \rightarrow G$ such that $u_\alpha(y) = \varphi_\alpha\big(\begin{bmatrix} 1 & y \\ 0 & 1 \end{bmatrix}\big)$, $u_{-\alpha}(x) = \varphi_\alpha\big(\begin{bmatrix} 1 & 0 \\ x & 1 \end{bmatrix}\big)$, and $\varphi_\alpha$ sends the diagonal torus of ${\rm SL}_2$ into $T$.
We have the following commutation relations for the bracket $[g,h] := g h g^{-1}h^{-1}$:
\begin{align} 
[u_{\alpha}(t^m x) \, , \,  u_{\beta}(t^n y) ] &= \prod u_{i\alpha + j \beta} (c_{\alpha, \beta; i,j} (t^mx)^i (t^ny)^j)  \la{comm_rel1} \\
[ u_{-\alpha}(x) \, , \, u_\alpha(y)] &= \varphi_\alpha\big(\begin{bmatrix} 1-xy & xy^2 \\ -x^2y & 1+xy+x^2y^2 \end{bmatrix}\big) \la{comm_rel2}
\end{align}
where in the first relation $\alpha \neq \pm \beta$ and the product ranges over pairs of integers $i,j > 0$ such that $i \alpha + j \beta$ is a root, and the $c_{\alpha, \beta;i,j} \in k$ are the structure constants for the group $G$ over $k$ (see \cite[9.2.1]{Spr}). The second relation is for $\alpha > 0$ but an obvious analogue holds for $\alpha < 0$.

\subsection{The ``negative'' parahoric loop group} \la{neg_loop_grp_sec}

Fix a facet ${\bf f}$ in the closure of ${\bf a}$.  We write $\alpha + n \overset{\bf f}{<} 0$ if the affine root $\alpha + n$ takes negative values on ${\bf f}$. Note that $\alpha + n \overset{\bf f}{<} 0$ implies $\alpha + n < 0$.

Let $H$ be any affine $k$-group (not necessarily reductive). For $m \geq 1$, let $L^{(-m)}H$ be the group ind-scheme representing the functor
$$
R \mapsto {\rm ker}[H(R[t^{-1}]) \rightarrow H(R[t^{-1}]/t^{-m})].
$$
Note that $L^{(-1)}H = L^{--}H$.  Also, set $L^{(-0)}H := L^-H$.

Let us define 
$$L^{--}P_{\bf a} := L^{(-1)}G \cdot \bar{U}.$$ 
Note this is a group as it is contained in $L^-G$ and $L^{(-1)}G$ is normal in $L^-G$. We wish to define $L^{--}P_{\bf f}$ for any facet ${\bf f}$ in the closure of ${\bf a}$.  Its Lie algebra should be generated by Lie subalgebras of the form ${\rm Lie}(U_{\alpha + n})$ where $\alpha + n \overset{\bf f}{<} 0$.  Thus it should be contained in $L^{--}P_{\bf a}$.  In general it might not contain $L^{(-1)}G$, but it should always contain $L^{(-2)}G$.

Let $P = M U_P \supset B$ and $\bar{P} = M\bar{U}_P \subset \bar{B}$ be opposite parabolic subgroups of $G$ with the same Levi factor $M$; let $W_M \subset W$ be the finite Weyl group generated by the simple reflections for roots appearing in ${\rm Lie}(M)$.  Then it is easy to prove that
$$
\bar{U}_{P} = \bigcap_{w \in W_M} \, ^w\bar{U}.
$$
This fact is the inspiration for the following definition (we thank Xuhua He for suggesting this alternative to our original definition, which appears in Proposition \ref{alt_defn}).

\begin{defi} \la{neg_parahoric_defn}
We define $L^{--}P_{\bf f}$ to be the ind-affine group ind-scheme over $k$ defined by
$$
L^{--}P_{\bf f} = \bigcap_{w \in \w{W}_{\bf f}} \, ^wL^{--}P_{\bf a},
$$
where the intersection is taken in $LG$.
\end{defi}

Here we consider conjugation by $w \in \w{W}_{\bf f}$ viewed as an element of $G(k[t,t^{-1}])$ according to our convention in $\S \ref{affterm}$. This definition gives what it should in the ``obvious'' cases. For example, if ${\bf f} = {\bf a}$, then $\w{W}_{\bf a} = \{e \}$ and the r.h.s.\,is $L^{--}P_{\bf a}$.  If ${\bf f} = {\bf 0}$, the $\w{W}_{\bf 0} = W$ and the r.h.s.\,is $L^{--}G$.

We would like another, more concrete, understanding of $L^{--}P_{\bf f}$, given in Proposition \ref{alt_defn} below. Before that, we will need a series of definitions and lemmas. 

For $m \geq 0$, we define $L^{--}P_{\bf a}(m)$ to be the group ind-$k$-scheme representing the group functor sending $R$ to the preimage of $T(R[t^{-1}]/t^{-m})$ under the natural map
\begin{equation} \la{nat_map}
L^{--}P_{\bf a}(R) \hookrightarrow L^-G(R) \rightarrow G(R[t^{-1}]/t^{-m}).
\end{equation}
In particular $L^{--}P_{\bf a}(0) = L^{--}P_{\bf a}$.

Now we also define $L^{--}P_{\bf a}[m]$ to represent the functor sending $R$ to the preimage of $\bar{B}(R[t^{-1}]/t^{-m})$ under (\ref{nat_map}). Further, for $m \geq 1$ let 
$$
L^{--}P_{\bf a}\langle m \rangle  := L^{--}P_{\bf a}[m] \cap L^{--}P_{\bf a}(m-1).
$$
Note that $L^{--}P_{\bf a}\langle 1 \rangle = L^{--}P_{\bf a}[1] = L^{--}P_{\bf a}$. For $m \geq 1$ we have
\begin{align}
L^{--}P_{\bf a}(m+1) \, &\triangleleft \, L^{--}P_{\bf a}(m) \la{normal_1} \\
L^{--}P_{\bf a}\langle m+1 \rangle \, &\triangleleft \, L^{--}P_{\bf a}\langle m \rangle. \la{normal_2}
\end{align}
By Remark \ref{Borel_pair_rem}, (\ref{normal_1}) and (\ref{normal_2}) reduce to the case $G = {\rm GL}_N$, where they can be checked by matrix calculations. Therefore, we have for $m \geq 1$ a very useful chain of subgroups
\begin{equation} \la{norm_sgs}
L^{--}P_{\bf a}\langle m+1 \rangle \, \triangleleft \, L^{--}P_{\bf a}(m) \, \triangleleft \, L^{--}P_{\bf a}\langle m \rangle.
\end{equation}
Their usefulness hinges on the normalities in (\ref{norm_sgs}) and on the fact that the quotients in (\ref{norm_sgs}) are isomorphic as $k$-functors to
\begin{equation} \la{quots}
\frac{L^{(-m)}U}{L^{(-m-1)}U} \hspace{.2in}, \hspace{.2in} \frac{L^{(-m+1)}\bar{U}}{L^{(-m)}\bar{U}},
\end{equation}
respectively. Let us prove this assertion. Using the $k$-variety isomorphisms $U \cong \prod_{\alpha > 0} U_\alpha$ (resp.\,$\bar{U} \cong \prod_{\alpha < 0} U_{\alpha}$) -- with indices taken in any order -- it is easy to show
\begin{equation} \la{prod_eq}
L^{(-m-1)}U \backslash L^{(-m)}U \, \cong \, \prod_{\alpha > 0} U_{\alpha - m} \hspace{.2in} , \hspace{.2in} L^{(-m)}\bar{U} \backslash L^{(-m+1)}\bar{U} \, \cong \, \prod_{\alpha < 0} U_{\alpha - m+1}. \hspace{.2in} 
\end{equation}
Then it is straightforward to identify the two subquotients of (\ref{norm_sgs}) with the terms in (\ref{quots}). For example, we show that the map $L^{(-m)}U \rightarrow L^{--}P_{\bf a}\langle m+1 \rangle \backslash L^{--}P_{\bf a}(m)$ is surjective by right-multiplying an element in the target by a suitable sequence of elements in the groups $U_{\alpha - m}$, $\alpha > 0$, until it becomes the trivial coset.

\begin{rmk} \la{abelian}
In fact \textup{(}\ref{prod_eq}\textup{)} gives isomorphisms {\em of group functors}, which are all {\em abelian}, except for $L^{(-1)}\bar{U} \backslash L^-\bar{U} \cong \prod_{\alpha < 0} U_\alpha$. Similarly, the groups $L^{--}P_{\bf a}\langle m\!+\!2\rangle \backslash L^{--}P_{\bf a}\langle m\!+\!1 \rangle$ and  $L^{--}P_{\bf a}(m\!+\!1) \backslash L^{--}P_{\bf a}(m)$ are abelian for $m \geq 1$.
\end{rmk}

\begin{lm} \la{facto} Let $m \geq 1$. There is a factorization of functors of $k$-algebras
\begin{equation} \la{facto_eq}
L^{--}P_{\bf a} = L^{--}P_{\bf a}\langle m+1 \rangle \cdot \prod_{\alpha > 0} U_{\alpha}\{m,1\} \cdot \prod_{\alpha < 0} U_{\alpha}\{m\!-\!1,0\},
\end{equation}
where 
\begin{itemize}
\item for $j \geq i$, the factor $U_{\alpha}\{j,i\}$ is the affine $k$-space whose $R$-points consist of the elements of the form $u_{\alpha}(x_{-j,\alpha} t^{-j} + \cdots + x_{-i, \alpha} t^{-i})$, where $x_{-l, \alpha} \in R$ for $i \leq l \leq j$;
\item the products $\prod_{\alpha > 0}$ and $\prod_{\alpha < 0}$ are taken in any order.
\end{itemize}
\end{lm}

\begin{proof}
From (\ref{norm_sgs}), (\ref{quots}), and (\ref{prod_eq}), it is clear that
\begin{equation} \la{proto_facto}
L^{--}P_{\bf a} = L^{--}P_{\bf a}\langle m + 1 \rangle \cdot \prod_{\alpha > 0} U_{\alpha-m} \cdot \prod_{\alpha<0} U_{\alpha-m+1} \cdots  \prod_{\alpha < 0} U_{\alpha -1} \cdot \prod_{\alpha > 0} U_{\alpha - 1} \cdot \prod_{\alpha < 0} U_\alpha,
\end{equation}
and the only task is to reorder the affine root groups to achieve (\ref{facto_eq}).

By (\ref{comm_rel1}) and (\ref{comm_rel2}), any factor in $\prod_{\alpha < 0}U_{\alpha-1}$ can be commuted to the right past all factors of any element in $\prod_{\alpha >0} U_{\alpha-1}$, at the expense of introducing after each commutation an element of $L^{(-2)}G$, which can be conjugated and absorbed (since $L^{(-2)}G\triangleleft L^{--}P_{\bf a}$) into the group to the left of $\prod_{\alpha < 0} U_{\alpha-1} \cdot \prod_{\alpha > 0} U_{\alpha-1}$, which is
$$
L^{--}P_{\bf a}\langle m + 1 \rangle \cdot \prod_{\alpha > 0} U_{\alpha-m} \cdot \prod_{\alpha<0} U_{\alpha-m+1} \cdots  \prod_{\alpha < 0} U_{\alpha -2} \cdot \prod_{\alpha > 0} U_{\alpha - 2} = L^{--}P_{\bf a}(2).
$$
Thus the above product can be written
$$
L^{--}P_{\bf a}\langle m + 1 \rangle \cdot \prod_{\alpha > 0} U_{\alpha-m} \cdot \prod_{\alpha<0} U_{\alpha-m+1} \cdots \prod_{\alpha <0} U_{\alpha-2} \cdot \big(\prod_{\alpha >  0} U_{\alpha -2} \cdot \prod_{\alpha > 0} U_{\alpha -1}\big) \cdot \big(\prod_{\alpha < 0} U_{\alpha - 1} \cdot \prod_{\alpha <0} U_\alpha \big).
$$
Similarly, we commute factors of $\prod_{\alpha <0} U_{\alpha-2}$ past factors of $\big(\prod_{\alpha >0} U_{\alpha -2} \prod_{\alpha > 0} U_{\alpha-1}\big)$, introducing commutators, which thanks to (\ref{comm_rel1}), (\ref{comm_rel2}) belong to $L^{(-3)}G$, hence can be absorbed into the group appearing to the left of $\prod_{\alpha < 0} U_{\alpha -2} \cdot \prod_{\alpha > 0} U_{\alpha-2}$, which is
$$
L^{--}P_{\bf a}\langle m + 1 \rangle \cdot \prod_{\alpha > 0} U_{\alpha-m} \cdot \prod_{\alpha<0} U_{\alpha-m+1} \cdots  \prod_{\alpha < 0} U_{\alpha -3} \cdot \prod_{\alpha > 0} U_{\alpha - 3} = L^{--}P_{\bf a}(3). 
$$
Continuing, we get an equality
$$
L^{--}P_{\bf a} = L^{--}P_{\bf a}\langle m+1\rangle \cdot \big(\prod_{\alpha>0}U_{\alpha-m} \cdots \prod_{\alpha > 0} U_{\alpha-1}\big) \cdot \big( \prod_{\alpha <0}U_{\alpha -m+1} \cdots \prod_{\alpha<0} U_\alpha \big).
$$
Consider $U_{\alpha}\{\infty, i\} := \cup_{j \geq i} U_{\alpha}\{j, i\} = L^{(-i)}U_\alpha$.
Clearly $\prod_{\alpha <0}U_{\alpha -m+1} \cdots \prod_{\alpha<0} U_\alpha$ belongs to the group
$$
\prod_{\alpha < 0} U_{\alpha} \{\infty, 0\} = L^{(-m)}\bar{U} \cdot  \prod_{\alpha < 0} U_{\alpha}\{m\!-\!1,0\}.
$$
We then commute the part in $L^{(-m)}\bar{U}$ to the left past the $\prod_{\alpha>0}U_{\alpha-m} \cdots \prod_{\alpha > 0} U_{\alpha-1}$ factor; the commutators which arise lie in $L^{(-m-1)}G$, and so they, like $L^{(-m)}\bar{U}$, get absorbed into $L^{--}P_{\bf a}\langle m+1 \rangle$. Finally we arrive at a decomposition
$$
L^{--}P_{\bf a} = L^{--}P_{\bf a}\langle m+1 \rangle \cdot  \big(\prod_{\alpha>0}U_{\alpha-m} \cdots \prod_{\alpha > 0} U_{\alpha-1}\big) \cdot \prod_{\alpha < 0} U_\alpha \{m-1,0\},
$$
and applying the same argument to 
$$
\prod_{\alpha>0}U_{\alpha-m} \cdots \prod_{\alpha > 0} U_{\alpha-1} \subseteq L^{(-m-1)}U \cdot \prod_{\alpha > 0} U_\alpha \{m,1\}
$$
yields the decomposition (\ref{facto_eq}). The fact that the latter is really a direct product is straightforward.
\end{proof}

For each root $\alpha$, let $i_{\alpha, \bf f}$ be the smallest integer such that $\alpha - i_{\alpha, \bf f} \overset{\bf f}{<} 0$.  Of course, $i_{\alpha, \bf f} \geq 0$ for all $\alpha$, and $i_{\alpha, \bf f} \geq 1$ if $\alpha > 0$.

\begin{pr} \la{alt_defn}
For any integer $m \geq 1$ such that $L^{--}P_{\bf a}\langle m+1 \rangle \subseteq L^{--}P_{\bf f}$,  we have the equality of functors on $k$-algebras
\begin{align} 
L^{--}P_{\bf f} ~ &=  ~ L^{--}P_{\bf a}\langle m+1 \rangle \cdot \langle U_{\alpha + n} \, | \, \alpha + n \overset{\bf f}{<} 0 \rangle \la{alt_defn_eq1}\\
&= L^{--}P_{\bf a} \langle m+1 \rangle \cdot \prod_{\alpha > 0} U_\alpha \{ m, i_{\alpha, \bf f}\} \cdot \prod_{\alpha < 0} U_\alpha\{m-1, i_{\alpha, \bf f}\} \la{alt_defn_eq2}
\end{align}
where $\langle  U_{\alpha + n} \, | \, \alpha + n \overset{\bf f}{<} 0 \rangle$ is the smallest ind-Zariski-closed subgroup of $LG$ containing the indicated affine root groups $U_{\alpha + n}$. Moreover, \textup{(}\ref{alt_defn_eq2}\textup{)} is a direct product of functors.
\end{pr}

\begin{proof}
Suppose $\alpha + n \overset{\bf f}{<} 0$ and $w \in \w{W}_{\bf f}$. By (\ref{U_conj}), $^wU_{\alpha + n} = U_{w(\alpha + n)}$. As $\w{W}_{\bf f}$ preserves ${\bf f}$, $w(\alpha  + n) \overset{\bf f}{<} 0$. Thus $w(\alpha + n) < 0$, which implies that $^wU_{\alpha + n} \subset L^{--}P_{\bf a}$. This shows that the r.h.s.\,of (\ref{alt_defn_eq1}) is contained in $L^{--}P_{\bf f}$.  It is clear that (\ref{alt_defn_eq2}) is contained in the r.h.s.\,of (\ref{alt_defn_eq1}).  Therefore it remains to show that $L^{--}P_{\bf f}$ is contained in (\ref{alt_defn_eq2}).

Suppose an element $g$ in (\ref{facto_eq}) belongs to $L^{--}P_{\bf f}$, yet its factor corresponding to some $\alpha$ does not lie in (\ref{alt_defn_eq2}).  Without loss of generality, the element $g$ has trivial component in $L^{--}P_{\bf a}\langle m+1 \rangle$, hence it belongs to $U(R[t^{-1}]) \cdot \bar{U}(R[t^{-1}])$; write it as a tuple $g = (g_\beta)_\beta$, where $\beta \in \Phi(G, T)$ and $g_\beta \in U_\beta(R[t^{-1}])$. We may write $g_\beta = u_\beta(x_{-j, \beta}t^{-j} + \cdots + x_{-i, \beta}t^{-i})$ for some $0 \leq i \leq j$ depending on $g, \beta$. 

We must have $g_\alpha = u_\alpha(x_{-j, \alpha}t^{-j} + \cdots + x_{n, \alpha}t^{n})$ where $\alpha +n \overset{\bf f}{\geq} 0$, $-j \leq n \leq 0$, and $x_{n, \alpha} \neq 0$. We will show this leads to a contradiction. This will prove the proposition.

First note that $\alpha +n \overset{\bf f}{\geq} 0$ implies $n = 0$ or $n = -1$.  Indeed, on ${\bf a}$ we have $-1 < \alpha < 1$, so on ${\bf f}$ we have $n-1 \leq \alpha + n \leq n+1$ by continuity. Therefore $n + 1 \geq 0$. 

\noindent {\bf Case 1}: $n = 0$.  Then by (\ref{facto_eq}) we have $\alpha < 0$. From $\alpha \overset{\bf f}{\geq} 0$, it follows that $\alpha = 0$ on ${\bf f}$, that is, $s_\alpha \in \w{W}_{\bf f}$. Now as $g \in L^{--} P_{\bf a} \cap \, ^{s_\alpha}L^{--}P_{\bf a}$, the reduction $\bar{g}$  modulo $t^{-1}$ belongs to 
$$
\bar{U} \cap \, ^{s_\alpha}\bar{U} = \prod_{\overset{\beta < 0}{s_\alpha(\beta) < 0}} U_\beta.
$$
But $\bar{g}$ contains the nontrivial factor $u_{\alpha}(x_{0,\alpha})$, which is impossible since $s_{\alpha}(\alpha) > 0$.

\smallskip

\noindent {\bf Case 2}: $n = -1$. Since $\alpha - 1 < 0$ and $\alpha - 1 \overset{\bf f}{\geq} 0$, we have $\alpha  -1 = 0$ on ${\bf f}$, that is, $s_{\alpha - 1} = t_{\alpha^\vee} s_\alpha = s_\alpha t_{-\alpha^\vee} \in \w{W}_{\bf f}$. 

By assumption, $g' := \, ^{s_{\alpha -1}}g \in L^{--}P_{\bf a} \subset L^-G$. Writing $g_\alpha = u_{\alpha}(x_{-j, \alpha}t^{-j} + \cdots + x_{-1, \alpha} t^{-1})$, using that $t_{\alpha^\vee}$ is identified with $\alpha^\vee(t^{-1}) \in T(k(\!(t)\!))$, and using $\langle \alpha, \alpha^\vee \rangle = 2$, we compute 
\begin{equation} \la{bad}
g'_{-\alpha} = u_{-\alpha}(\pm(x_{-j,\alpha}t^{2-j} + \cdots + x_{-1, \alpha} t)).
\end{equation}
Since $x_{-1, \alpha} \neq 0$, $g'_{-\alpha}$ does not belong to $U_{-\alpha}(R[t^{-1}])$.

On the other hand, $g' \in G(R[t^{-1}]) \, \cap \,  \big(U'(R[t, t^{-1}]) \cdot \bar{U}'(R[t, t^{-1}])\big)$, where $U' := \, ^{s_\alpha}U$.  But (\ref{bad}) shows that either the $\bar{U}'$-component or the $U'$-component of $g'$ does not lie in $\bar{U}'(R[t^{-1}])$, resp.\,$U'(R[t^{-1}])$. This contradicts Lemma \ref{GLN_calc} below (use $U', \bar{U}'$ as the $U,\bar{U}$ there). The proposition is proved.
\end{proof}

\begin{lm} \la{GLN_calc}
Let $R$ be a $k$-algebra, and suppose $\bar{u} \in \bar{U}(R[t, t^{-1}])$, and $u \in U(R[t, t^{-1}])$ have the property that $u \cdot \bar{u} \in G(R[t^{-1}])$.  Then $\bar{u} \in \bar{U}(R[t^{-1}])$ and $u \in U(R[t^{-1}])$.
\end{lm}

\begin{proof}
Use the fact that the multiplication map $U \times \bar{U} \rightarrow G$ is a closed immersion of $k$-varieties. Alternatively, use Remark \ref{Borel_pair_rem} to reduce to the case $G = {\rm GL}_N$, and then use a direct calculation with matrices. 
\end{proof}


\subsection{Iwahori-type decompositions} \la{Iwah_type_sec}

Let $L^{++}P_{\bf a} \subset L^+P_{\bf a}$ be the sub-group scheme over $k$ representing the functor which sends $R$ to the preimage of $U$ under the natural map
$$
P_{\bf a}(R[\![t]\!]) \hookrightarrow G(R[\![t]\!]) \rightarrow G(R[\![t]\!]/t).
$$ 
We abbreviate by setting $\mathcal U := L^{++}P_{\bf a}$. For two facets ${\bf f'}, {\bf f}$ in the closure of ${\bf a}$, we similarly use the abbreviations $\mathcal P := L^+P_{\bf f}$, \,$\overline{\mathcal U}_{\mathcal P} := L^{--}P_{\bf f}$, \,$\mathcal Q := L^+P_{\bf f'}$, and $\overline{\mathcal U}_{\mathcal Q} := L^{--}P_{\bf  f'}$.

Our first goal is to prove the following result.

\begin{pr} \la{proto_decomp2} For $w \in \widetilde{W}$, we have a decomposition of group functors
\begin{equation} \label{proto_decomp2_eq}
\overline{\mathcal U}_{\mathcal Q} = (\overline{\mathcal U}_{\mathcal Q} \cap \, ^w\overline{\mathcal U}_{\mathcal P}) \cdot (\overline{\mathcal U}_{\mathcal Q} \cap \, ^w\mathcal P),
\end{equation}
and moreover $$
\overline{\mathcal U}_{\mathcal Q} \cap \, ^w\mathcal P = \prod_{a} U_a
$$
where $a$ ranges over the finite set of negative affine roots with $a \overset{\bf f'}{<} 0$ and $w^{-1}a \overset{\bf f}{\geq} 0$, and the product is taken in any order.
\end{pr}

We will need a few simple lemmas before giving the proof.

\medskip

Let $x_0 \in {\bf a}$ be a sufficiently general point that distinct affine roots take distinct values on $x_0$.  For $a, b \in \Phi_{\rm aff}$, define $a \prec b$ if and only if $a(x_0) < b(x_0)$.  This is a total order on $\Phi_{\rm aff}$.  Let $\alpha_1, \dots, \alpha_r$ be the positive roots, written in increasing order with respect to $\prec$.  Choose $x_0$ sufficiently close to the origin so that for all $m \geq 1$ we have
$$
\alpha_1 - m \prec \alpha_2 - m \prec \cdots \prec \alpha_r -m \prec -\alpha_r - m+1 \prec \cdots \prec -\alpha_1 - m + 1.
$$
Choose an integer $m>>0$ and list all the affine roots appearing explicitly in (\ref{proto_facto}), as
$$
r_1, r_2, ......, r_M,
$$
in increasing order for $\prec$. This sequence has the advantage that for $1 \leq j \leq M+1$ 
\begin{equation} \la{Hj_defn}
H_{j} := L^{--}P_{\bf a}\langle m\!+\!1\rangle \cdot U_{r_1} \cdots U_{r_{j-1}}
\end{equation}
is a chain of groups, each normal in its successor, with $H_{j+1}/H_{j} \cong U_{r_{j}}$\, ( $1 \leq j \leq M$). We call $H_{r_j}$  the {\em group to the left of $r_j$}.  Note that $H_\bullet$ refines the chain coming from (\ref{norm_sgs}).

\begin{lm} \la{bfH_lem}
Let $\alpha >0$, and consider an integer $k \geq 0$.  Define subgroups
\begin{align*}
{\bf H}_{-\alpha - k} &= L^{(-k-1)}G \\
{\bf H}_{\alpha - k} &= L^{(-k)}G \cap L^{--}P_{\bf a}\langle k+1 \rangle.
\end{align*}
Let $\sigma \in \{ \pm 1\}$ and set $\beta := \sigma \alpha$.  Then for $\beta -k < 0:$
\begin{enumerate}
\item[(1)] ${\bf H}_{\beta-k} \, \triangleleft L^{--}P_{\bf a}$ and ${\bf H}_{\beta-k}$ lies in the group to the left of $\beta-k$.
\item[(2)] Assume $-\beta-j <0$, i.e.,\,$j \geq 0$ and $j \geq 1$ when $\sigma = -1$. Then 
$$
[U_{\beta -k}, U_{-\beta - j}] \subset {\bf H}_{\beta - k}.
$$
\end{enumerate}
\end{lm}

\begin{proof}
Thanks to Remark \ref{Borel_pair_rem}, the normality statement can be reduced to $G = {\rm GL}_N$ and checked by a matrix calculation. Part (2) follows from (\ref{comm_rel2}). The rest is clear.
\end{proof}

\begin{lm} \la{[UaUb]} 
Let $a, b$ be negative affine roots. Then:
\begin{enumerate}
\item[(i)] ${\bf H}_a \subseteq {\bf H}_b$ if $a \preceq b$, and
\item[(ii)] 
$
[U_a, U_b] \subset {\bf H}_a \cdot \langle U_c \,  | c \preceq a + b \rangle.
$
\end{enumerate}
\end{lm}

\begin{proof}
Part (i) is clear, and part (ii) follows from Lemma \ref{bfH_lem}(2) combined with (\ref{comm_rel1}). 
\end{proof}

\medskip

\noindent {\em Proof of Proposition \ref{proto_decomp2}}: 

First consider the case where $\mathcal Q = \mathcal B$. Choose $m >\!> 0$ so that we have  $L^{--}P_{\bf a}\langle m+1 \rangle \subset L^{--}P_{\bf a} \cap \, ^wL^{--}P_{\bf f}$. Use (\ref{proto_facto}) to write $g \in L^{--}P_{\bf a}$ in the form
$$
g = h_\infty \cdot u_{r_1} \cdots u_{r_j} \cdots u_{r_M},
$$
where $h_\infty \in L^{--}P_{\bf a}\langle m+1 \rangle$, and $u_{r_j} \in U_{r_j}$.  We wish to commute ``to the far right'' all terms of the form $u_{r_j}$ with $w^{-1}r_j \overset{\bf f}{\geq} 0$, starting with the $\prec$-maximal such $r_j$ and continuing with the other such $r_j$ in decreasing order.  Fix $a = r_j$. It is enough to prove, inductively on the number $t$ of commutations of $u_a = u_{r_j}$ (to the right) already performed, that we can write
$$
h_j \cdot u_{r_{j+1}}\cdots u_{r_{j+t}} \cdot u_a \cdot u_b
$$
with $h_j \in H_j$, in the form
$$
h'_j \cdot u_{r_{j+1}} \cdots u_{r_{j+t}} \cdot u_b \cdot u_a,
$$
for a possibly different $h'_j \in H_j$. Write $u_a u_b = \Delta u_b u_a$, where by Lemma \ref{[UaUb]}(ii), $\Delta \in {\bf H}_{a} \cdot \langle U_c \, | \ c \preceq a+b \prec a \rangle$.  By Lemma \ref{bfH_lem}(1), the $u_{r_{j+1}} \cdots u_{r_{j+t}}$-conjugate of the ${\bf H}_{a}$-factor lies in $H_j$, so we can suppress it.  As for the product of $U_c$-terms, we successively commute $u_{r_{j+t}}$ past each of them until it is adjacent to $u_b$, introducing at each step more terms of the same form as $\Delta$; using Lemmas \ref{bfH_lem}(1) and \ref{[UaUb]} as needed, we can assume these are in $\langle U_c \, | \, c \prec a \rangle$. Then repeat with $u_{r_{j+t-1}}$, etc. In the end, all the commutators have been moved adjacent to $h_j$ and belong to $H_j$; then $h'_j$ is their product with $h_j$. 

Let us summarize what we have done so far: we started with the direct product factorization (\ref{proto_facto}), then we rearranged the $U_a$-factors, all the time retaining the factorization property, until at the end we achieved a factorization
$$
\ov{\mathcal U} = \big(\ov{\mathcal U} \cap \, ^w\ov{\mathcal U}_{\PP} \big) \cdot \prod_{a} U_a,
$$
where $a$ ranges over the set of roots with $a < 0$ and $w^{-1}a \overset{\bf f}{\geq} 0$.  It is therefore enough to prove that the closed embedding
$$
\prod_{a} U_a \, \hookrightarrow \, \ov{\mathcal U} \cap \,^w{\PP}
$$
is an isomorphism.  It suffices to check this after base-change to $\bar{k}$, so henceforth we work over $\bar{k}$.  

It is enough to prove that $\ov{\mathcal U} \cap \, ^w\PP$ is generated by the subgroups $U_a$ which it contains.  Choose $m >\!>0$ large enough that $L^{--}P_{\bf a}\langle m+1 \rangle \cap \, ^w\PP = \{ e \}$ (scheme-theoretically): this is possible because the off-diagonal coordinates of $^w\PP$ (in the ambient ${\rm GL}_N$ of Remark \ref{Borel_pair_rem}) are zero or have $t$-adic valuation bounded below, while the diagonal coordinates lie in $R[\![t]\!]$ (see Lemma \ref{R[[t]]_lem} below). Let us prove by induction on $j$ that $H_j \cap \, ^w\PP$ is generated by the subgroups $U_a$ it contains (see (\ref{Hj_defn})); the case $j=1$ was discussed above. Now abbreviate $H = H_j$, $U_a = U_{r_j}$, $P = \, ^w{\PP}$. It is enough to prove that $HU_a \cap P$ equals $H \cap P$ or $(H\cap P)U_a$.  We intersect the chain (\ref{Hj_defn}) with $P$, and get an inclusion of group schemes
\begin{equation*}
H \cap P \backslash HU_a \cap P \, \hookrightarrow \, H \backslash HU_a \, \cong \, U_a.
\end{equation*}
If the left hand side is not trivial, then, since the morphism is $T(\bar{k})$-equivariant, its image is not finite and hence it is all of $U_a$, and we have an isomorphism $H \cap P \backslash HU_a \cap P \overset{\sim}{\rightarrow} U_a$. Using the Lie algebra analogue of this, a variant of \cite[28.1]{Hum1} implies that $(H \cap P)U_a = HU_a \cap P$, as desired.  This completes the proof in the case where $\mathcal Q = \B$.

Now we consider the general case, where $\mathcal Q = L^+P_{\bf f'}$. By intersecting (\ref{norm_sgs}) with $L^{--}P_{\bf f'}$ we obtain an analogue of (\ref{proto_facto}) for $m >\!>0$:
\begin{equation} \label{Q-proto_facto}
L^{--}P_{\bf f'} = L^{--}P_{\bf a}\langle m+1 \rangle \cdot U^*_{r_1} \cdots U^*_{r_M}
\end{equation}
where 
$$
U^*_{r_j} = \begin{cases} U_{r_j}, \,\,\, \mbox{if $r_j \overset{\bf f'}{<} 0$} \\
e, \,\,\, \mbox{otherwise}.\end{cases}
$$
We have a chain of subgroups $H^*_{j} = L^{--}P_{\bf a}\langle m+1 \rangle \cdot U^*_{r_1} \cdots U^*_{r_{j-1}}$, and the same argument as above works.

Finally, we may order the $U_a$-factors in $\overline{\mathcal U}_{\mathcal Q} \cap \, ^w\mathcal P$ freely, thanks to \cite[Lem.\,2.1.4]{BTII}.  \qed

The following result is proved like Proposition \ref{proto_decomp2}.

\begin{pr} \la{proto_decomp1} In the notation above, we have a factorization of group functors
\begin{equation} \label{Iwah_P_facto_eq}
\mathcal U = (\mathcal U \cap \, ^w\overline{\mathcal U}_{\mathcal P}) \cdot (\mathcal U \cap \, ^w{\mathcal P}),
\end{equation}
and $\mathcal U \cap \, ^w\ov{\mathcal U}_{\mathcal P} = \prod_{a} U_a$, where $a$ ranges over the finite set of affine roots with $a > 0$ and $w^{-1}a \overset{\bf f}{<} 0$, and the product is taken in any order.
\end{pr}

We conclude this subsection with a lemma needed to complete the proof of Proposition \ref{proto_decomp2}.

\begin{lm} \la{R[[t]]_lem}
For any faithful representation $G \hookrightarrow {\rm GL}(V)$, there is a suitable $k$-basis $e_1, \dots, e_N$ for $V$ identifying ${\rm GL}(V)$ with ${\rm GL}_N$, and a corresponding ``diagonal'' torus $T_N$ as in Remark \ref{Borel_pair_rem}, such that the diagonal entries of $L^+P_{\bf f}(R)$ lie in $R[\![t]\!]$.
\end{lm}

\begin{proof}
For general $k$-algebras $R$, the proof uses some of the Tannakian description of Bruhat-Tits buildings and parahoric group schemes, and thus we will need to cite results from \cite{Wil, HW}. For reduced $k$-algebras (which suffice for the purposes of this paper), one can avoid citing this theory; see Remark \ref{avoid_Tannak_rem}. We abbreviate by writing $\mathcal O = k[\![t]\!]$ and $K = k(\!(t)\!)$. 

Let $x \in \mathbb X_*$ be a point in the apartment of $\mathcal B(G, K)$, let $V$ be any finite-dimensional $k$-representation of $G$, and write $V_\mathcal O$ for the representation $V \otimes_k \mathcal O$ of $G_{\mathcal O}$. Then in \cite{Wil, HW} is defined the {\em Moy-Prasad filtration}  by $\mathcal O$-lattices in $V \otimes_k K$
\begin{equation} \la{MP_filt_def}
V_{x,r} := \bigoplus_{\lambda \in X^*(T)} V_\lambda \otimes_\mathcal O \mathcal Ot^{\lceil r - \langle \lambda,x \rangle \rceil},
\end{equation}
where $r \in \mathbb R$ and $V_{\lambda}$ is the $\lambda$-weight space for the action of $T_\mathcal O$ on $V_\mathcal O$. Note that $V_{x, r} \subseteq V_{x, s}$ if $r \geq s$ and $V_{x, r+1} = tV_{x, r}$. One can define the automorphism group ${\rm Aut}(V_{x, \bullet})$ to be the $\mathcal O$-group-functor whose points in an $\mathcal O$-algebra $R'$ consist of automorphisms of $V_{x, \bullet}$, that is, of tuples $(g_r)_r \in {\rm Aut}_{R'}(V_{x, r} \otimes_\mathcal O R')$ such that the ``diagram commutes'' and $g_{r+1} = g_r$ for all $r$. The following is a consequence of \cite{HW}.

\begin{lm} \la{cl_emb_lem} 
If $V$ is a faithful representation of $G$ and $x \in {\bf f}$, then $L^+P_{\bf f}(R)$ is a subgroup of $L^+{\rm Aut}(V_{x, \bullet})(R)$ for every $k$-algebra $R$.
\end{lm}

Now let $\lambda_1, \dots, \lambda_t$ be the distinct $T$-weights appearing in $V$. Choose a split maximal torus $T'$ of ${\rm GL}(V)$ containing $T \subset G \subset {\rm GL}(V)$. Let $\lambda_{ij}$ be the distinct weights of $T'$ which restrict to $\lambda_i$, and let $e_1, \dots, e_N$ be a basis of eigenvectors corresponding to $\{\lambda_{ij} \}_{i,j}$ for the $T'$-action on $V$, listed in some order. Using this we identify ${\rm GL}(V) \cong {\rm GL}_N$ and $T' \cong T_N$, the ``diagonal'' torus.  In this set-up, $L^+{\rm Aut}(V_{x, \bullet})$ is the group $k$-scheme parametrizing $R[\![t]\!]$-automorphisms of $\Lambda^{V,\bf f}_\bullet \otimes_\mathcal O R[\![t]\!]$ for some partial chain of $\mathcal O$-lattices $\cdots \Lambda^{V, \bf f}_{i} \subset \Lambda^{V, \bf f}_{i+1} \subset \cdots \subset \mathcal O^N$ of the form
$$
\Lambda^{V,\bf f}_i = t^{a_{i1}}\mathcal O e_1 \oplus \cdots t^{a_{iN}} \mathcal O e_N
$$
for certain integers $a_{ij}$. It is enough to prove that the diagonal elements of any $R[\![t]\!]$-automorphism of a 
single  $\Lambda^{V, \bf f}_i \otimes_\mathcal O R[\![t]\!]$ belong to $R[\![t]\!]$.  But this follows from a simple computation with matrices.
\end{proof}

\begin{rmk} \la{avoid_Tannak_rem}
If we only want to prove $L^{--}P_{\bf f, \rm red} \times L^+P_{\bf f} \rightarrow LG_{\rm red}$ is an open immersion \textup{(}which is what we use in all applications after $\S\ref{notsy}$\textup{)}, then we need $L^{--}P_{\bf f}(\bar{k}) \cap L^+P_{\bf f}(\bar{k}) = \{ e \}$, and thus we need Lemma \ref{R[[t]]_lem} for $R = \bar{k}$. In lieu of Lemma \ref{cl_emb_lem}, this can be proved by showing that $L^+P_{\bf f}(\bar{k})$ is contained in {\em some} parahoric subgroup $L^+P^{\rm GL_N}_{\bf f_N}(\bar{k})$ of an ambient ${\rm GL}_N$: note that $x$ belongs to {\em some} facet of the ambient apartment; as $P_{\bf f}(\bar{k}[\![t]\!])$ fixes that point and has trivial Kottwitz invariant \textup{(}cf.\,\cite[$\S5$]{PR},\cite[Prop.\,3]{HRa}\textup{)}, it fixes all the points in the ambient facet and belongs to the parahoric subgroup for that ambient facet.
\end{rmk}

\subsection{Parahoric big cells}

\subsubsection{Statement of theorem}

Our aim is to prove the following theorem, which plays a fundamental role in this article.

\medskip

\noindent{\bf Theorem 2.3.1}\,\, {\em The multiplication map gives an open immersion}
$$
L^{--}P_{\bf f} \times L^+P_{\bf f} \longrightarrow LG.
$$

\medskip

It is clear that $L^{--}P_{\bf f} \cap L^+P_{\bf f} = \{e\}$, ind-scheme-theoretically: take $\mathcal Q = \mathcal P$ and $w=1$ in Proposition \ref{proto_decomp2}.  Thus, we just need to check that $L^{--}P_{\bf f} \cdot L^+P_{\bf f}$ is open in $LG$.

\medskip

Suppose ${\bf f'}$ is in the closure of ${\bf f}$. By \cite[1.7]{BTII}, the inclusion $P_{\bf f}(k[\![t]\!]) \subset P_{\bf f'}(k[\![t]\!])$ prolongs to a homomorphism of group $k[\![t]\!]$-schemes $P_{\bf f} \rightarrow P_{\bf f'}$ and hence to a homomorphism of group $k$-schemes $L^+P_{\bf f} \rightarrow L^+P_{\bf f'}$. The latter is a closed immersion: as $P_{\bf f}$ is finite type and flat over $k[\![t]\!]$, it has a finite rank faithful representation over $k[\![t]\!]$ (\cite[1.4.3]{BTII}), which implies $L^+P_{\bf f} \hookrightarrow LP_{\bf f} = LG$ is a closed immersion. We obtain natural morphisms of ind-schemes
\begin{equation} \label{want_open}
\pi_{\bf f'} \, : \, LG \overset{\pi_{\bf f}}{\longrightarrow} LG/L^+P_{\bf f} \overset{\pi_{\bf f', \bf f}}{\longrightarrow} LG/L^+P_{\bf f'}.
\end{equation}
By \cite[Thm.\,1.4]{PR} the morphisms $\pi_{\bf f}$ and $\pi_{\bf f'}$ are surjective and locally trivial in the \'{e}tale topology, hence in particular $\pi_{\bf f}$, 
$\pi_{\bf f'}$, and $\pi_{\bf f', \bf f}$ are open morphisms.  As $\pi_{\bf f}$ is open,  the multiplication map $L^{--}P_{\bf f} \times L^+P_{\bf f} \rightarrow LG$ is an open immersion if and only if the map $L^{--}P_{\bf f} \rightarrow \mathcal F_{P_{\bf f}}$ given by $g \mapsto g\cdot x_e$ is an open immersion.  This allows us to define the big cell:

\begin{defi} \la{big_cell_defn}
We call the image of the open immersion $L^{--}P_{\bf f} \rightarrow \mathcal F_{P_{\bf f}}$, namely 
$$
\mathcal C_{\bf f} := L^{--}P_{\bf f} \cdot x_e,
$$
the {\em big cell} at $x_e$; it is a Zariski-open subset of the partial affine flag variety $\mathcal F_{P_{\bf f}}$.
\end{defi}

Before proving Theorem \ref{big_cell_thm}, we state an immediate consequence, which is used to prove Lemma \ref{ZLT_lem}:

\begin{cor} \la{big_cell_cor}
The morphisms in \textup{(}\ref{want_open}\textup{)} are locally trivial in the Zariski topology, and in particular, if $R$ is local, we have
$$
\mathcal F_{P_{\bf f}}(R) = G(R(\!(t)\!))/P_{\bf f}(R[\![t]\!]).
$$
\end{cor}

\subsubsection{Preliminary lemmas}

\begin{lm} \la{atof_lem}
If Theorem \ref{big_cell_thm} holds for $G$ and ${\bf a}$, it also holds for $G$ and ${\bf f}$.
\end{lm}
\begin{proof}
Using (\ref{decomp3}) we have
$$
L^{--}P_{\bf a} \cdot L^+P_{\bf a} = L^{--}P_{\bf f} \cdot (L^{--}P_{\bf a} \cap L^+P_{\bf f}) \cdot L^+P_{\bf a}.
$$
By the result for ${\bf a}$, this is an open subset of $LG$. Its right-translates under $L^+P_{\bf f}$ cover $L^{--}P_{\bf f} \cdot L^+P_{\bf f}$.
\end{proof}

Our plan is to reduce the theorem to the group ${\rm SL}_d$. We may choose a closed embedding $G \hookrightarrow {\rm SL}_d$, and by Remark \ref{Borel_pair_rem}, well-chosen Borel and unipotent radical subgroups of ${\rm SL}_d$ restrict to the corresponding objects in $G$.  However, it does not follow from this that the big cell $\bar{U}^{\rm SL_d} B^{\rm SL_d}$ in ${\rm SL}_d$ restricts to its counterpart in $G$. In order to reduce the theorem to ${\rm SL}_d$, therefore, we need to use a more flexible notion of big cell, where the restriction property is automatic.

For a homomorphism of $k$ groups $\lambda: \mathbb G_m \rightarrow G$, we define subgroups $P_G(\lambda)$ and $U_G(\lambda)$ of $G$ to consist of the elements $p$ (resp.\,$u$) with $\underset{t \rightarrow 0}{\rm lim}\, \lambda(t) p \lambda(t)^{-1}$ exists (resp.\,$= e$); see \cite[$\S2.1$]{CGP}. Define $\Omega_G(\lambda) = U_G(-\lambda) \cdot P_G(\lambda)$, a Zariski-open subset of $G$ isomorphic to $U_G(-\lambda) \times P_G(\lambda)$, by \cite[Prop.\,2.1.8]{CGP}. If $\lambda$ is $B$-dominant and regular, $\Omega_G(\lambda) = \bar{U}B$, the usual big cell in $G$.
\begin{lm} \la{bigcell_fctr}
Suppose $\pi: G \rightarrow G'$ is an inclusion. Let $\lambda : {\mathbb G}_m \rightarrow G$ be a homomorphism and define $\lambda' = \pi \circ \lambda$. Then
\begin{equation} \label{big_cell_pullback}
\pi^{-1}\Omega_{G'}(\lambda') = \Omega_G(\lambda).
\end{equation}
\end{lm}
\begin{proof}
This follows from \cite[Prop.\,2.1.8(3)]{CGP}. 
\end{proof}

For the next lemma, we fix a $B$-dominant and regular homomorphism $\lambda: \mathbb G_m \rightarrow T \hookrightarrow G$, and suppose we have a homomorphism of $k$-groups $f: G \overset{\iota}{\hookrightarrow} {\rm SL}_d \overset{\rho}{\rightarrow} {\rm SL}(V)$ where  $\iota$ is a closed embedding identifying $G$ with the scheme-theoretic stabilizer in ${\rm SL}_d$ of a line $L = kv$ in $V$.  (Given $\iota$, such a pair $(V,L)$ exists by e.g.\,\cite[Prop.\,A.2.4]{CGP}.)  Let $\lambda' = \iota \circ \lambda$, and let $\mathcal P(\lambda')$ (resp.~$\overline{\mathcal U}(-\lambda')$) be the groups $L^+P^{{\rm SL}_d}_{\bf f'}$ (resp.~$L^{--}P^{{\rm SL}_d}_{\bf f'}$) 
for ${\rm SL}(V)$ associated to the parabolic subgroup $P(\lambda')$ (resp.\,opposite unipotent radical $U(-\lambda')$) of ${\rm SL}_d$ (i.e.\,$P(\lambda')$ is the ``reduction modulo $t$ of $\mathcal P(\lambda') \subset {\rm SL}_d(k[\![t]\!])$,'' etc.).

\begin{lm} \la{par_bigcell_fctr}
In the above situation,
\begin{equation} \label{par_bigcell_pullback}
\iota^{-1}\big(\overline{\mathcal U}(-\lambda') \cdot \mathcal P(\lambda') \big) = L^{--}P_{\bf a} \cdot L^+P_{\bf a}.
\end{equation}
\end{lm}

\begin{proof}
The proof is a variation on the theme of \cite[proof of Cor.\,3]{F}, which concerns the case ${\bf f} = {\bf 0}$. Think of $\iota$ as an inclusion. By construction, we have $LG \cap \ov{\mathcal U}(-\lambda') = L^{--}P_{\bf a}$ and $LG \cap \mathcal P(\lambda') = L^+P_{\bf a}$, which proves the r.h.s.~of (\ref{par_bigcell_pullback}) is contained in the l.h.s.

Suppose we have $g^- \in \overline{\mathcal U}(-\lambda')$ and $g^+ \in \mathcal P(\lambda')$ and $g^- \cdot g^+ \in LG$.  We need to show that $g^-, \, g^+ \in LG$. Let $L_v$ be the scheme-theoretic line generated by $v$. Write $g^-(0)$ (resp.~$g^+(0)$) for the value of $g^-$ (resp.~$g^+$) at $t^{-1} =0$ (resp.~$t=0$), and also set $g^-_\infty := g^-(0)^{-1} g^-$ (resp.~$g^+_\infty := g^+ g^+(0)^{-1}$). Starting with
\begin{equation} \la{prod_eq}
\rho(g^-(0) g^-_\infty) \cdot \rho(g^+_\infty g^+(0)) \, v \in L_v,
\end{equation}
comparing coefficients of $t^{-1}$ and $t$ shows that $\rho(g^-(0)) \cdot \rho(g^+(0)) \, v \in L_v$, and hence $g^-(0) \cdot g^+(0) \in G \cap \big(U(-\lambda') \cdot P(\lambda')\big)$. By Lemma \ref{bigcell_fctr}, we see that $g^-(0) \in G$ and $g^+(0) \in G$.  Now going back to (\ref{prod_eq}), we deduce 
$$
\rho(g^-_\infty \cdot g^+_\infty) \, v \in L_v.
$$
Then (looking at $R$-points), there is a $c \in R^\times$ such that
$$
\rho(g^-_\infty)^{-1} \, cv = \rho(g^+_\infty)\, v.
$$
Therefore this element belongs both to $cv + t^{-1}V[t^{-1}]$ and to $v + tV[\![t]\!]$; thus both sides are equal to $cv$, and we see $g^-_\infty, \, g^+_\infty \in LG$, as desired.
\end{proof}

\subsubsection{Reduction to case ${\rm SL}_d$, ${\bf f} = {\bf a}$}

Suppose we know the theorem holds for ${\rm SL}_d$ when the facet is a particular alcove.  Since all alcoves are conjugate under the action of ${\rm SL}_d(k(\!(t)\!))$ on its Bruhat-Tits building, the theorem holds for ${\rm SL}_d$ and {\em any} alcove. Then by Lemma \ref{atof_lem}, it holds for ${\rm SL}_d$ and any facet.  Therefore the subset $\overline{\mathcal U}(-\lambda') \cdot \mathcal P(\lambda')$ of Lemma \ref{par_bigcell_fctr} is open in $L{\rm SL}_d$, and hence by that lemma, the theorem holds for any $G$ when the facet is an alcove.  By Lemma \ref{atof_lem} again, it holds for any $G$ and any facet.

\subsubsection{Proof for ${\rm SL}_d$, ${\bf f} = {\bf a}$}

In \cite[p.\,42-46]{F}, Faltings proved Theorem \ref{big_cell_thm} for ${\bf f} = {\bf 0}$ and ${\bf f} = {\bf a}$,  for any semisimple group $G$. For ${\rm SL}_d$ and ${\bf f} = {\bf 0}$, this result was proved earlier by Beauville and Laszlo \cite[Prop.\,1.11]{BLa}.

Here, we simply adapt the method Faltings used for ${\rm SL}_d$ and ${\bf f} = {\bf 0}$ to elucidate, in an elementary way, the case ${\rm SL}_d$ and ${\bf f} = {\bf a}$. (For the most part, this amounts to giving an elaboration of the remarks at the end of \cite[$\S2$]{F}.) 

Let $R$ be a $k$-algebra, and we define for $0 \leq i \leq d$
\begin{align*}
\Lambda_i &:= R[\![t]\!]^i \, \oplus \, \big(t R[\![t]\!]\big)^{d-i} \\
 M_i &:= \big(t^{-1}R[t^{-1}]\big)^i \, \oplus \, R[t^{-1}]^{d-i}.
\end{align*}
We have $\Lambda_0 \subset \Lambda_1 \subset \cdots \subset \Lambda_d = t^{-1}\Lambda_0$ and $M_0 \supset M_1 \supset \cdots \supset M_d = tM_0$.  Also,
$$
\Lambda_i \, \oplus \, M_i = R(\!(t)\!)^d,
$$
for $0 \leq i \leq d$.

Write $H = {\rm SL}_d$.  The affine flag variety $\mathcal F_H$ for $H$ is the ind-scheme parametrizing chains of projective $R[\![t]\!]$-modules $L_0 \subset L_1 \subset \cdots \subset L_d = t^{-1}L_0 \subset R(\!(t)\!)^d$, such that
\begin{enumerate}
\item[(1)] $t^n\Lambda_i  \subset L_i \subset t^{-n}\Lambda_i$ for all $i$ and $n >\!>0$
\item[(2)] ${\rm det}(L_i) = {\rm det}(\Lambda_i) = t^{d-i}R[\![t]\!]$.
\end{enumerate}

We consider the complex of projective $R$-modules, supported in degrees $-1$ and $0$ and of virtual rank $0$,
$$
0 \, \longrightarrow \, \frac{L_i \oplus M_i}{t^n\Lambda_i \oplus M_i} \, \overset{\varphi}{\longrightarrow} \, \frac{R(\!(t)\!)^d}{t^n\Lambda_i \oplus M_i} \, \longrightarrow \, 0.
$$
The determinant of this complex determines a line bundle $\mathcal L_i$ on $\mathcal F_H$ and the determinant of $\varphi$ gives a section $\nu_i$ of $\mathcal L_i$.   Let $\Theta_i$ be the zero locus of $\nu_i$. Then $\bigcap_i \mathcal F_H - \Theta_i$ is an open subset of $\mathcal F_H$ and consists precisely of the points $L_\bullet$ satisfying $L_i \oplus M_i = R(\!(t)\!)^d$ for all $i$. This locus contains the $L^{--}P_{\bf a}$-orbit of $x_e$, as it contains the base point $x_e = \Lambda_\bullet$ and is stable under $L^{--}P_{\bf a}$ since this stabilizes  $M_\bullet$. Our goal is to prove that the locus is precisely $L^{--}P_{\bf a}\cdot x_e$.

Assume $L_\bullet \in \bigcap_i \mathcal F_H - \Theta_i$. Write the $i$-th standard basis vector $e_i$ as
$$
e_i = \lambda_i + \big(\sum_{j \leq i} t^{-1}a_{ji} e_j \, + \, \sum_{j > i} a_{ji} e_j\big) \in L_i \oplus M_i$$ 
where $a_{ji} \in R[t^{-1}]$ and $\lambda_i \in L_i$, \,$\forall i,j$. It follows that there is a unique matrix $h \in R[t^{-1}]^{d \times d}$ whose reduction modulo $t^{-1}$ is strictly lower triangular, such that $h(e_i) \in L_i$ for all $i$. Since $tL_{d} \subseteq L_i$, we easily see that $h(\Lambda_i) \subset L_i$ for all $i$. 

We claim that $h(\Lambda_i) = L_i$ for all $i$.  We start with $i=d$. It is enough to prove $h(\Lambda_d)$ generates the $R[\![t]\!]$-module $L_d/t^n\Lambda_d$. Each element in this quotient can be represented by an element $f \in L_d \subset \Lambda_d \oplus M_d$ whose projection to $\Lambda_d$ is an $R$-linear combination of $t^l e_1, \dots, t^{l} e_d$ for $l < n$.  But $h(e_j)$ is $e_j$ plus an $R[t^{-1}]$-linear combination of the elements $t^{-1}e_1, \dots, t^{-1}e_j, e_{j+1}, \dots, e_d$. Thus by decreasing induction on $l$ (and working with the coefficients of $e_1, e_{2}, \dots,$ in that order) we can make the $\Lambda_d$-projection of $f$ and thus also $f$ itself vanish, proving $h(\Lambda_d) = L_d$.

Since ${\rm det}(\Lambda_d) = {\rm det}(L_d)$, we see ${\rm det}(h) \in R[\![t]\!]^\times \cap (1+ R[t^{-1}])$, so ${\rm det}(h) =1$ and therefore $h \in L^{--}P_{\bf a}$. Also $h$ induces an isomorphism $\Lambda_d/\Lambda_0 \overset{\sim}{\rightarrow} L_d/L_0$. Therefore, by induction on $i$, $h : \Lambda_i/\Lambda_0 \overset{\sim}{\rightarrow} L_i/L_0$ and $h(\Lambda_i) = L_i$ for all $i$.

We conclude that the morphism $\bigcap_i \mathcal F_H - \Theta_i \, \rightarrow \, L^{--}P_{\bf a}$ defined by $L_\bullet \mapsto h$ is inverse to the $L^{--}P_{\bf a}$-action on $\Lambda_\bullet = x_e$. This completes the proof of Theorem \ref{big_cell_thm}.
\qed

\subsection{Uniform notation  for the finite case $G$ and for the affine case $LG$}\la{notsy}$\,$

We introduce a   unified notational system that allows us to discuss the usual partial flag varieties and the partial affine flag varieties at the same time.  We use symbols $\mathcal G$, $\mathcal W$, etc., to abbreviate the objects above them in the following table:

\vskip.3cm

\begin{center}
\begin{tabular}{c|c|c|c|c|c|c|c|c|c|c|c}
 $LG$ & $\widetilde{W}$ & $S_{\rm aff}$  & $L^+P_{\bf a}$ & $L^+P_{\bf f}$ & $L^{++}P_{\bf a}$  & $L^{--}P_{\bf a}$ & $L^{--}P_{\bf f}$ & $\widetilde{W}_{\bf f}$ & $LG/L^+P_{\bf f}$ & $Y_{\bf f}(w)$ & $X_{\bf f', \bf f}(w)$  \\ \hline
\vspace{.1in}  
$\mathcal G$ & $\mathcal W$ & $\mathcal S$ & $\mathcal B$ & $\mathcal P$ & $\mathcal U$ & $\overline{\mathcal U}$ & $\overline{\mathcal U}_{\mathcal P}$ & $\mathcal W_{\mathcal P}$ & $\mathcal G/\mathcal P$ &  $Y_{\mathcal P}(w)$ & $X_{\mathcal P' \mathcal P}(w)$
\end{tabular}
\end{center}

\vskip.3cm
In particular, we will denote the big cell $\mathcal C_{\bf f}$ in $LG/L^+P_{\bf f} = \mathcal G/\PP$ attached to $L^+P_{\bf f} = \PP$ by
$$
\mathcal C_{\PP} = \ov{\mathcal U}_{\PP}\, x_e.
$$
We define the {\em big cell at $x_v$} to be
\begin{equation} \la{bigcell}
v\mathcal C_{\PP}  = \,^{v}\ov{\mathcal U}_{\PP} \, x_v.
\end{equation}

Also, if $\PP = L^+P_{\bf f}$, we sometimes write $\overset{\PP}{<} $ intead of $\overset{\bf f}{<} $.  From now on, we will call $\PP$ a ``parahoric" group. Recall that $\mathcal W_{\PP}$ is always a finite subgroup of $\mathcal W = \w{W}$.

The new notation is modeled on the customary notation for finite flag varieties. If $P \supset B$ is a standard $k$-rational parabolic subgroup of $G$, it corresponds to a standard parahoric subgroup $\mathcal P$, and we have embeddings $G/P \hookrightarrow \mathcal G/\mathcal P$, and similar inclusions on the level of Schubert varieties, Weyl groups, etc. All of our results for convolution morphisms or Schubert varieties for partial affine flag varieties for $LG$ have analogues for partial flag varieties for $G$.  The big cells $\mathcal C_{\PP}$ are then just the more standard objects $\bar{U}_P P/P \subset G/P$.

Henceforth, when we discuss $\mathcal G/\mathcal P$, $X_{\PP}(w)$, etc., we shall give these object their {\em reduced} structure.

The following is familiar and it exemplifies the use of big cells, in this case in $\G/\PP$.

\begin{lm}\la{ZLT_lem}  {\rm ({\bf  $\G/\B \to \G/\PP$ is a $\PP/\B$-bundle})}
 The $k$-ind-projective map  of $k$-ind-projective varieties  $\mathcal G/\mathcal B \rightarrow \mathcal G/\mathcal P$ is 
a Zariski locally trivial fibration  over  $\mathcal G/\mathcal P$ with fiber the  geometrically integral, 
smooth, 
projective, rational and homogenous $k$-variety $\mathcal P/\mathcal B$.
\end{lm}
\begin{proof} 
We may cover $\mathcal G/\mathcal P$ with open big cells $g\m{C}_\PP= g \ov{\mathcal U}_\mathcal P  \, x_e$. It is immediate to verify that
 the inverse image in $\mathcal G/\mathcal B$ of such a set is isomorphic to the product  $g\bar{\mathcal U}_\mathcal 
 P \times (\mathcal P/\mathcal B)$: send $(hu_\PP, p) \mapsto hu_\PP p$. 
\end{proof}


For reference purposes we list some consequences of Propositions \ref{proto_decomp2} and \ref{proto_decomp1}: for $v \in \w{W}$
\begin{align}
\,^{v^{-1}}\mathcal{U} &= (\,^{v^{-1}}\mathcal{U}\cap \ov{\mathcal{U}}_{\mathcal{P}}) 
\,(^{v^{-1}}\mathcal{U}\cap\mathcal{P})
\label{decomp1}
\\
\, ^v\overline{\mathcal U}_\mathcal P&=(\mathcal U\cap\, ^v\overline{\mathcal U}_\mathcal P) ( \overline{\mathcal U}\cap \, ^v\overline{\mathcal U}_\mathcal P)
\label{decomp2}
\\
\, \ov{\mathcal U} &=  \ov{\mathcal U}_{\PP} \cdot (\ov{\mathcal U} \cap \PP). \label{decomp3}
\end{align}

At least some of these were known before (although we could not locate proofs in the literature). For example, the decomposition $\ov{\mathcal U} = (\ov{\mathcal U} \cap \,^w\ov{\mathcal U}) (\ov{\mathcal U} \cap \, ^w\mathcal U)$, a special case of (\ref{decomp2}), was stated by Faltings \cite[page 47--48]{F}.

\subsection{Orbits and relative position}\la{orbtz}$\;$

For the purpose of discussing orbits, let us fix, temporarily and for ease of exposition, a  configuration of parahoric subgroups of $\G$:  $\m{Z} \supseteq \m{X}  \supseteq \B \subseteq \m{Y}.$

The group $\m{W}$ contains the finite subgroups $\m{W}_\m{Z} \supseteq \m{W}_\m{X}  \supseteq  \m{W}_{\B} = \{1\}  \subseteq \m{W}_\m{Y}.$ The double  coset spaces $\bwp{\m{Z}}{\m{Y}}: = \m{W}_\m{Z} \backslash 
\m{W} / \m{W}_{\m{Y}}$ inherit a natural poset structure from $(\m{W}, \leq)$.

We have the natural surjective map of posets   $\bwp{\m{X}}{\m{Y}} \to  {\bwp{\m{Z}}{\m{Y}}}$.

The Bruhat-Tits decomposition takes the form 
$
\G = \coprod_{z \in {  \bwp{\m{Z}}{\m{Y} }}}   {\m{Z}z \m{Y}}.
$

For  every $z \in {\bwp{\m{Z}}{\m{Y}}},$ we have  the finite union decomposition 
$Y_{\m{Z}\m{Y}}(z)= \coprod_{x\mapsto z} Y_{\m{X}\m{Y}}(x)$
of the corresponding  $\m{Z}$-orbit in $\G/\m{Y}$, where $x \in {\bwp{\m{X}}{\m{Y}}}.$
Similarly, for the orbit closures
$X_{\m{Z}\m{Y}}(z) = \coprod_{\zeta \leq z} Y_{\m{Z}\m{Y}}(\zeta)$ (inequality in the poset
${\bwp{\m{Z}}{\m{Y}}}$). Of course, we have that $X_{\m{Z}\m{Y}}(z)= \ov{\m{Z}z \m{Y}/\m{Y}},$ etc.


The decomposition \eqref{decomp1} implies that $Y_{\B \PP}(v)$ is an affine space:
\begin{equation} \label{Y_B_in_big_cell}
Y_{\B \PP}(v) = \mathcal U v \, x_e = v 
\left(^{v^{-1}}\mathcal U \cap \overline{\mathcal U}_{\PP}\right) x_e
\cong\;^{v^{-1}}\mathcal U \cap \overline{\mathcal U}_{\PP}=\prod_{\alpha+n\in S} U_{\alpha+n}\cong \mathbb{A}^{|S|}
\end{equation}
where 
$S=\{\alpha+n\;|\;v(\alpha+n)>0, \textrm{ and }\alpha+n\stackrel{\PP}{<}0\}$.  
The dimension $|S|$ can also be described as the length $\ell(v_{min})$, where $v_{min}$ is the minimal representative in the coset $v\mathcal{W}_\mathcal{P}$.

\begin{lm}\la{proj}
Let $X$ be an ind-projective ind-scheme over $k$. Let $Y\subseteq X$ be a closed sub-ind-scheme over $k$ that is also an integral $k$-scheme. Then $Y$ is a $k$-projective scheme.
\end{lm}
\begin{proof}
Let $X= \cup_{n\geq 0} X_n$ be an increasing sequence of closed projective $k$-subschemes of $X$ which exhaust $X$. There is $n_0$ such that the generic point of $Y$ is contained in $X_{n_0}$.  The intersection $Y \cap X_{n_0}$  is a closed subscheme of $Y$ and contains the generic point of $Y$; since $Y$ is reduced, $Y = Y \cap X_{n_0}$.  It follows
that $Y$ is a closed $k$-subscheme of the projective $k$-scheme $X_{n_0}$ and, as such, it is $k$-projective.
\end{proof} 

Faltings \ci{F} (in the $\G/\B$-setting)  and Pappas-Rapoport \ci{PR} have proved that the $\PP$-orbit closures
 in $\G/\PP$, when given their reduced structure, are geometrically integral, normal, projective $k$-varieties. 
 The following is a consequence of their results.
 
\begin{pr}\la{fqco} {\rm ({\bf Normality of orbit closures})}
The orbit closures $X_{\m{Z}\m{Y}}(z)$, endowed with their 
reduced structure, are geometrically integral, normal, projective $k$-varieties.
\end{pr}
\begin{proof} 
Let $z_{max} \in \m{W}$ be the maximal representative of $z.$ Then $\ov{\B z_{max} \B}= \ov{\m{Y} z_{max} \m{Y}}$.
It follows that the  natural
map  $p: X_{\m{B}\m{B}}(z_{max}) \to X_{\m{Z}\m{Y}}(z)$  is  a Zariski locally trivial  fiber bundle with fiber the 
geometrically connected, nonsingular and projective $k$-variety $\m{Y}/\B$, as it coincides with the full pre-image
of $X_{\m{Z}\m{Y}}(z)$ under the natural projection map $\G/\B \to \G/\m{Y}$ (cf.\,Lemma \ref{ZLT_lem}).

 According to \ci{F} and \ci{PR},
$X_{\m{B}\m{B}}(z_{max})$,   being a $\B$-orbit closure in $\G/\B$ endowed with the reduced structure,
is  a geometrically integral, normal, projective $k$-variety.  By using the  fact that $p$ is a Zariski locally trivial bundle,  and the fact that
$X_{\m{B}\m{B}}(z_{max})$ is  
quasi-compact and  of finite type
over $k,$ we have that the same is true for    $X_{\m{Z}\m{Y}}(z)$. According to Lemma 
\ref{proj},  the orbit-closure $X_{\m{Z}\m{Y}}(z)$ is then $k$-projective.

The desired conclusions, except the  already-proved projectivity assertion, follow by descending the desired geometric integrality and normality
from the pre-image to the image, along the smooth projection map $p.$
\end{proof}

Let us now fix a configuration of parahoric subgroups of $\G$:  $\m{Q} \supseteq \PP  \supseteq \B.$
We have the natural projections, which are maps of posets
\beq\la{e3.5}
\xymatrix{
& & \W \ar[ld] \ar[rd]  & &   \\
&_\PP \W_\B  \ar[rd] \ar[ld] &&  \,_\B \W_\PP \ar[ld] \ar[rd] &  \\
 _\m{Q} \W_\B  \ar[rd] && _\PP \W_\PP \ar[ld] \ar[rd]  &  & _\B \W_\m{Q}  \ar[ld] \\
& _\m{Q} \W_\PP  \ar[rd] &&  \,_\PP \W_\m{Q} \ar[ld] &  \\
&  & \,_\m{Q} \W_\m{Q}. & &
}
\eeq
Each rhomboid, including the big one, is determined by two of the three parahorics. 
For each rhomboid, we denote the system of images of an element $w$ in the summit as follows
\beq\la{e3.6}
\xymatrix{
& w \ar[ld] \ar[rd]  &   \\
'{w}  \ar[rd] &&  w' \ar[ld]  \\
 & w'', & 
}
\eeq
Of course, when dealing with orbits and their closures, we write $Y_\B (w)$ in place of $Y_{\B \B}(w)$, etc.

Consider the rhomboid  determined by $\PP,\m{Q}$. The pre-image in $_{\m{Q}} W _\PP$ of $w''$ is
the finite collection of closures of  $\m{Q}$-orbits in $\G/\PP$ that surject onto $X_{\m{Q}}(w'')$.  The pre-image in $_\PP W _\m{Q}$ of $w''$ is
the finite collection of closures of $\PP$-orbits in $\G/\m{Q}$ that are in the closure $X_\m{Q} (w'')$. 
The pre-image in $_\PP W_\PP$ of $w''$ is
the finite collection of the closure of  $\PP$-orbits in $\G/\PP$ that  map into $X_\m{Q}(w'')$ and, among them, we find
$X_\PP (w_{max})$ i.e. the full-preimage in $\G/\PP$ of $X_\m{Q}(w'')$ under $\G/\PP \to \G/\m{Q}$, so that the resulting map
is a  Zariski locally trivial  $\PP/\B$-bundle.
If $w_1,w_2 \mapsto w''$, then we have
\beq\la{4es}
\ov{ \PP w_1\PP /\PP} \to \ov{\m{Q}w_1\m{Q}/\m{Q}} = \ov{\m{Q}w''\m{Q}/\m{Q}} = 
 \ov{\m{Q} w_2 \m{Q}/\m{Q}} \leftarrow \ov{\PP w_2 \PP/ \PP}
\eeq
and $ \ov{\m{Q}w''\m{Q}/\m{Q}}$ is the smallest $\m{Q}$-orbit-closure   in $\G/\m{Q}$ containing
the $\PP$-orbits-closures  $\ov{\PP w_i \m{Q}/\m{Q}}$.

\begin{defi}\la{qpma}{\rm
We say that $w \in {\bwp{\PP}{\PP}}$ is  of {\em $\m{Q}$-type}  if $X_\PP (w)$ is $\m{Q}$-invariant; this is equivalent to
having $\ov{\PP w \PP} =\ov{\m{Q}w\PP}$; it is also equivalent to $w$ admitting a lift
in $^\Q\W^\PP \subseteq \W$, the set of maximal representatives of $_\Q\W_\PP$ inside  $\W$.
We say that it is {\em $\m{Q}$-maximal} if it is the maximal  representative of its image $w''$; 
this is equivalent to $\ov{\PP w \PP} =\ov{\m{Q}w{Q}}$;
it is also equivalent to $w$ admitting a lift
in $^\Q\W^\Q \subseteq \W$, the set of maximal representatives of $_\Q\W_\Q$ inside  $\W$.
}
\end{defi}

Note that being of $\m{Q}$-type means that  $X_\PP(w) = X_\PP (w')$, and it implies that $X_\PP(w) \to X_\m{Q}(w'')$ is surjective
(the converse is not true: take $\G/\PP \ni \PP/\PP \to \m{Q}/\m{Q} \in \G/\m{Q}$).
Note that being of $\m{Q}$-maximal type is equivalent to $X_\PP(w) \to X_\m{Q}(w'')$ being a Zariski locally trivial bundle with fiber $\m{Q}/\PP$. Finally,  if $w$ is $\m{Q}$-maximal, then $w$ is of $\m{Q}$-type, but not vice versa.

Let $\PP_1, \PP_2  \in \G/\PP$.  According to the Bruhat-Tits decomposition of $\G$, 
there is a unique  and well-defined   $w\in {\bwp{\PP}{\PP}}$ such that, having written $\PP_i:= g_i \PP$, $g_i \in \G$, we have that  $g_1^{-1} g_2 \in 
\PP w \PP$.

\begin{defi}\la{zrlp}
{\rm
Given $P_1, P_2 \in \G/\PP$ we define their  {\em relative position} to be the unique element $w \in \, _\PP \W_\PP$ 
such that $g_1^{-1} g_2 \in 
\PP w \PP$, and we denote this property by
$
P_1 \, \overset{ w}{\textendash \textendash} \, P_2
$
We say that their 
{\em relative position is less then or equal to $w$} if their relative position
is so, i.e.  $g_1^{-1} g_2 \in \ov{\PP w\PP},$ and we denote this property by $
P_1 \, \overset{\leq w}{\textendash \textendash} \, P_2
$. 
}
\end{defi}

The following statements can be interpreted at the level of $k$ or $\bar{k}$-points, but we will suppress this from the notation. 
We have
\begin{equation*}
Y_\mathcal B(w) = \{ \mathcal B' ~ | ~ \mathcal B \, \overset{w}{\textendash \textendash}\, \mathcal B' \} \hspace{.25in} \mbox{and} \hspace{.25in} X_\mathcal B(w) = \{ \mathcal B' ~ | ~ \mathcal B \, \overset{\leq w}{\textendash \textendash} \, \mathcal B' \}.
\end{equation*}

The BN-pair relations hold for $v \in \mathcal W$ and $s \in \mathcal S$:
\begin{align} \la{BN-pair_eq}
\mathcal B v \mathcal B s \mathcal B &= \begin{cases} \mathcal B vs \mathcal B, \,\,\,\,\,\,\,\,\,\,\,\,\,\,\,\,\,\,\,\,\,\,\, \mbox{if $v < vs$,} \\
\mathcal B vs \mathcal B \cup \mathcal B v \mathcal B, \,\,\,\,\, \mbox{if $vs < v$.} \end{cases} \\
s\mathcal B s &\nsubseteq \mathcal B. \notag
\end{align}
Note that for every $v \in \mathcal W$ and $s \in \mathcal S$, there is an isomorphism $\{ \mathcal B' ~ | ~ v\mathcal B \,\, \overset{\leq s}{\textendash \textendash} \,\mathcal B' \} \cong \mathbb P^1$ and $\{ \mathcal B' ~ | ~ v\mathcal B \, \,\overset{\leq s}{\textendash \textendash} \,\mathcal B' \} \subset Y_\mathcal B(v) \cup Y_\mathcal B(vs)$.

\section{Twisted products and generalized convolutions} \la{twcon_sec}

\subsection{Twisted product varieties}\la{ztpv} $\;$

Let  $r\geq 1$ and let $w_\bullet = (w_1, \ldots w_r) \in (\bwp{\PP}{\PP})^r.$
\begin{defi}\la{zdeftp}
{\rm 
The {\em twisted product scheme associated with $w_\bullet$} is   the closed $k$-ind-subscheme of $(\G/\PP)^r$  defined by setting
\beq\la{9i8}
X_\PP(w_\bullet) := \left\{ (P_1, \dots, P_r) ~ | ~ \PP \, \overset{\leq w_1}{\textendash \textendash} \, P_1 \, \overset{\leq w_2}{\textendash \textendash} \, \cdots \, \overset{\leq w_r}{\textendash \textendash} P_r \right\}.
\eeq
endowed with the reduced structure.}
\end{defi}

\begin{lm}\label{loc_triv}
Twisted products $X_\mathcal P(w_\bullet)$
are  Zariski-locally isomorphic to the usual products $X_\mathcal P(w_1) \times \cdots \times X_\mathcal P(w_r)$.
\end{lm}
\begin{proof}
For every $1\leq i \leq r$, pick any point $P_i \in X_\PP(w_i)$, and $\gamma_i \in \G$ so that 
$P_i = {\gamma_i}\PP$.
In particular, $\gamma_i \in \ov{\PP w_i \PP}.$
Intersect with the big cell (\ref{bigcell}) at $\gamma_i \PP$ to obtain the dense open subset $X_\PP (w_i) \bigcap \gamma_i \m{C}_\PP.$ Its elements
have the form 
$\gamma_i u_i$,  for a unique $u_i \in \bar{\m{U}}_{\PP}$, and
with $\gamma_i, \gamma_i u_i \in  \ov{\PP w_i \PP}.$
Let $\gamma := (\gamma_1, \dots, \gamma_r)$, and set $A_\gamma:= \prod_{i=1}^r X_\PP (w_i) \bigcap \gamma_i \m{C}_\PP.
$
 Then $A_\gamma$ is open and dense in $\prod_{i=1}^r X_\PP (w_i)$, and 
its points have the form 
$
(\gamma_1 u_1, \ldots , \gamma_r u_r), \; \gamma_i, \gamma_i u_i \in  
\ov{\PP w_i \PP}.
$
Define the map $A_\gamma \to X_\PP(w_\bullet)$ by the assignment:
\[
(\gamma_1 u_1, \ldots , \gamma_r u_r) \mapsto 
\left( \gamma_1 u_1,  \gamma_1 u_1  \gamma_2 u_2  , \ldots,
  \gamma_1 u_1  \gamma_2 u_2 \cdots  \gamma_r u_r   \right).
\]
Set $\w{\gamma}_i= \prod_{j=1}^i \gamma_{j} u_j \in X_\PP(w_\bullet)$.
The image  of this map in $X_P(w_\bullet)$ is contained in the  open and dense set $\w{A}_{\w{\gamma}}$ defined by requiring 
that:
 $g_1 \in \pi^{-1}(X_P (w_1) \bigcap \gamma_1 C_P)$,  
$\;{g_1}^{-1}g_2 \in \pi^{-1}(X_P (w_2) \bigcap \gamma_2 C_P)$, etc., where $\pi: \G \to \G/\PP$.
The map $A_\gamma \to \w{A}_{\w{\gamma}}$ admits an evident  inverse and is thus  an isomorphism.
Finally, every point in $(g_1, \ldots g_r) \in X_\PP (w_\bullet)$ can be written
in the form $g_i=\prod_{j=1}^i \gamma_j$, with $\gamma_i \in X_\PP (w_i)$ for every $1\leq i\leq r$
(just take  $\gamma_i := g_{i-1}^{-1} g_i$, with $g_0:=1$).
It follows, that the $\w{A}_{\w{\gamma}}$, with $\gamma \in \prod_{i=1}^r X_\PP(w_i)$,  cover $X_\PP (w_\bullet)$
\end{proof}

\begin{lm}\la{zlt1} 
The first projection defines a map  ${\rm pr}_1 :X_\PP (w_\bullet) \to X_\PP (w_1)$ which is a Zariski locally trivial
bundle 
over the base with fiber $X_\PP(w_2, \ldots, w_r)$.
\end{lm}
\begin{proof}
We trivialize over the intersection $X_\PP(w_1) \cap \gamma \m{C}_\PP$ with   a big cell centered at
 a fixed arbitrary point of  $X_\PP(w_1)$:
 denote by $(P_2, \ldots, P_r)$  the points in $X_\PP (w_2, \ldots, w_r)$ and define
 the trivialization of the map ${\rm pr}_1$   over $X_\PP(w_1) \cap \gamma \m{C}_\PP$ by the assignment:
 \[
 (\gamma_1 u_1, (P_2, \ldots, P_r) ) \longmapsto (\gamma_1 u_1 {\PP}, 
 \gamma_1 u_1 P_2, \ldots, \gamma_1 u_1 P_r).
 \]
 \end{proof}

 \begin{cor}\la{vst}
 The twisted product varieties $X_\PP (w_\bullet)$ are geometrically integral, normal, projective $k$-varieties.
  \end{cor}
 \begin{proof}

 The twisted product $X_\PP(w_\bullet)$ is a $k$-scheme of finite type. This can be easily  proved by induction on $r$, using the fact that ${\rm pr}_1 : X_\PP(w_\bullet) \rightarrow X_\PP(w_1)$ is surjective, Zariski locally trivial, and has fibers isomorphic to $X_\PP(w_2, \dots, w_{r})$. 

Next, we prove that $X_\PP(w_\bullet)$ is geometrically irreducible. We may replace the field of definition $k$ with its algebraic closure. We argue as before by induction on $r$.  The Zariski locally trivial  bundle map  ${\rm pr}_1$  above is open. Then
irreducibility follows from the fact that any open morphism of schemes with irreducible image and irreducible fibers, has irreducible domain.

Since $X_\PP(w_\bullet)$ is given its reduced structure, it is geometrically integral. As it is closed inside the ind-scheme $(\G/\PP)^r$, Lemma \ref{proj} implies that it is $k$-projective.

Finally, the normality can be checked Zariski locally, hence follows from the normality of each Schubert variety $X_\PP(w_i)$ (Proposition \ref{fqco}) thanks to Lemma \ref{loc_triv}.
\end{proof}

 \subsection{Geometric $\PP$-Demazure product  on $\bwp{\PP}{\PP}$}\la{gpdem} $\;$
 
 \begin{defi}\la{dstp}
{\rm  Define  the {\em geometric $\PP$-Demazure product} as the  binary operation
 \beq\la{r1}
 \star = \star_\PP: 
 {\bwp{\PP}{\PP}} \times \, \bwp{\PP}{\PP} \to \, \bwp{\PP}{\PP} \quad (w_1 ,w_2) \mapsto w_1\star w_2,
 \eeq
 by means of the defining equality
 \beq\la{r1.5}
 X_\PP (w_1 \star w_2) := \im \, \left\{ X_\PP(w_1, w_2) \stackrel{pr_2}\to \G/\PP\right\},
 \eeq
 the point being that, since  the image is irreducible, closed and  $\PP$-invariant, then it is 
 the closure of a $\PP$-orbit.
 More generally, given $w_\bullet \in ({\bwp{\PP}{\PP}})^r$, we define
 \beq\la{r1bis}
 X_\PP (w_{\star}):= X_\PP (w_1 \star \cdots \star w_r) := \im \, \left\{ X_\PP(w_\bullet) \stackrel{pr_r}\to \G/\PP\right\}.
 \eeq
 The resulting surjective  and $k$-projective map 
 \beq\la{e56}
 p: X_\PP (w_\bullet) \to X_\PP (w_\star)
 \eeq
 is called {\em the convolution morphism associated with  $w_\bullet \in ({\bwp{\PP}{\PP}})^r.$}
 }
 \end{defi}

 \begin{rmk}\la{4c4}
 We have an inclusion $X_\B (w_1w_2) \subseteq X_\B (w_1 \star w_2)$, which in general is strict.
 \end{rmk}

\begin{rmk}\la{conn}
{\rm ({\bf  Relation geometric/standard Demazure product})}
In \S\ref{Dem_comp_subsec} below, we shall show that the geometric Demazure product can be easily described in terms of the usual notion of Demazure product defined on the group $\mathcal W$.
\end{rmk}

 \begin{rmk}\la{obvious}
 By definition, given $w_\bullet \in ({\bwp{\PP}{\PP}})^r$ and  $P_r \in X_\PP(w_\star)$, there is $(P_1, \ldots, P_r) \in X_\PP(w_\bullet)$ mapping to it. More generally, given $1 \leq s \leq r$, the natural map
 $X_\PP (w_1, \ldots, w_r) \to X_\PP (w_1 \star \cdots \star w_s, w_{s+1} \star \cdots \star  w_r)$ is also surjective:
 given $(P_s = g\PP, P_r)$ in the target, take $(P_1, \ldots , P_{s-1}, g \PP) \in X_\PP (w_1 \star \cdots \star w_s)$ and 
 $ (P_{s+1}, \ldots , P_r) \in X_\PP(w_{s+1} \star \cdots \star  w_r)$. Then, by the invariance of relative position 
 with respect to left multiplication by elements $g \in \G$, we have that  
 $(P_1, \ldots , P_{s-1}, g \PP, g P_{s+1}, \ldots , gP_r) \in X_\PP (w_\star)$ and maps to $(P_s,P_r)$.
 \end{rmk}
 
 \begin{lm}\la{r2}
 $(\bwp{\PP}{\PP} , \star_\PP)$ is an associative monoid with unit the class $1 \in \, \bwp{\PP}{\PP}$.
 \end{lm}
 \begin{proof}
 Fix an arbitrary sequence $1 \leq i_1 < \ldots  < i_m =r$. It is easy to verify, by using Remark \ref{obvious}, 
 that   the natural map 
 \[X_\PP(w_\star) \to
 X_\PP (w_1 \star \cdots \star w_{i_1}, \ldots , w_{i_{m-1}+1} \star \cdots \star w_{i_m=r}),  \qquad
 (P_1, \ldots, P_r) \mapsto (P_{i_1}, \ldots, P_{i_m})
 \] is surjective.
 This implies that
 \[
 w_1 \star \cdots \star w_r = (w_1 \star \cdots \star w_{i_1}) \star (w_{i_1+1} \star \cdots \star w_{i_2}) \star \cdots
 \star (w_{i_{m-1}+1} \star \cdots \star w_{i_m=r}),
 \]
and in particular $w_1 \star \cdots \star w_r$ is the $r$-fold extension of an associative product $w_1 \star w_2$
which, as it is immediate to verify, has the properties stated in the lemma.
 \end{proof}

 In general, we  have  the following inequality  in  the poset $(\bwp{\m{Q}}{\m{Q}}, \leq_\m{Q})$
 \beq\la{4vbo}(w_1 \star_\PP \cdots\star_\PP w_r)'' \leq_\m{Q} w''_1 \star_\m{Q} \ldots
\star_\m{Q} w''_r,
\eeq 
more precisely, if we set $w_{\star \PP}: =w_1 \star_\PP \cdots\star_\PP w_r$  and 
$w''_{\star \m{Q}}: =w_1'' \star_\m{Q} \cdots\star_\m{Q} w''_r$, then we have the following inclusions, each of which
may be strict (cfr. (\ref{4es}))
\beq\la{3cv5}
\ov{\PP w_{\star \PP} \m{Q}/\m{Q}}  \subseteq \ov{\m{Q} w_{\star \PP} \m{Q}/\m{Q}}=
\ov{\m{Q} w_{\star \PP}'' \m{Q}/\m{Q}} \subseteq \ov{\m{Q} w''_{\star \m{Q}} \m{Q}/\m{Q}}.
\eeq
 The inequality (\ref{4vbo})  also  follows immediately from the comparison with the standard Demazure product in \ref{Dem_comp_subsec}.

The notion of $\m{Q}$-type (cfr. Definition \ref{qpma}) is related to the potential strictness of the inclusions in (\ref{3cv5}).

\begin{pr}\la{01012}
Assume that $w_\bullet \in (\bwp{\PP}{\PP})^r$ is of $\m{Q}$-type (resp., $\m{Q}$-maximal), i.e.  that $w_i$  is of $
\m{Q}$-type
(resp. $\m{Q}$-maximal), $\forall i$. Then 
we have:
\ben
\item[i)]
  $(w_1 \star_\PP \cdots\star_\PP w_r)'' = w''_1 \star_\m{Q} \ldots
\star_\m{Q} w''_r$; 
\item[ii)] $w_1 \star_\PP \ldots
\star_\PP w_r$ is of $\m{Q}$-type   (resp., $\m{Q}$-maximal).
\een
\end{pr}
\begin{proof}
We have the commutative diagram
\beq\la{dcf}
\xymatrix{
X_\PP(w_\bullet) \ar@{->>}[d]^{a_\bullet}  \ar@{->>}[r]  & X_\PP (w_1 \star_\PP \cdots \star_\PP w_r) \ar[d]^a & \subseteq 
\G/\PP \;\;\mbox{($r$-th copy)} \ar[d] \\
X_\m{Q}(w''_\bullet)  \ar@{->>}[r]  & X_\m{Q} (w''_1 \star_\m{Q} \cdots \star_\m{Q} w''_r) & \subseteq \G/\m{Q} \;\;\mbox{($r$-th copy)},
}
\eeq
where the horizontal convolution morphisms are  surjective by  their very definition.

 We claim that $a_\bullet$ is surjective.
Let $(g_1\m{Q}, \ldots ,g_r \m{Q}) \in X_{\m{Q}} (w''_{\bullet}).$ Then, having set for convenience $g_0:=1$, we have
$g_{i-1}^{-1} g_i \in \ov{\m{Q} w_i\m{Q}}.$ The assumption that the $w_i$ are of $\m{Q}$-type, implies that $X_\PP (w_i) 
\to X_\m{Q} (w''_i)$ is surjective, so that, for every $i$ there is $q_i \in \m{Q}$ such that
$g_{i-1}^{-1} g_i  q_i \in  \ov{\PP w_i \PP}$, which, again by $w_i$ being of $\m{Q}$-type, equals $\ov{\m{Q}w_i \PP}.$
Clearly, the $r$-tuple $(g_1q_1 \PP, \ldots g_rq_r \PP)$ maps to $(g_1\m{Q}, \ldots ,g_r \m{Q}).$ In order to establish the surjectivity of $a_\bullet$, it remains to show that
$(g_1q_1 \PP, \ldots , g_rq_r \PP) \in X_\PP (w_\bullet)$, i.e. that $(g_{i-1} q_{i-1})^{-1} g_i q_i 
\in \ov{\PP w_i \PP}$. This latter equals  $ q_{i-1}^{-1}  (g_{i-1}^{-1} g_i q_i) \in q_{i-1}^{-1} \ov{\PP w_i \PP}\subseteq 
  \ov{\m{Q} w_i \PP}=  \ov{\PP w_i \PP}.$

Given the commutative diagram (\ref{dcf}), we have that the map  $a$  is  surjective as well, which
yields  the desired equality i). 

In order to prove the statement ii) in the $\m{Q}$-type case, we need to prove
that the image  of the top horizontal arrow  is $\m{Q}$-invariant.  
This follows immediately once we note that the source of the arrow is $\m{Q}$-invariant
 for the left-multiplication  diagonal action on $(\G/\PP)^r$  and that  the $r$-th projection map  onto $\G/\PP$ is $\m{Q}$-invariant.
 
  In order to prove the statement ii) in the $\m{Q}$-maximal case,
we need to show that the domain of $a$ is the full pre-image under $\G/\PP \to \G/ \m{Q}$ of
the target of $a$. 
For simplicity, set $w_\star:= w_1 \star_\PP \cdots \star_\PP w_r$ and 
set $w''_\star:= w''_1 \star_\m{Q} \cdots \star_\m{Q} w''_r$.
Let $x =g \m{Q} \in X_\m{Q} (w''_\star)$.
By Remark \ref{obvious}, we there is $(g_1 \m{Q}, \ldots, g_{r-1} \m{Q}, g \m{Q}) \in X_\m{Q} (w''_\bullet).$
By the surjectivity of $a_\bullet$  observed above ($\m{Q}$-maximal implies $\m{Q}$-type), there are
$q_i \in \m{Q}$ such that
$(g_1q_1 \PP, \ldots, g_{r-1}q_{r-1} \PP, gq_r \PP) \in X_\PP (w_\bullet)$
maps to $(g_1 \m{Q}, \ldots, g_{r-1} \m{Q}, g \m{Q})$.
Since we are assuming $\m{Q}$-maximality, i.e. that $\ov{\PP w_i \PP}=\ov{\m{Q} w_i \m{Q}}$ for every $1\leq i \leq r$,
we see that 
for every $q \in \m{Q}$, we have that $(g_1q_1q \PP, \ldots, g_{r-1}q_{r-1}q \PP, gq_r q \PP) \in X_\PP (w_\bullet)$.
As $q$ varies in $\m{Q}$, $gq_r q \PP$ traces the full pre-image of $g \m{Q}$ under $\G/\PP \to \G/ \m{Q}$.
\end{proof}

\begin{rmk}\la{4vo0} 
The inequality \textup{(}\ref{4vbo}\textup{)},  Lemma \ref{r2} and Proposition \ref{01012}  also  follow immediately from the comparison with the standard Demazure product in \ref{Dem_comp_subsec}.
 \end{rmk}

\subsection{Comparison of geometric and standard Demazure products} \label{Dem_comp_subsec}$\;$

In what follows, we shall use, sometimes without mention, a standard lemma about the Bruhat order (see e.g.\,\cite[Prop.\,5.9]{Hum2}).

\begin{lm} \la{trick} Let $(W,S)$ be a Coxeter group and $x,y \in W$ and $s \in S$. Then $x \leq y$ implies $x \leq ys$ or $xs \leq ys$ \textup{(}or both\textup{)}.  
\end{lm}

\medskip

In this subsection we will explain how the geometric Demazure product is expressed in terms of the usual Demazure product (indeed we will show they are basically the same thing).  Recall that the usual Demazure product is defined on any quasi-Coxeter system $(\W, \mathcal S)$.  For $w_1, \dots, w_r \in \W$, we will denote this product by $w_1 * \cdots * w_r \in \W$.  Its precise definition can be given neatly using the $0$-Hecke algebra, as follows. Associated to $(\mathcal W, \mathcal S)$ we have the (affine) Hecke algebra $\mathcal H = \mathcal H(\mathcal W, \mathcal S)$ which is an associative $\mathbb Z[v, v^{-1}]$-algebra with generators $T_w$, $w \in \mathcal W$, 
and relations
\begin{align*}
T_{w_1} T_{w_2} &= T_{w_1 w_2},  \,\,\,\, \mbox{if $\ell(w_1 w_2) = \ell(w_1) + \ell(w_2)$} \\
T^2_s &= (v^2-1) T_s + v^2 T_1, \,\,\,\, \mbox{if $s \in \mathcal S$}.
\end{align*}
The 0-Hecke algebra $\mathcal H_0$ is defined by taking the $\mathbb Z[v]$-algebra generated by the symbols $T_w, \,\,\, w \in \mathcal W$, subject to the same relations as above, and then specializing $v = 0$.  

We define the Demazure product $x * y \in \mathcal W$ for $x, y \in \mathcal W$ as follows: set $T'_x = (-1)^{\ell(x)}T_x$ as elements in $\mathcal H_0$, and note that in that algebra we have that 
$$
T'_x T'_y = T'_{x*y}
$$
for a certain element $x*y \in \mathcal W$ (see Remark \ref{Dem_def_rem} below). Clearly $(\mathcal W, *)$ is  a monoid (since $\mathcal H_0$ is associative). It follows from the definitions that for $w \in \mathcal W$ and $s \in \mathcal S$, we have
\begin{equation} \label{Dem_orig_defn}
w * s = {\rm max}(w, ws),
\end{equation}
where the maximum is taken relative to the Bruhat order on $\mathcal W$. 

\begin{rmk} \label{Dem_def_rem}
Sometimes \textup{(}\ref{Dem_orig_defn}\textup{)} is taken as the definition of the Demazure product $w*s$, and then one has the challenge of showing this can be extended uniquely to a monoid product $\W \times \W \to \W$. With the $0$-Hecke algebra definition, this challenge is simply avoided, and the only exercise one does is  to verify that the element $T'_x T'_y \in \mathcal H_0$ is supported on a single element, which we define to be $x*y$. One does that simple exercise using induction on the length of $y$.
\end{rmk}

In considering $X_{\PP}(w_\bullet)$, we are free to represent each element $w_i \in \, _{\PP}\W_{\PP}$ by any lift in $\W$.  We shall use the same symbol $w_i$ to denote both an element in $\W$  and its image in $_{\PP}\W_{\PP}$. 

Recall $^{\PP}\W^{\PP}$ denotes the set of elements $w \in \W$ such that $w$ is the unique {\em maximal} length element in $\W_{\PP} w \W_{\PP}$. For $w_\bullet \in (\, ^{\PP}\W^{\PP})^r$, note that $(\mathcal G/\B)^r \to (\mathcal G/\PP)^r$ induces a {\em surjective} morphism
\begin{equation} \label{surjective}
X_{\B}(w_\bullet) \twoheadrightarrow X_{\PP}(w_\bullet).
\end{equation}

\begin{pr} \label{im_of_p} Suppose that $w_i \in \, ^{\PP}\W^{\PP}$ for all $i=1,\dots, r$.  Then the geometric Demazure product $w_\star = w_1 \star \cdots \star w_r$ is the image of the Demazure product $w_* = w_1 * \cdots * w_r$ under the 
natural quotient map $\W \to \, _{\PP}\W_{\PP}$.
\end{pr}

\begin{proof} Combining Lemma \ref{demazmax} below with (\ref{surjective}), we easily see that it is enough to prove the proposition in the case $\mathcal P = \mathcal B$. Indeed, the lemma implies $w_* \in \, ^{\PP}\W^{\PP}$, and then using $X_{\B}(w_*) \twoheadrightarrow X_{\PP}(w_*)$ (cf.\,(\ref{surjective})) we would get a commutative diagram
$$
\xymatrix{
X_{\B}(w_\bullet) \ar@{->>}[d]  \ar@{->>}[r] & X_{\B}(w_*) \ar@{->>}[d] \\
X_{\PP}(w_\bullet) \ar[r] & X_{\PP}(w_*)}
$$
whose top arrow is surjective by the $\PP = \B$ case of the proposition. Actually, a priori we do not know the bottom arrow actually exists. Instead we only know we have a diagram
$$
\xymatrix{
X_{\B}(w_\bullet) \ar@{->>}[d]  \ar@{->>}[r] & X_{\B}(w_*) \ar[d] \\
X_{\PP}(w_\bullet) \ar[r] & \mathcal{G}/\PP.}
$$
But the left arrow of this diagram is surjective and the right arrow has image $X_{\PP}(w_*)$, by \eqref{surjective}. So the image of the bottom arrow is $X_{\PP}(w_*)$. It follows that the previous diagram exists, and that $X_{\PP}(w_\bullet) \to X_{\PP}(w_*)$ is surjective.

So, let us prove the proposition in the case $\PP = \B$; note there is no longer any hypothesis on the elements $w_i \in \W$. Write reduced expressions $w_i = s_{i1} \cdots s_{i k_i}$ for $s_{ij} \in \mathcal S$, for each $i$.  Clearly
\begin{equation} \label{Dem_prod}
w_1 * w_2 * \cdots * w_r = (s_{11}* \cdots * s_{1k_1}) * (s_{21} * \cdots * s_{2k_2}) * \cdots * (s_{r1} * \cdots * s_{r k_r}).
\end{equation}
Let $s_{\bullet \bullet} = (s_{11}, \dots, s_{r k_r})$. There is a commutative diagram
$$
\xymatrix{
X_\mathcal B (s_{\bullet \bullet}) \ar@{->>}[r] \ar[dr] & X_\mathcal B(w_\bullet)  \ar[d] \\
& \mathcal G/\mathcal B}
$$
where the horizontal arrow forgets the elements in the tuple except those indexed by $i k_i$.
Therefore, by (\ref{Dem_prod}), it is enough to replace $s_{\bullet \bullet}$ with an arbitrary sequence $s_\bullet = (s_1, \dots, s_k)$, set $s_* = s_1 * \cdots * s_k$, and show that the image of the morphism
\begin{align*}
p_k: X_{\B}(s_\bullet) &\longrightarrow \mathcal G/\mathcal B \\
(\mathcal B_1, \dots, \mathcal B_k) &\longmapsto \mathcal B_k
\end{align*}
is precisely $X_{\B}(s_*)$.  We will prove this by induction on $k$. Let $s'_\bullet = (s_1, \dots, s_{k-1})$ and $s'_* = s_1 * \cdots * s_{k-1}$. By induction, the image of 
\begin{align*}
p_{k-1}: X_{\B}(s'_\bullet) &\longrightarrow \mathcal G/\mathcal B \\
(\mathcal B_1, \dots, \mathcal B_{k-1}) &\longmapsto \mathcal B_{k-1}
\end{align*}
 is precisely $X_{\B}(s'_*)$.

First we claim the image of $p_k$ is contained in $X_{\B}(s_*)$.  Suppose $(\mathcal B_1, \dots, \mathcal B_{k-1}, \mathcal B_k) \in X_{\B}(s_\bullet)$. By induction we have
$$
\mathcal B \, \overset{\leq s'_*}{\textendash \textendash} \, \mathcal B_{k-1} \, \overset{\leq s_k}{\textendash \textendash} \, \mathcal B_k.
$$
If $\mathcal B_k = \mathcal B_{k-1}$, then $\mathcal B_k \in X_{\B}(s_*') \subseteq X_{\B}(s_*)$, the inclusion holding since
\begin{equation} \label{Dem_fact}
s_* = {\rm max}(s'_*, s'_* s_k).
\end{equation}
If $\mathcal B_k \neq \mathcal B_{k-1}$, then $\mathcal B \,\overset{v}{\textendash \textendash} \,\mathcal B_{k-1} \,\overset{s_k}{\textendash \textendash} \, \mathcal B_k$ for some $v \leq s'_*$. Thus $\mathcal B \, \overset{u}{\textendash \textendash} \, \mathcal B_k$ for $u \in \{ v, vs_k  \}$. Note we are implicitly using the BN-pair relations (\ref{BN-pair_eq}) here.) On the other hand, $v \leq s'_*$ implies $vs_k \leq s'_*$ or $vs_k \leq s'_* s_k$ (Lemma \ref{trick}), so by (\ref{Dem_fact}), we have both $v \leq s_*$ and $vs_k \leq s_*$. This implies that $\mathcal B_k \in X_{\B}(s_*)$.

Conversely, assume $\mathcal B_k \in X_{\B}(s_*)$; we need to show that $\mathcal B_k \in {\rm Im}(p_k)$.  We have $\mathcal B \, \overset{v}{\textendash \textendash} \, \mathcal B_k$ for some $v \leq s_*$. 

 If $v \leq s'_*$, then, by induction, $\mathcal B_k =: \mathcal B_{k-1} = p_{k-1}(\mathcal B_1, \dots, \mathcal B_{k-1})$ for some $(\mathcal B_1, \dots, \mathcal B_{k-1}) \in X_{\B}(s'_\bullet)$. But then $\mathcal B_k = p_k(\mathcal B_1, \dots, \mathcal B_{k-1}, \mathcal B_{k-1}) \in p_k(X_{\B}(s_\bullet))$. 

If $v \nleq s'_*$, then $vs_k \leq s'_*$ and $vs_k < v$. Then there exists $\mathcal B_{k-1} \in Y_{\B}(vs_{k}) \subset X_{\B}(s'_*)$ with $\mathcal B_{k-1} \, \overset{s_k}{\textendash \textendash} \, \mathcal B_k$. By induction $\mathcal B_{k-1} = p_{k-1}(\mathcal B_1, \dots, \mathcal B_{k-1})$ for some $(\mathcal B_1, \dots, \mathcal B_{k-1}) \in X_{\B}(s'_\bullet)$ and we see $\mathcal B_k = p_k(\mathcal B_1, \dots, \mathcal B_{k-1}, \mathcal B_k) \in p_k(X_{\B}(s_\bullet))$.
\end{proof}

We conclude this subsection with the following lemma, a special case of which was used in the proposition above. Let $^\PP\W$ (resp.\, $\W^{\PP}$) be the set of $w \in \W$ which are the unique maximal elements in their cosets $\W_{\PP} w$ (resp. $w \W_\PP$). It is a standard fact that $\W^\PP = \{ w \in \W ~ | ~ ws < w, \,\, \forall s \in \mathcal S \cap \W_P \}$, and similarly for $\, ^\Q\W$ and $\,^\Q\W^\PP$. Thus $^\Q \W^\PP = \, ^\Q \W \cap \W^\PP$. The reader should compare the statement below with Lemma \ref{r2} and Proposition \ref{01012}.

\begin{lm}\la{demazmax}
Let $\Q$ and $\PP$ be parahoric subgroups, with no relation to each other. If $w_1  \in \, ^{\Q}\W$, and $w_2 \in \, \W^{\PP}$, then $w_1 * w_2 \in \, ^{\Q}\W^{\PP}$. In particular, the Demazure
product defines an associative product $\pwp \times \, \pwp \to \,\pwp$. 
\end{lm}

\begin{proof}
Let $s \in \W_{\PP}$ be a simple reflection. It is enough to prove that $(w_1 * w_2)s < (w_1 * w_2)$ (the same argument will also give us $s(w_1 * w_2) < (w_1*w_2)$ when $s \in \W_\Q$).  Recall that $x*s = {\rm max}(x, xs)$. Using that $*$ is associative, we compute
\begin{align*}
(w_1 * w_2) * s &= w_1 * (w_2 * s) \\
&= w_1 * w_2.
\end{align*}
We are using $w_2 s < w_2$ to justify $w_2 *s = w_2$.  But then we see 
$$(w_1 * w_2) = {\rm max}((w_1*w_2), (w_1 * w_2)s),$$ 
and we are done. \end{proof}

\subsection{Connectedness of fibers of convolution morphisms}

Before proving the connectedness, we need a few definitions and lemmas. Let $f: X \twoheadrightarrow Y$ be a finite surjective morphism between integral varieties over a field of characteristic $p$. 

\begin{defi}
We say $f$ is {\em separable} if the field extension $K(X)/K(Y)$ is separable.  We say $f$ is {\em purely inseparable} \textup{(}or {\em radicial}\textup{)} if $f$ is injective on topological spaces, and if for every $x \in X$ the field extension $k(x)/k(f(x))$ is purely inseparable.
\end{defi}


Radicial morphisms are defined in  \ci[Def.\,3.5.4]{EGAI};  we have adopted 
the equivalent re-formulation given by \ci[Prop.\,3.5.8]{EGAI}.  
A radicial morphism is universally injective by \ci[Rem.\,3.5.11]{EGAI}. Recall that a morphism is a universal homeomorphism if and only if it is integral, surjective and radicial; see \ci[Prop.\,2.4.4]{EGAIV2}.
Since finite morphisms are automatically integral, we see that a finite, surjective and radicial morphism is a universal homeomorphism.

The following lemma is trivial in characteristic zero, for then every morphism is separable.

\begin{lm} \la{factorize}
Lef $f : X \twoheadrightarrow Y$ be a finite surjective morphism between integral varieties over a field of characteristic $p$. Then $f$ factors as $X \overset{i}{\rightarrow} Y' \overset{s}{\rightarrow} Y$, where $i, s$ are finite surjective, $i$ is radicial \textup{(}hence a universal homeomorphism\textup{)}, and $s$ is separable.  Moreover, generically over the target, $s$ is \'{e}tale.
\end{lm}


\begin{proof}
First, assume $Y$ is affine.  We write $Y = {\rm Spec}(A)$ for an integral domain $A$. As $f$ is finite, $X = {\rm Spec}(B)$ where $B$ is an $A$-finite integral domain.

Let $K(A) \subseteq K(B)$ be the inclusion of fraction fields of $A,\,B$. Let $K^s$ be the maximal separable subextension, and let $A^s$ be the integral closure of $A$ in $K^s$. Define $A' := A^s \cap B$. We have $A \subseteq A' \subseteq B$, and the map $f$ factors as the composition of $i: {\rm Spec}(B) \to {\rm Spec}(A')$ and $s: {\rm Spec}(A') \twoheadrightarrow {\rm Spec}(A)$. Note that as $B$ is $A$-finite and $A$ is Noetherian, $A'$ is also $A$-finite.

\n
{\bf Claim:} {\em $i$ is radicial, and $s$ is separable}.

First, we prove that $K(A') = K^s$, which will prove the morphism $s$ is separable, since $K^s/K(A)$ is separable. It is easy to see that $K(A^s) = K^s$, using the fact that $A^s$ is the integral closure of $A$ in $K^s$. Let $b \in B$ be a nonzero element chosen so that the localization $B_b$ is normal.  The element $b$ satisfies a minimal monic polynomial with coefficients in $A$; let $a \in A$ be its constant term.  Thus $B_a$ is normal.  Now $A^s_a$ (resp.\,$B_a$) is the integral closure of $A_a$ in $K^s$ (resp.\,$K(B)$), so that $A_a \subseteq A^s_a \subseteq B_a$. Hence $A'_a = A^s_a \cap B_a = A^s_a$, which implies $K(A') = K^s$. We have used here that taking finite intersections and integral closures commutes with localization.


Next, we prove that $i$ is radicial. Since $K(B)/K^s$ is a finite, purely inseparable extension of characteristic $p$ fields, there is a
positive power  $p^n$ such that $b^{p^n} \in K^s$ for all $b \in K(B)$. Therefore $b^{p^n} \in A'$ for all $b \in B$. Now let $x=P, \, Q \in {\rm Spec}(B)$ lie over $i(x) = p' \in {\rm Spec}(A')$. We have $b \in P \Leftrightarrow b^{p^n} \in p' \Leftrightarrow b \in Q$, which shows $i$ is injective. The extension $k(x) \supset k(i(x))$ is the extension ${\rm Frac}(B/P) \supset {\rm Frac}(A'/p')$.  We have $b^{p^n} \in A'/p'$ for every $b \in B/P$, and this shows that $k(x)/k(i(x))$ is purely inseparable, and the claim is proved.

Now, we turn to the generic \'{e}taleness of $s$. Shrinking the target of $s$, we may assume $s$ is flat (cf.\,e.g.\,\cite[Cor.\,10.85]{GW}). The $A'$-module $\Omega_{A'/A}$ is finitely generated as an $A$-module, and $\Omega_{A'/A} \otimes_{A} K(A) = \Omega_{K(A')/K(A)} = 0$, the last equality holding since $K(A')/K(A)$ is finite separable. Thus there is a non-zero $a \in A$ such that the localization $(\Omega_{A'/A})_{a} = 0$.  Therefore $\Omega_{A'_{a}/A_a} = 0$, and over the complement of the divisor $a = 0$, we see that $s$ is \'{e}tale.  The lemma is proved in the case when $Y$ is affine.

To prove the general case, we need to show that the construction  $(A \to B) \leadsto (A\to A' \to B)$
``glues;" for this it is enough to prove it is compatible with restriction to smaller open affine subsets ${\rm Spec}(A_U) \subset {\rm Spec}(A)$. Using the corresponding homomorphism $A \rightarrow A_U$, we obtain a homomorphism $\phi: A' \otimes_A A_U \rightarrow (A_U)'$, and we need to show this is an isomorphism. By covering ${\rm Spec}(A_U)$ with principal open subsets of ${\rm Spec}(A)$ we are reduced to proving
that $\phi$ is an isomorphism in the special case $A_{U}= A_a,$ with $a \in A \setminus \{0\}.$
But then $\phi$ is just the isomorphism $A' \otimes_A A_a \overset{\sim}{\rightarrow}  (A')_a = (A_a)'$. 
\end{proof}

\begin{rmk}\la{nmz}
If in Lemma {\rm \ref{factorize}} we assume that $X$ is normal, then there is a {\em unique} radicial/separable factorization $f:X\to Y'\to Y$ with the requirement that $Y'$ is normal. 
\end{rmk}

The following proposition establishes a quite general principle of connectedness. As pointed out to us by Jason Starr,
the proposition also admits a proof via the use of \ci[Cor.\,4.3.7]{EGAIII1}. We are also very grateful to Jason Starr
for providing us with an alternative proof of Lemma \ref{factorize} (omitted).

\begin{pr}  \la{ext_conn}
Let $p : X \rightarrow Y$ be a surjective proper morphism of  integral varieties over an algebraically closed field.
Assume that $Y$ is normal. If $p^{-1}(y)$ is connected for all points $y$ in a dense open subset $V \subseteq Y$, then $p^{-1}(y)$ is connected for all $y \in Y$. In particular, if $p : p^{-1}(V) \rightarrow V$ is isomorphic to ${\rm pr}_1: V \times p^{-1}(y) \rightarrow V$ for all $y \in V$, then $p^{-1}(y)$ is connected for all $y \in Y$.
\end{pr}

 \begin{proof} 
First assume that the fibers $p^{-1}(y)$ for $y \in V$ are connected. By the Stein factorization theorem, we may factor $p$ as the composition $X \overset{p_c}{\rightarrow} \hat{X} \overset{p_f}{\rightarrow} Y$,  where $p_f, \, p_c$ are proper and surjective, the fibers of $p_c$ are connected, and $p_f$ is finite. Since $p_{c,*}\mathcal O_X = \mathcal O_{\hat{X}}$, the scheme $\hat{X}$ is an integral variety over $\bar{k}$. Let $\hat{X} \overset{i}{\rightarrow} Y' \overset{s}{\rightarrow} Y$ be the factorization $p_f = s \circ i$ from Lemma \ref{factorize}, so that $K(Y')/K(Y)$ is the maximal separable subextension of $K(\hat{X})/K(Y)$.

The surjective morphism $s: Y' \rightarrow Y$ is finite separable, hence finite \'{e}tale generically over the target (cf.\,end of Lemma \ref{factorize}), and it follows that generically its fibers have some finite cardinality $n:= [K(Y')\!:\! K(Y)]$. Our assumption on the fibers of $p$ forces $n = 1$. Since $Y$ is normal, the finite birational morphism $s:Y' \rightarrow Y$ must be an isomorphism, and so the fibers of $s$ are singletons. Therefore the fibers of $p = p_c \circ i \circ s$ are connected as this holds for $p_c \circ i$.

For the second assertion, note that $p^{-1}(V)$ is irreducible since $X$ is, and so the triviality assumption forces $p^{-1}(y)$ to be irreducible for $y \in V$. Then the first part implies that all fibers of $p$ are connected.
\end{proof}


\begin{cor}\la{fibco} 
The fibers of the convolution morphism $p: X_\mathcal P(w_\bullet) \rightarrow X_\mathcal P(w_\star)$ are  geometrically connected.
\end{cor}

 We remark that a somewhat stronger result is proved cohomologically in Theorem \ref{dtm}, but here we give a direct geometric proof.

\begin{proof}  We pass to the fixed  algebraic closure of our finite field; for simplicity,
we do not alter the notation. By Corollary \ref{vst}, $p$ is proper, surjective, and the source and target of $p$ are normal and integral. Hence by Proposition \ref{ext_conn}, it is sufficient to show that $p$ is trivial over an open dense subset of $Y_{\PP}(w_\star)$, in the sense of the second assertion of that proposition. We may represent $w_i$ by an element in $\, ^\PP \W ^\PP$, so that $w_\star$ is represented by $w_* \in \,^\mathcal P \mathcal W^\mathcal P$ (Lemma \ref{demazmax}), and so $Y_\mathcal P(w_\star)$ contains $\mathcal U w_*\mathcal P/\mathcal P$ as an open subset.  We have
$$
\mathcal U w_* \mathcal P/\mathcal P \, = \, (\mathcal U \cap \, ^{w_*}\overline{\mathcal U}_\mathcal P) w_* \mathcal P/\mathcal P \, \cong \, \mathcal U \cap \, ^{w_*}\overline{\mathcal U}_\mathcal P ,
$$ 
by \eqref{decomp1}. Since $p$ is $\B$-equivariant, it is clearly trivial over this subset in the sense of Proposition \ref{ext_conn}. More precisely, an element $P \in \mathcal U w_* \PP/\PP$ can be written in the form 
$$
P = \, uw_* \PP/\PP
$$
for a unique element $u \in \mathcal U \cap \, ^{w_*} \overline{\mathcal U}_\PP$. We can then define an isomorphism
$$
p^{-1}(\mathcal Uw_*\PP/\PP) ~ \widetilde{\rightarrow} ~ \, p^{-1}( w_* \PP/\PP) \times \mathcal Uw_*\PP/\PP
$$
by sending $(P_1, \dots, P_{r-1},\, uw_* \PP/\PP)$ to $(\, u^{-1} P_1, \cdots, u^{-1} P_{r-1}, \, w_*\mathcal P/\PP) \times \, u w_*\PP/\PP$. 
\end{proof}

\subsection{Generalized convolution morphisms $p: X_\PP(w_\bullet) \to X_\m{Q} (w''_{I,\bullet})$}\la{098}$\;$

Let $1 \leq r' \leq r$ and let  $1\leq i_1 < \ldots <i_m =r'$ and denote these data   by $I$.
Let  $w_\bullet \in (\bwp{\PP}{\PP})^r$, set  $i_0:=0$ and define
\beq\la{g56}
w_{I,k} := w_{i_{k-1}+1} \star_\PP \cdots \star_\PP w_{i_k}, \quad w''_{I,k} := w''_{i_{k-1}+1} \star_\m{Q} \cdots
 \star_\m{Q} w''_{i_k}.
\eeq
\begin{defi}\la{cnvzmpz}
{\rm 
We define 
{\em the convolution morphism $p:X_\PP(w_\bullet) \to X_\m{Q}(w_{I,\bullet})$
associated with $w_\bullet$ and with $I$} by setting
\beq\la{wwi}
 (g_1\PP, \ldots, g_r \PP) \mapsto (g_{i_1}\m{Q}, \ldots, g_{i_m}
\m{Q}).
\eeq

}
\end{defi}

The convolution morphism (\ref{wwi}) factors through  the natural convolution morphism 
$X_\PP(w_1, \ldots, w_r) \to X_\PP(w_1, \ldots, w_{r'})$.
The composition of convolution morphisms is a convolution morphism.
Generalized convolution morphisms are typically not surjective.

We have the commutative diagram of convolution morphisms with surjective horizontal arrows
\beq\la{e431}
\xymatrix{
X_\PP(w_\bullet) \ar[r] \ar[d] &  X_\PP(w_{I,\bullet}) \ar[d] \\
X_\m{Q}(w''_\bullet) \ar[r]  & X_\m{Q}(w''_{I,\bullet}).
}
\eeq

\begin{pr}
\la{gt50}
Let $p:X_\PP(w_\bullet) \to X_\m{Q}(w''_{I,\bullet})$ be a  convolution morphism.
Assume that the $w_i$ are of $\m{Q}$-type, i.e.,\,$X_P(w_i)=QX_P(w_i)$. Then, locally over 
$X_\m{Q}(w''_{I,\bullet})$, the map $p$ is isomorphic to the product
of the maps $p_k: X_\PP(w_{i_{k-1}+1}, \ldots w_{i_k}) \to X_\m{Q}(w''_{I,k})$ and $c: X_\PP(w_{r'+1}, \ldots
 w_{r}) \to \{pt\}.$
\end{pr}

\begin{proof}
The conclusion can phrased by stating the existence   of a cartesian diagram 
\beq\la{vizo}
\xymatrix{
X_\PP(w_\bullet) \ar[d]_p & {X_\PP (w_\bullet)}_{\w{A}} \ar[d]^{p_{\w{A}}}  \ar@{_(^->}[l]_{\,\, open} && \prod_{k=1}^m 
X_\PP( w_{i_{k-1}+1}, \ldots w_{i_k})_{A_k}  \times X_\PP(w_{r'+1}, \ldots w_r)
\ar[d]^{p_A:=\prod {p_k}_{A_k} \times c} \ar[ll]_{\hskip -2.85cm \sim} \\
X_\m{Q}(w''_{I,\bullet})  &     \w{A}  \ar@{_(^->}[l]_{\,\,\,\, open}  && A:= \prod_{k=1}^m A_k \times \{pt\} \ar[ll]_{\hspace{-.3in} \sim},
}
\eeq
where, for each $k$, the open subset   $A_k \subseteq X_\m{Q} (w''_{I,k})$ is the analogue of the  $A_\gamma$ appearing  in the proof of Lemma \ref{loc_triv} (we are now dropping $\gamma$ from the notation), and where the isomorphism $A\cong \w{A}$ is given 
explicitly by the assignment $\{\gamma_i u_i \m{Q}\}_{i=1}^r \mapsto \{ \prod_{j=1}^i \gamma_j u_j  \m{Q}\}_{i=1}^r.$

Our task is to provide the isomorphism on the top row of (\ref{vizo}).
The assignment is as follows:
\beq\la{ert5}
\left(
\left\{
\left(
T_{i_{k-1}+1}, \ldots, T_{i_k} =\gamma_ku_k q_k \PP
\right)     
\right\}_{k=1}^m   ,
\left(
T_{r'+1}, \ldots, T_r
\right)
\right)
\eeq
\vspace{.1in}
\[
\mbox{maps to}
\vspace{.1in}
\]
\[
\left(
\left\{ \prod_{j=1}^{k-1} \gamma_ju_j
\left(
T_{i_{k-1}+1}, \ldots, T_{i_k} =\gamma_ku_k q_k \PP
\right)   
\right\}_{k=1}^m   , 
\left(\prod_{j=1}^m \gamma_j u_j\right) q_m
\left(
T_{r'+1}, \ldots, T_r
\right)
\right) 
\]
which does the job: the verification  of this can be done by the reader  with the aid of the following list of items to
be considered and/or verified
\ben
\item We use the local isomorphisms 
\beq\la{azzr}
\{\gamma_j u_k \m{Q}\}_{k=1}^m \mapsto
\left\{\prod_{j=1}^k \gamma_j u_j  \m{Q} \right \}_{k=1}^m
\eeq between the targets of the maps $\prod p_k$ and $p$.

\item
The assignment (\ref{ert5})  should agree with the local isomorphisms
(\ref{azzr}).

\item
The $\m{Q}$-type assumption on the $w_i$ ensures, via Proposition \ref{01012}, that the maps
 $X_\PP(w_{I,k}) \to X_\m{Q}(w''_{I,k})$ are surjective.
 
\item
Given $\gamma_iu_i \m{Q}$, the expression $\gamma_ku_k q_k \PP$, with variable $q$, describes a point in the fiber of 
$\G/\PP\to \G/\m{Q}$
 over $\gamma_i u_i \m{Q}$ that, in addition lies in $X_\PP (w_{I,k})$ (this constrains $q_k$), i.e. a point in the fiber 
 over $\gamma_i u_i \m{Q}$ of the  surjective  map
 $X_\PP(w_{I,k}) \to X_\m{Q}(w''_{I,k})$. Note that $q_k$ has ambiguity $q_k p_k$.

\item
Once we have $T_{i_k}$ as above, we use the surjectivity of the convolution morphisms of type $X_\PP(w_\bullet) \to X_\PP(w_{\star_\PP})$ and Remark \ref{obvious}
to infer that we indeed can complete each $T_{i_k}$  with variables $T_{i_{k-1}+1}, \ldots, T_{i_k}$
with the correct set of consecutive relative position, to the left as indicated in the first line of (\ref{ert5}).
Of course, by construction, each $T_{i_k} \mapsto \gamma_k u_k \m{Q}.$

\item
The assignment (\ref{ert5}) is well-defined with values in 
$(\G/\PP)^r$: in fact, the ambiguities $q_kp_k$ do not effect the assignment.

\item
The assignment  (\ref{ert5}) is well-defined with values into $X_\PP(w_\bullet) \subseteq (\G/\PP)^r$: this is where we use that
the $w_{i_{k}+1}$ are of $Q$-type for $2\leq k \leq m-1$; in fact, we need to verify that, if we write
$P_{i_k+1} = g \PP$,  so that $g \in \ov{\PP w_{i_{k}+1}\PP}$, then we also have that
  $q_k^{-1} g \in \ov{\PP w_{i_{k}+1} \PP}$, and this follows from the $Q$-type assumption 
  on $w_{i_{k}+1}.$
  
  \n
  Note that if we replace the expression $\prod_{j=1}^{k-1} \gamma_j u_j$ in (\ref{ert5}) with
  $\prod_{j=1}^{k-1} \gamma_j u_j q_j$, or  even with  $(\prod_{j=1}^{k-1} \gamma_ju_j) q_{k-1}$,
  then what is  above works, but what follows does not.
  
  \item
 It is immediate to  verify that $p$ maps  the expression target of (\ref{ert5}) to the lhs of (\ref{azzr}).
  
  \item
  The assignment (\ref{ert5}), defined over our suitable open subsets, has an evident inverse.
\een
\end{proof}

\begin{rmk}\la{rt41}
As the proof of Proposition \ref{gt50} shows, if we assume that  $r=r'=m,$ i.e. that $p: X_\PP (w_\bullet)
\to X_\m{Q} (w''_\bullet)$ and that   the $w_i$ are $\m{Q}$-maximal,
then the map $p$ is a Zariski locally trivial bundle with smooth  fiber $(\m{Q}/\PP)^r$, in fact the 
elements $q_k$ in part (4) of the proof of Proposition \ref{gt50} are no longer constrained.
\end{rmk}

\subsection{Relation of convolution morphisms to convolutions of perverse sheaves}\la{fupo}$\;$

The twisted products are  
close in spirit to ordinary product varieties (see Lemma \ref{loc_triv}).  
A more  standard  notation for twisted products is
$\xbp{ \PP}{w_1} \tilde{\times} \cdots \tilde{\times} \xbp{ \PP}{w_r}$; we opted
for lighter notation.
The remark that follows clarifies the relation between the convolution morphism
$p: X_\PP (w_\bullet) \to X_\PP (w_\star)$ and the convolution of equivariant shifted-perverse sheaves
$\m{IC}_{X_\PP (w_1)} * \cdots * \m{IC}_{X_\PP (w_1)}.$

\begin{rmk}\la{conv_rmk} {\rm ({\bf  Lusztig's convolution product \cite{Lusz}})}
Let $P_{\PP}(\mathcal G/\mathcal P) \subset D^{b}_c(\mathcal G/\PP, \, \bar{\mathbb Q}_\ell)$ be the full subcategory consisting of $\PP$-equivariant perverse sheaves on the \textup{(}ind-\textup{)}scheme $\mathcal G/\PP$. Lusztig has defined a convolution operation 
$$
* ~ : ~ P_{\PP}(\mathcal G/\mathcal P) \times P_{\PP}(\mathcal G/\mathcal P) ~ \longrightarrow ~ D^b_c(\mathcal G/\PP, \, \bar{\mathbb Q}_\ell)
$$
as follows. There is a twisted product space $\mathcal G \times^{\PP} \mathcal G/\PP$ \textup{(}the quotient of the product with respect to the anti-diagonal action of $\PP$\textup{)} which fits into a diagram of \textup{(}ind-\textup{)}schemes
$$
\xymatrix{
\mathcal G/\mathcal P \times \mathcal G/\PP & \ar[l]_{\hspace{.2in}p_1} \ar[r]^{p_2\hspace{.1in}} \mathcal G \times \mathcal G/\mathcal P & \mathcal G \times^{\PP} \mathcal G/\PP \ar[r]^{\hspace{.2in} m} & \mathcal G/ \PP.}
$$
The morphisms $p_1$ and $p_2$ are the quotient morphisms; both are locally trivial with typical fiber $\PP$. The map $m$ is the ``multiplication'' morphism. Given $\mathcal F_1, \mathcal F_2 \in P_{\PP}(\mathcal G/\mathcal P)$, there exists on the twisted product a unique perverse \textup{(}up to cohomological shift\textup{)} sheaf $\mathcal F_1 \widetilde{\boxtimes}  \mathcal F_2$, such that there is an isomorphism $p^*_1(\mathcal F_1 \boxtimes \mathcal F_2) \cong p^*_2(\mathcal F_1 \widetilde{\boxtimes} \mathcal F_2)$.  Lusztig then defines
$$
\mathcal F_1 * \mathcal F_2 := m_! \left(\mathcal F_1 \widetilde{\boxtimes} \mathcal F_2\right) \, \in 
D^{b}_c(\mathcal G/\PP, \, \bar{\mathbb Q}_\ell).
$$
It is a well-known fact that there is a natural identification
\beq\la{luc}
p_* \m{IC}_{X_\PP (w_\bullet)} = \m{IC}_{X_\PP (w_1)} * \cdots * \m{IC}_{X_\PP (w_r)}.
\eeq
Of course, the right hand side is an abuse of notation since our intersection complexes are only perverse up-to-shift, but the meaning should be clear.
\end{rmk}

\section{Proofs of Theorems \ref{tma} and \ref{tmb} and a semisimplicity question}\la{abssq}$\;$

\subsection{The decomposition theorem over a finite field}\la{dtoff}$\,$

The following  proposition may be  well-known to experts.
We  could not find an adequate explicit reference in the literature. A stronger result, also possibly
well-known, holds and we refer to \ci[Prop.\,2.1]{dC} for this stronger statement and its proof, which follows  from some results
in \ci{bbd}.

\begin{pr}\la{dtff}
Let $f: X \to Y$ be a proper map  of varieties over the  finite field $k$, let $P$ be a pure perverse
sheaf of weight
$w$ on $X$. Then the direct image complex $f_* P \in D^b_m(Y,\oql)$ is pure of weight $w$
and splits into the direct sum of terms of the form $\m{IC}_{Z}(L)[i]$,
where   $i \in \zed,$ $Z\subseteq Y$ is a closed integral subvariety of $Y$,   and $L$ is lisse, pure and indecomposable
on a suitable Zariski dense smooth  subvariety ${Z}^o \subseteq Z$.
\end{pr}

\begin{ex}\la{jordan} {\rm ({\bf Jordan-block sheaves})}
Let ${\m J}_n$ be the lisse rank $n$-sheaf on ${\rm Spec} (k)$
with stalk $\oql^n$ and Frobenius acting by means of the unipotent
rank $n$ Jordan block \ci[p.138-139]{bbd}. The lisse sheaf ${\m J}_n$ is pure of weight zero, indecomposable,  and when $n > 1$, neither 
semisimple nor Frobenius  semisimple. The same is true after pull-back to 
a smooth irreducible variety. Of course, $\ov{{\m J}_n}$ is constant, hence semisimple, on 
${\rm Spec} (\ov{k}).$
\end{ex}

\begin{fact}\la{factjod} {\rm ({\bf Indecomposables})}
The indecomposable pure perverse sheaves on a variety $X$ are of the form $\m{S}\otimes
{\m J}_n$ for some $n$ and for some simple pure perverse sheaf $\m{S};$ see \ci[Prop.\,5.3.9]{bbd}.
\end{fact}

\begin{rmk}\la{rtkk} {\rm ({\bf Simple, yet not Frobenius semisimple?})} We are not aware of an example of a simple  lisse sheaf
that is not Frobenius semisimple.  According to general expectations related to the Tate conjectures over finite fields, there should be no such sheaf. 
\end{rmk}

\subsection{Proof of the semisimplicity criterion Theorem \ref{tma}}\la{sscrt}$\,$

We need the following elementary
\begin{lm} \label{ss_Kron}
Suppose $T_1: V_1 \rightarrow V_1$ and $T_2: V_2 \rightarrow V_2$ are linear automorphisms of finite dimensional vector spaces over an algebraically closed field. Suppose $T_1 \otimes T_2 : V_1 \otimes_k V_2 \rightarrow V_1 \otimes_k V_2$ is semisimple. Then both $T_1$ and $T_2$ are semisimple.
\end{lm}
\begin{proof}
We may write in a unique way $T_i = S_i U_i,$ where $S_i U_i = U_i S_i$ and $S_i$ is semisimple and $U_i$ is unipotent. Then $T_1 \otimes T_2 = (S_1 \otimes S_2)(U_1 \otimes U_2) = (U_1 \otimes U_2)(S_1 \otimes S_2),$  where  $S_1 \otimes S_2$ is  semisimple and $U_1 \otimes U_2$ is unipotent (for the latter, observe that $U_1 \otimes U_2 - {\rm id}\otimes {\rm id} = (U_1 - {\rm id}) \otimes U_2 \, +\, {\rm id} \otimes (U_2 - {\rm id})$ is nilpotent). Thus $U_1 \otimes U_2 = {\rm id} \otimes {\rm id}$, which implies $U_i = {\rm id}$ and hence $T_i = S_i$ for $i = 1,2$.
\end{proof}

\smallskip
\noindent {\bf {\em Proof of   the semisimiplicity criterion for direct images Theorem \ref{tma}.}} 

One direction is trivial from the definitions, if
$ \m{F} \in \db{Y}$ and $f_* \m{F}$  is  Frobenius semisimple for every closed point $y$ in $Y$, then, by proper base change, we have that
$H^*(\ov{f^{-1} (y)}, \ov{\m F})$
 is Frobenius semisimple for every closed point $y$ in $Y$.

We argue the converse as follows. By the definition of semisimple complex,
it is enough to prove the assertion for a simple --hence pure-- perverse sheaf  ${\m{F}}$.
According to the decomposition theorem over a finite field Proposition \ref{dtff}, the direct image complex $f_* \m{F}$ splits into a direct
sum of cohomologically-shifted terms of the form 
$\ic{Z}(\m{R})$ where $Z$ is a closed integral subvariety of $Y$ and $\m{R}$
is a pure lisse sheaf on a suitable Zariski-dense open subset
${Z}^o \subseteq Z$. Without loss of generality, we may assume that the pure  lisse sheaves $\m{R}$ are indecomposable.

By applying Fact \ref{factjod},  we obtain that each lisse $\m{R}$ has the form $\m{L} \otimes
{\m J}_h$, for  some $h \geq 1$ and some  lisse simple $\m{L}$.
The desired conclusion follows if we show that  in each direct summand
above, we must have  that the only possible value for $h$ is $h = 1$. 

\n 
Fix such a summand. Pick any point $\ov{y} \in {Z}^o (\ov{k})$.  By proper base change, the semisimplicity assumption
ensures that the graded stalks $\m{H}^*(f_* \m{F})_{\ov{y}}$ are semisimple graded  Galois modules.
It is then clear that $\m{L}_{\ov{y}} \otimes {\m J}_h$, being a graded  Galois module which is a
graded subquotient of 
the  graded semisimple  Galois module  $\m{H}^*(f_* \m{F})_{\ov{y}}$,
is also semisimple. We conclude by using Lemma \ref{ss_Kron}. \qed

\subsection{Proof that the intersection complex splits off Theorem \ref{tmb}}\la{icsumd}$\,$

Recall that one can define the intersection complex $\ic{X}\in \db{X}$ for any variety over the finite field $k$  as follows (see \ci[\S4.6]{dC12}): since nilpotents are invisible for the \'etale topology,
we may assume that $X$ is reduced; let $\mu: \coprod_iX_i \to X$ be the natural finite
map from the disjoint union
of the irreducible components of $X;$ define $\ic{X}:= \mu_* \left(\oplus_i  \ic{X_i}\right).$
Note that $\ic{X}$ is then pure of weight zero and semisimple on $X$. 

\smallskip

\n
{\em \bf Proof of  Theorem \ref{tmb}.}
We may replace $Y$ with $f(X)$ and assume that $f$ is surjective. We may work with  irreducible components and assume that $X$ and $Y$ are integral.

\n
In view of Theorem \ref{dtff}, we have an isomorphism $f_* \ic{X} \cong \bigoplus_{a,i}
\ic{Z_a} (R_{ai})[-i],$ where the $Z_a$ range among a finite set of   closed  integral  subvarieties
of $Y$, $i \in \zed^{\geq 0}$ and the $R_{ai}$ are lisse  on suitable, smooth, open and dense 
subvarieties $Z^o_a \subseteq Z_a$. By removing from $Y$ all of the closed  subvarieties $Z_a \neq Y$, and possibly by further shrinking $Y$, we 
may assume that  $Y$ is smooth and that the direct sum decomposition takes the form
$f_* \ic{X} \cong \oplus_{i\geq 0} R^i[-i],$ where each  $R^i:=R^if_*\ic{X}$  is lisse on $Y$.

\smallskip

\noindent {\bf Claim.} After having shrunk $Y$ further, if necessary, we have that ${\oql}_Y$ is a direct summand of $R^0:=R^0f_* \ic{X}$.

Note that the claim implies immediately the desired conclusion: if the restriction of $(f_*\m{IC}_X)_{|U}$ over an open subset $U \subseteq Y$
admits a direct summand, then the intermediate extension of such summand to $Y$ is a direct summand  of $f_*\m{IC}_X.$

\noindent {\em Proof of the Claim.} 
Let $f= h\circ g: X \to Z \to Y$ be the Stein factorization of $f.$ In particular, $g$ and $h$ are proper surjective,  the fibers of $g$ are geometrically connected and $h$ is finite.  By  functoriality,
we have that $R^0:=R^0f_* \ic{X}= R^0h_* R^0g_* \ic{X}$. 
Without loss of generality, we may assume that $X$ is normal: take the normalization $\nu: \hat{X} \to X;$
we have $\nu_* \m{IC}_{\hat{X}} =\m{IC}_X$; then $(f \circ \nu)_* \m{IC}_{\hat{X}}= f_*\m{IC}_X$.
Since now $X$ is normal, we have that the natural map  ${\oql}_X  \to \m{IC}_X$ induces an isomorphism ${\oql}_X \cong \m{H}^0(\m{IC}_X).$ In particular, we get a distinguished triangle
${\oql}_X \to \m{IC}_X \to \tu{1} \m{IC}_X\to.$ We apply $Rg_*$ and obtain the distinguished triangle $Rg_* {\oql}_X \to Rg_* \m{IC}_{X} \to  Rg_* \tu{1}\m{IC}_{X} \to$.
Since   $g_*$ is left-exact for the standard $t$-structure, 
we see that $R^0 g_* (\tau_{\ge1}\m{IC}_{X})=0$.
We thus get natural isomorphism  ${\oql}_Z \cong R^0g_* {\oql}_X \cong R^0g_* \ic{X}$, where the first one stems from the fact that $g$ has geometrically connected fibers. It remains to show that
$R^0h_* {\oql}_Z$ admits ${\oql}_Y$ as a direct summand. By shrinking $Y$ if necessary, 
we may assume that $h: Z \to Y$ is finite surjective between smooth varieties.
Using Lemma \ref{factorize}, we factorize $h= s\circ i$, where $s$ is separable and $i$ is purely inseparable.
Since $i$ is a universal homeomorphism, we have that $i_*$ is isomorphic to the identity.

It remains to show that, possibly after shrinking $Y$ further, $R^0s_* {\oql}_Z$ admits ${\oql}_Y$ as a direct summand.  After shrinking $Y$, if necessary, we may assume by Lemma \ref{factorize} that $s$ is finite and \'etale. 
It follows that $s^! {\oql}_Y\cong {\oql}_Z.$ By consideration of the natural adjunction maps, we thus get natural maps $R^0s_* {\oql}_Z = R^0 s_! {\oql}_Y\stackrel{a} \to {\oql}_Y\stackrel{b} \to R^0 s_* {\oql}_Z$,  with $a\circ b= (\deg{s}) \, {\rm Id}$ (see \cite[Lem.\,V.1.12]{Mil}). The desired splitting follows. The Claim is thus proved, and so is the theorem.
\qed

\subsection{A semisimplicity conjecture}\la{ssconj}$\,$

Given a complete variety $X$ over the finite field $k$, one may conjecture that the graded Galois module
$H^*(X,\oql)$ is semisimple, i.e.,\,\,that there should be  no non-trivial Jordan factors under the action of the Frobenius
automorphism. We have the following 

\begin{conj}\la{conjic} Let $f: X\to Y$ be a proper map of  varieties over the finite field $k.$
For every closed point $y$ in $Y,$ the graded Galois modules  $H^*(\ov{f^{-1}(y)}, \ov{\ic{X}})$  are semisimple.  In particular, in view of Theorem  \ref{tma}, the direct image
$f_* \ic{X}$  is semisimple and Frobenius semisimple.
\end{conj}

Let us remark that in view of de Jong's theory of  alterations \ci{deJ}, Conjecture \ref{conjic}, concerning intersection cohomology, follows from the semisimplicity conjecture
in ordinary $\oql$-adic cohomology stated at the very  beginning of this subsection. This implication follows immediately by combining the   proper base change theorem with  the splitting-off  of the intersection complex Theorem \ref{tmb} (N.B.\, given that we are working with generically finite morphisms, in place of Theorem \ref{tmb} we may use the more elementary \cite[Lemma 10.7]{GH}), for then
we can take the composition $f_1: = f \circ a: X_1 \to X \to Y$, where $a$ is an alteration, and use the conjectural semisimplicity
of the graded Galois module $H^*(\ov{f_1^{-1} (y)}, \oql)$ to deduce it for its (non-canonical) Galois module direct summand
$H^*(\ov{f_1^{-1} (y)}, \ov{\ic{X}})$.

One may ask the  even more general 

\begin{???}\la{conjpure} 

Let $\m{F}$ be a simple mixed \textup{(}hence pure\textup{)} perverse sheaf on a variety $X$ over a finite field $k$. 
Is $\m{F}$ Frobenius semisimple, i.e.\,\,is the action of Frobenius on its stalks semisimple?
Recall that this does not seem to be known even in the case of  a simple lisse sheaf on $X$ smooth
and geometrically connected, nor in the case of the intersection complex $\m{IC}_X$.
Let $f: X\to Y$ be a proper morphism  of  $k$-varieties. 
Are
the graded Galois modules  $H^*(\ov{f_1^{-1} (y)}, \ov{\m{F}})$  semisimple for every $\ov{y} \in Y(\ov{k})$,
so that,  in view of Theorem  \ref{tma}, the direct image
$f_*  \m{F}$ is semisimple and Frobenius semisimple? \end{???}

\section{Proofs of Theorems \ref{tmff}, \ref{ctm} and \ref{dtm}}\la{tmctmffpfs}$\;$

\subsection{Proof of the surjectivity for fibers criterion Theorem \ref{tmff}}\la{pftmff}$\,$

In this section, we prove Theorem \ref{tmff}, which is the key to proving Theorems \ref{ctm}, \ref{dtm}.   We first remind the reader of the ``retraction" Lemma \ref{retrlm}.
We then establish  the  local product structure Lemma \ref{ulemma}.
We are unaware of a reference for these local product structure
results in the generality we need them here.
We introduce a certain  contracting $\mathbb G_m$-action on (\ref{prodfo}). With (\ref{prodfo}) and the contracting action  we then conclude the proof of Theorem \ref{tmff} by means 
of the  retraction Lemma \ref{retrlm} followed by weight considerations.
This  kind of argument has already appeared in the context of proper toric fibrations \ci{dC} and it can
be directly fed into to the context 
of this paper, once we have  the local product structure Lemma \ref{ulemma}.

\begin{lm}\label{retrlm} {\rm ({\bf Retraction lemma})}
Let $S$ be a $k$-variety endowed with a $\mathbb G_m$-action
that ``contracts" it to a $k$-rational point $s_o \in S$, i.e.
the action $\mathbb G_m \times S \to S$ extends
to a map $h: \mathbb A^1 \times S \to S$ such that
\[
h^{-1} (s_o) = (\mathbb A^1 \times \{s_o\}) 
\bigcup (\{0\} \times S).\]
Let $\m{E} \in \db{S}$ be $\mathbb G_m$-equivariant. 
Then the natural restriction map of graded Galois modules
$H^*(\ov{S},\ov{\m{E}}) \to \m{H}^*(\m{E})_{\ov{s}}$
is an isomorphism. 
\end{lm}
\begin{proof}
This lemma is proved in \ci[Lemma 6.5]{DeLo}, in the case
when $\m{E}=\m{IC}_S$ is the intersection complex (automatically $\mathbb G_m$-equivariant); this seems
to be rooted in \ci[Lemma 4.5.(a)]{KL}.
The proof of the above simple generalization to the direct image under a proper map  of a weakly equivariant  $\mathbb G_m$-equivariant
complex is contained in the proof of  \ci[Lemma 4.2]{dMM}. We also draw the reader's attention to \ci[Cor.\,1]{Spr84}, which is probably the original reference for this result. \end{proof}

\begin{rmk}\la{retrrml}
If  $\,\mathbb G_m$ acts linearly on
$\mathbb A^n$ with positive weights,
 $S \subseteq \mathbb A^n$ is a $\mathbb G_m$-invariant
closed subscheme and $\m{E}$ is $\mathbb G_m$-equivariant on
$S$,  then $(S,\m{E})$ satisfy the hypotheses of Lemma \ref{retrlm}. If, in addition, $f: T \to S$ is a proper  $\mathbb G_m$-equivariant map and $\m{F}$ is $\mathbb G_m$-equivariant on $T$, then Lemma \ref{retrlm} combined with proper base change
yields natural isomorphisms of graded Galois modules
$H^*(\ov{T}, \ov{\m{F}}) \to H^*( \ov{f}^{-1} (\ov{s}_o), \ov{\m{F}}),$ where $s_o$ is the origin in $\mathbb A^n$.
\end{rmk}

Consider  the ``dilation'' action $c$ of $\mathbb G_m$ on $k[\![t]\!]$ which sends $t$ to $at$ for $a\in k^\times$. 
We can define the same kind of  dilation action  on  
$T(k[\![t]\!])$, $T(k(\!(t)\!))$, $\mathcal B$,  $\mathcal{P}$,  $\G$, and  $\G/\PP$, thus on the
closures of $\B$ and of $\PP$-orbits.

Recall that we are in the context of Theorem \ref{tmff}: $X:= \xbp{\B \PP}{w} \subseteq \G/\PP$ is the closure
of a $\B$-orbit (special case: the closure of a $\PP$-orbit); we are fixing $\ov{x} \in X(\ov{k})$.
By passing to a finite extension of the finite ground field  $k$, if necessary, and by using the
$\B$-action, we may assume that the point $x$ is a $T(k)$-fixed point $x_v$ for a suitable $v\le w\in 
{\bwp{\B}{\PP}}$. This latter parameterizes the $\B$-orbits $Y_{\B\PP} (v)$  in $\G/\PP$, which, in what follows, we simply denote 
by $Y(v)$.

\begin{lm}\label{ulemma}
There is a commutative diagram with cartesian squares
\begin{equation}\la{pir}
\xymatrix{
Y(v) \times g^{-1} (S_v) \ar[r]^{\hspace{.15in}\sim}  \ar[d]^{1 \times g} &
g^{-1} (X_v) \ar@{^(_->}[r]  \ar[d] ^g & Z \ar[d]^g \\
Y(v) \times S_v \ar[r]^{\hspace{.12in} \sim} & X_v \ar@{^(_->}[r] & X
}
\end{equation}
where, $S_v \subseteq X_v$ is a  closed subvariety
containing $x_v$,  the inclusions are open immersions and the indicated isomorphisms are equivariant for the actions of the
groups $(\mathcal U\cap\, ^v\overline{\mathcal U}_\mathcal P)$, $T(k)$, $c$ \textup{(}actually, the given $c_Z$ on $Z$\textup{)}.  Moreover, there is a $\mathbb G_m$-action
on $S_v$ that  contracts it to $x_v$ and that lifts
to $g^{-1} (S_v)$.
\end{lm}
\begin{proof}
Denote $^v\overline{\mathcal U}_\mathcal P:=v\overline{\mathcal U}_\mathcal Pv^{-1}$. 
In $\S\ref{notsy}$, we stated the product decompositions
\begin{equation}\label{profo}
\;^v\overline{\mathcal U}_\mathcal P=(\mathcal U\cap\, ^v\overline{\mathcal U}_\mathcal P)( \overline{\mathcal U}\cap \, ^v\overline{\mathcal U}_\mathcal P).
\end{equation}

Let $v\mathcal{C}_\mathcal{P}=\,^v\overline{\mathcal{U}}_\mathcal{P}\cdot x_v\cong \,^v\overline{\mathcal{U}}_\mathcal{P}$ be the open big cell in $\mathcal{G}/\mathcal{P}$ at $x_v:=v\mathcal{P}/\mathcal{P}$ (cf.\,(\ref{bigcell})). 
According to (\ref{profo}) it admits a product decomposition 
\beq\la{evt}
v\mathcal{C}_\mathcal{P} \cong Y(v) \times C^v_\infty,
\eeq
where the $Y(v)$ factor can be identified, thanks to (\ref{decomp1}), as 
$$Y(v)=(\mathcal U \cap \, ^v\overline{\mathcal U}_\mathcal P)\cdot x_v \, \cong \, \mathcal U \cap \, ^v\overline{\mathcal U}_\mathcal P,$$
and the second factor, which  is not of finite type, is defined by setting
$$v\mathcal{C}_\mathcal{P}\supset C^v_\infty:=(\overline{\mathcal U} \cap \, ^v\overline{\mathcal U}_\mathcal P)\cdot x_v \, \cong \, \overline{\mathcal U} \cap \, ^v\overline{\mathcal U}_\mathcal P.$$
The $\mathcal{B}$-orbit $Y(v)$ is a $({\mathcal U} \cap \, ^v\overline{\mathcal U}_\mathcal P)$-torsor; also, let this latter group act trivially on $C^v_\infty$. 
By \eqref{profo}, we have  that (\ref{evt}) is an  $({\mathcal U} \cap \, ^v\overline{\mathcal U}_\mathcal P)$-equivariant isomorphism.

We set
 $X_v:=v\mathcal{C}_\mathcal{P}\cap X$. Then the composition
 $X_v\to v\mathcal{C}_{\mathcal P}\to Y(v)$
 is $({\mathcal U} \cap \, ^v\overline{\mathcal U}_\mathcal P)$-equivariant. Let
$$S_v:= C^v_\infty \cap X_v = C^v_\infty \cap X.$$
We thus see that there is an $({\mathcal U} \cap \, ^v\overline{\mathcal U}_\mathcal
 P)$-equivariant isomorphism 
\begin{equation*}
X_v \,\overset{\sim}{\rightarrow} \,Y(v)\times S_v.
\end{equation*}
Note that $Y(v)$ is a finite dimensional affine space; $S_v$ is what one calls    {\em  the slice of  $X_v$ at $x_v$ transversal to $Y(v)$.}

Let $2\rho^\vee$ be the sum of the positive coroots (viewed as a cocharacter), let $n, i$ be integers with $i >\!\!> n >\!\!> 0$, and define a cocharacter $\mu = -2n\rho^\vee$. We claim that for sufficiently large $i >\!\!> n >\!\!> 0$, the $\mathbb{G}_m$-action on $X$ defined using $(\mu,c^{-i})$ contracts $S_v$ to $x_v$, where {\em contract} means that the action extends to a morphism ${\mathbb A}^1 \times S_v \rightarrow S_v$ such that the hypotheses of Lemma \ref{retrlm} are satisfied with $s_o = x_v$. In order to prove this, it is enough to find an affine space $\mathbb A_v$  endowed with a $\mathbb G_m$-action and a closed embedding $(S_v, x_v) \hookrightarrow (\mathbb A_v, 0)$, such that 
\begin{enumerate}
\item[(i)] the $\mathbb G_m$-weights on $\mathbb A_v$ are $>0$ (see Remark \ref{retrrml});
\item[(ii)] the $\mathbb G_m$-action on $\mathbb A_v$ preserves $S_v$ and restricts to the action on $S_v$ via $(\mu, c^{-i})$.
\end{enumerate}
It is enough to prove these statements over $\bar{k}$, so we write $k$ for $\bar{k}$ in the rest of this argument.

Recall that $\ov{\mathcal U}$ is an ind-scheme which is ind-finite type and ind-affine.  We need to make this more precise.  Choose a faithful representation of $G \hookrightarrow {\rm GL}_N$, a maximal torus $T_N$ in ${\rm GL}_N$ as well as Borel subgroups $B_N = T_N U_N$ and $\bar{B}_N = T_N \bar{U}_N$ as in Remark \ref{Borel_pair_rem}.

We have an exact sequence of group ind-schemes
$$
1 \rightarrow L^{--}G \rightarrow \overline{\mathcal U} \rightarrow \bar{U} \rightarrow 1,
$$
where $L^{--}G$ is the kernel of the map $G(k[t^{-1}]) \rightarrow G(k[t^{-1}]/t^{-1})$. Using the embedding $G(k[t,t^{-1}]) \subset {\rm GL}_N(k[t,t^{-1}])$,  an element $g \in L^{--}G(k)$ can be regarded as a matrix of polynomials 
\begin{equation} \la{poly_space}
g_{ij} = \delta_{ij} + a^1_{ij}t^{-1} + a^2_{ij} t^{-2} + \cdots
\end{equation}
whose coefficients $a^k_{ij}$ satisfy certain polynomial relations which ensure that $(g_{ij})$ lies in $G(k[t^{-1}])$. Fix an integer $m \geq 0$, and let $L_{m}^{--}G$ be the set of $g \in G(k[t^{-1}])$ such that ${\rm deg}_{t^{-1}}(g_{ij}) \leq m$ for all $i,j$.  Similarly define $\overline{\mathcal U}_m$. The ind-scheme structure on $\overline{\mathcal U}$ is given by this increasing union of closed affine $k$-varieties: $\overline{\mathcal U} = \bigcup_m \overline{\mathcal U}_m$.  On the other hand, $S_v \subset (\overline{\mathcal U} \cap \, ^v\overline{\mathcal U}_{\PP}) \, x_v$ is an integral $k$-subvariety, and the closed subschemes $\overline{\mathcal U}_m x_v \cap S_v$ exhaust $S_v$.  Henceforth we fix $m$ so large that the generic point of $S_v$ is contained in $\overline{\mathcal U}_m x_v \cap S_v$; then this intersection coincides with $S_v$ and hence there is a closed embedding $S_v \subset (\overline{\mathcal U}_m \cap \, ^v\overline{\mathcal U}_{\PP})x_v$.  

As $S_v$ is isomorphic to a closed subscheme of $\overline{\mathcal U}_m \cap \, ^v\overline{\mathcal U}_{\PP}$,  it is enough to find a closed embedding of $\overline{\mathcal U}_m$ into an affine space ${\mathbb A}_v$ carrying a $\mathbb G_m$-action which satisfies (i) and (ii). Clearly $L^{--}_mG$ is a closed $k$-subvariety of the affine space ${\mathbb A}_m$ consisting of {\em all} matrices $(g_{ij})$ whose entries have the form (\ref{poly_space}) with ${\rm deg}_{t^{-1}}(g_{ij}) \leq m$ for all $i, j$. The group $\bar{U}$ is isomorphic as a variety to $\prod_{\alpha < 0} U_\alpha$ and each $U_\alpha$ is isomorphic to $\mathbb A^1$ (non-canonically).  We can therefore identify $\bar{U}$ with an affine space.  The space $\mathbb A_v := \mathbb A_m \times \bar{U}$ carries the diagonal ${\mathbb G}_m$-action via $(\mu, c^{-i})$ (by construction, $c$ acts trivially on $\bar{U}$). 

The exact sequence above splits, so  there is a canonical isomorphism of affine $k$-varieties
$$
\overline{\mathcal U}_m = L^{--}_m(G) \cdot \bar{U},
$$
and hence a closed embedding $\overline{\mathcal U}_m \hookrightarrow \mathbb A_v = \mathbb A_m \times \bar{U}$, compatible with the ${\mathbb G}_m$-actions defined via $(\mu, c^{-i})$.  The weights of the latter on $\mathbb A_v$ are clearly positive for $i >\!\!> n >\!\!>0$.  Also, these actions preserve the image of $(\overline{\mathcal U} \cap \, ^v \overline{\mathcal U}_{\PP}) x_v \cap X = S_v$.  Hence (i) and (ii) are verified, and we have constructed the desired contracting action of $\mathbb G_m$-action on $S_v$.

Finally, let us observe that  the $\mathbb G_m$-action $(\mu,c^{-i})$ on $X$ lifts to $Z$. Indeed, $\mu$ can be lifted because $\mu$ has image in $T(k) \subset \mathcal B$, and $g$ is $\mathcal B$-equivariant. By assumption, $g$ is also $c$-equivariant
($c$ on $X$, $c_Z$ on $Z$). It follows that  the $\mathbb G_m$-action given by  $(\mu,c_Z^{-i})$ acts on $Z$ and that $g$ is equivariant with respect to these $\mathbb G_m$-actions $(\mu,c_Z^{-i})$ and $(\mu, c^{-i})$.

Moreover, the map $g:Z\to X$ is $\mathcal{B}$-equivariant, hence $(\mathcal U \cap \, ^v\overline{\mathcal U}_\mathcal P)$-equivariant.  We thus have the   $(\mathcal U \cap \, ^v\overline{\mathcal U}_\mathcal P)$-equivariant isomorphism
of varieties 
\beq\la{prodfo}g^{-1}(X_v) \, \overset{\sim}{\rightarrow} \,
(\mathcal U \cap \,^{v}\overline{\mathcal U}_\mathcal P)
\times g^{-1}(C^v_\infty \cap X_v) \, \overset{\sim}{\rightarrow}\, Y(v)\times g^{-1}(S_v).
\eeq
This establishes  \eqref{pir}.
\end{proof}

{\bf Proof  Theorem \ref{tmff}.}
Recall that, by using the $\B$-action,  we have reduced ourselves to the  case of the special $k$-rational points $x_v \in X
:= X_{\B\PP}(w)$, with $v \leq w$ in $\m{W}/\m{W}_\PP.$
We use (\ref{pir}).
Consider the following natural restriction/pull-back maps of graded Galois modules
\beq\la{4rf}
H^*(\ov{Z}, \ic{\ov{Z}}) \to H^*(\ov{g}^{-1} (\ov{X}_{v}), {\ic{\ov{Z}}})  \,\overset{\sim}{\rightarrow} \,
H^* (\ov{g}^{-1}(\ov{S_{v}}), {\ic{\ov{Z}}}) \, \overset{\sim}{\rightarrow} \, H^*(\ov{g}^{-1} (\ov{x_v}), {\ic{\ov{Z}}}),
\eeq
where the first isomorphism is due to the K\"unneth formula, and the second is due to the retraction Lemma \ref{retrlm}.
We freely use   the  weight argument  in \cite[Lemma 2.2.1]{dC}, which we summarize.
First we establish purity by means of a classical argument: the second module is mixed
with weights $\geq 0$, the last is mixed with  weights $\leq 0,$ so that the second module
is pure with  weight zero. Next, the first module is pure with weight zero and surjects onto the pure weight zero part
of the second, which is the whole thing (let $i$ be the closed embedding of the complement of $g^{-1}(X_v)$ in $Z$; then use the exact sequence $H^*(\ov{Z}, \ic{\ov{Z}}) \rightarrow H^*(\ov{g}^{-1} (\ov{X}_{v}), {\ic{\ov{Z}}})  \rightarrow H^{*+1}(\ov{Z}, \bar{i}_* \bar{i}^! \ic{\ov{Z}})$ and the fact that the r.h.s.\,\,has weights $\geq *+1$). Therefore we conclude that the composition 
$H^*(\ov{Z}, \ic{\ov{Z}}) \to  H^*(g^{-1} (\ov{x_v}), {\ic{\ov{Z}}})$ is surjective. 
All the assertions of the theorem, except for the last one follow at once.

If we replace $Z$ with a dense open subset $U\subseteq Z$ containing the fiber over any closed point, then  the weight argument above
can be repeated: we no longer have an open set in the shape of a  nice product, but we can argue in the same way
that the images of $I\!H^*(\ov{Z}, \oql)$ and of $I\!H^*(\ov{U}, \oql)$ in $H^*( \ov{g}^{-1} ( \ov{x}), \m{IC}_{\ov{Z}}),$
coincide.
This completes the proof of Theorem \ref{tmff}.  \qed

\subsection{Proof of Theorem \ref{ctm}}\la{pftmc}$\,$

\begin{lm}\la{pavingfr}
Let $X$ be a $k$-scheme which is paved  by affine spaces.
The compactly supported cohomology   $H^*_c(\ov{X}, \oql)$  is a \good \, graded Galois module.
In particular, if $X$ is proper, then the ordinary cohomology $H^*(\ov{X}, \oql)$  is a \good \, graded  Galois module.
\end{lm}
\begin{proof} Recall Definition \ref{affpav} (affine paving). The Borel-Moore homology $H^{BM}_*(\ov{X},\oql):
= H^{-*}(\ov{X}, \omega_{\ov{X}})$ ($\omega_X$ the  dualizing complex of $X$) is even and there is a natural isomorphism given by the cycle class map
\[{\rm cl}: A_*(X)\otimes_\zed \oql \cong H^{BM}_{2*}(\ov{X},\oql)(-*)^{Frob}=
H^{BM}_{2*}(\ov{X},\oql)(-*), \]
see, e.g.\,\ci[Example, 19.1.11]{Fulton} and \cite[$\S1.1$]{Olss}. Thus, there is a basis of Borel-Moore homology
given by Tate twists of the cycle classes of the closures $C_{ij} = \ov{\mathbb A^{n_{ij}}} \subseteq X$ of the affine cells.
In particular, each ${\rm cl}(C_{ij})(n_{ij}) \in H^{BM}_{2n_{ij}}(\ov{X}, \oql)$  is an eigenvector  of Frobenius with eigenvalue $|k|^{- n_{ij}}$ (note that, here,
$H^{BM}_{2k} (\ov{X} ,{\oql})$ is pure of weight $-2k$).
The conclusion follows by the Verdier duality isomorphisms of graded Galois modules 
$H^{BM}_*(\ov{X}, \oql)\cong  H^{*}_c (\ov{X}, \oql)^\vee$.
\end{proof}
\begin{lm}\la{good}
 {\rm ({\bf Demazure varieties are \good})} 
Let $X_{\B} (s_{\bullet})$ be a Demazure variety, i.e. a  twisted product with $s_\bullet \in \m{S}^r$ a vector of
simple reflections. Then we have that  $I\!H^*(\ov{X_\B (s_\bullet)}, \oql)= H^*(\ov{X_\mathcal B (s_\bullet)}, \oql)$  is \good \, and
generated 
by algebraic cycle classes.
\end{lm}
\begin{proof}
Since, by construction,  $X_\mathcal B (s_\bullet)$ is an iterated $\pn{1}$-bundle,  it is smooth 
of dimension $r$
so that we have natural isomorphims of   graded Galois modules
\[I\!H^*(\ov{X_\mathcal B (s_\bullet)}, \oql) \cong H^*(\ov{X_\mathcal B (s_\bullet)}, \oql) \cong 
H^{BM}_{2r- *}(\ov{X_\mathcal B (s_\bullet)}, \oql) (-r).\]
As the proof of  Lemma \ref{pavingfr} shows, the middle term is \good \, with weight zero and the r.h.s  is generated by algebraic cycle classes. In order to apply Lemma \ref{pavingfr}, we invoke the special case of Theorem \ref{pavingtm}(3) which asserts that $X_\B(s_\bullet)$ is paved by affine spaces.
\end{proof}

\begin{lm}\la{covro}
A twisted product variety  is the surjective image of a convolution morphism with domain a Demazure variety.
\end{lm}
\begin{proof}
Let $X_\PP (w_\bullet)$ be a twisted product variety. Let $u_i$ be the  maximal representative
in $\m{W}$  of $w_i$. Let $s_{i\bullet}$ be a reduced word for $u_i$.
The composition of surjective convolution morphisms $X_\B (s_{\bullet \bullet}) \to X_\B (u_\bullet) \to X_\PP (w_\bullet)$
yields the desired conclusion. 
\end{proof}

Let us record for later use  (proofs of Theorem \ref{ctm} below and Theorem \ref{pavingtm} in $\S \ref{pf??}$) that the construction in the proof of Lemma \ref{covro}, coupled with Remark \ref{rt41}
and Proposition \ref{01012}
yields the following commutative diagram (to this end, note that: by construction, we have  $w_\bullet = u''_\bullet$; by the proposition, we have that $u_\star''= w_\star$; by the remark, we have the indicated bundle structures)  
\beq\la{mmh}
\xymatrix{
X_\B (s_{\bullet \bullet}) \ar[r]^{\pi}  & X_\B (u_\bullet) \ar[r]^{p'} \ar[d]_q  & X_\B (u_\star) \ar[d]^{q'} \\
& X_\PP (w_\bullet) \ar[r]^p & X_\PP (w_\star),
}
\eeq
where all maps are surjective, $q,q'$ are Zariski locally trivial bundles with respective fibers $(\PP/\B)^r$ and 
$(\PP/\B)$. By the associativity of the Demazure product, we have that the Demazure product of the $s_{\bullet \bullet}$
coincides with that of the $u_\bullet$, i.e.  $s_\star= u_\star$.

\medskip

\noindent {\bf Proof of  Theorem \ref{ctm}.}

\n
Let $X_\PP (w_\bullet)$ be a twisted product variety. We need to prove that
its intersection cohomology groups $I\!H^*(X_\PP(w_\bullet), \oql)$  and its intersection complex
 $\m{IC}_{X_\PP(w_\bullet)}$ are \good.

\n
The first statement follows from Lemmata \ref{good}, \ref{covro} and Theorem \ref{tmb}. As for the second, the twisted product variety $X_{\PP}(w_\bullet)$ is locally isomorphic to the usual product, and $\m{IC}_{X_{\PP}(w_\bullet)}$ is locally isomorphic to $\m{IC}_{X_{\PP}(w_1)} \boxtimes \cdots \boxtimes \m{IC}_{X_{\PP}(w_r)}$.  Therefore it is enough to prove the case $r=1$, i.e.  it is enough to prove that $\m{IC}_{X_{\PP}(w)}$ is \good \, for every $w.$ We use diagram (\ref{mmh}) in the case $r=1$. By Theorem \ref{tmb} applied to the surjective morphism $q \circ \pi$, it is enough to prove that $(q\pi)_*(\ov{\mathbb Q}_\ell)$ is \good. For any closed point $x \in X_{\PP}(w)$, Theorem \ref{tmff} gives a surjection $I\!H^*(\ov{X_\B(s_{\bullet})}, \ov{\mathbb Q}_\ell) \twoheadrightarrow H^*(\ov{q\pi}^{-1}(\ov{x}), \ov{\mathbb Q}_\ell)$ of graded Galois modules, which shows that $R(q\pi)_*(\ov{\mathbb Q}_\ell)$ is \good \, by Theorem \ref{tma} and Lemma \ref{good}. (Alternatively, in place of Theorem \ref{tmff} and Lemma \ref{good}, we can use the paving
results Theorem \ref{pavingtm}(2) and Lemma \ref{pavingfr}.) 
\qed

\subsection{Proof of Theorem \ref{dtm}}\la{spdtm}$\;$

We use freely the diagram (\ref{vizo}) and the notation used the proof of Proposition \ref{gt50}.

Let $x\in X_\m{Q} (w''_{I,\bullet})$ be a closed point. Pick $\w{A}$ so that $x\in \w{A}$.
Theorem \ref{tmff} applies to each factor of the product map $p_A$.
By the K\"unneth formula, it follows that the restriction map  $I\!H^*(\ov{p_{\w{A}}^{-1}(\w{A})}, \oql) \to  H^*(\ov{p^{-1}(x)}, \ov{\m{IC}_Z})$  is surjective. By using the same weight argument as in the proof of Theorem \ref{tmff}
(below (\ref{4rf})), we deduce that the restriction map from any Zariski  open subset $U$ of $Z$ containing
$p^{-1}(x)$  is a surjection.

By taking $U=Z$, we see that  the restriction map 
$I\!H^*(\ov{Z}, \oql) \to  H^*(\ov{p^{-1}(x)}, \ov{\m{IC}_Z})$  is surjective. By Theorem \ref{ctm},
the domain of this restriction map is \good, hence so is the target. 

The just-established  fact that  the fibers are \good, coupled with  the proper base change theorem and with Theorem \ref{tma} ensures that $p_* \m{IC}_X$ is \good.

Finally, since $I\!H^0(\ov{Z}, \oql)$ is one-dimensional, and the fibers of $p$ are non-empty, we deduce that they are geometrically connected.
\qed

\section{Proof of the affine paving Theorem \ref{pavingtm}} \la{pf??}$\;$

\subsection{Proof of the paving fibers of Demazure maps Theorem \ref{pavingtm}.(1)}\la{oiu34}$\,$

Our original proof went along the lines of \cite[Prop.\,3.0.2]{H05}; see our earlier arXiv posting arXiv:1602.00645v2. Here we will give a more conceptual approach which was suggested by an anonymous referee. The key step is the following general technique for producing affine pavings (cf.\,Def.\,\ref{affpav}) of fibers of morphisms using the Bialynicki-Birula decomposition \cite{BB}. In what follows $k$ will denote any field, and ``point'' will mean ``closed point.''

\begin{lm} \label{BB_paving}
Suppose a split $k$-torus $\mathbb T$ acts on $k$-varieties $X$ and $Y$  and let $f : X \rightarrow Y$ be a proper $\mathbb T$-equivariant $k$-morphism. Assume $X$ is smooth and can be $\mathbb T$-equivariantly embedded into the projective space of a finite dimensional $\mathbb T$-module. Suppose a $k$-rational fixed point $y \in Y^{\mathbb T}(k)$ possesses a $\mathbb T$-invariant open affine neighborhood $V_y \subset Y$ such that there exists a cocharacter $\mu_y : \mathbb G_m \rightarrow \mathbb T$ which contracts $V_y$ onto $y$, and such that the set $X^{\mu_y}$ of fixed-points is finite and consists of $k$-rational points. Then $f^{-1}(y)$ possesses a paving by affine spaces defined over the field $k$.
\end{lm}

\begin{proof}
The fixed-point $y$ is ``attractive'' for the $\mathbb G_m$-action defined by $\mu_y$.  It is therefore an (isolated) ``repelling'' fixed point for the action defined by $-\mu_y$. Let $\{x_i\}_i:=X^{-\mu_y}= X^{\mu_y} $ be the common finite set of fixed points, which, by our assumptions, are $k$-rational.

In what follows, the notion of $\underset{t \rightarrow 0}{\rm lim}\, t\cdot x$ is made precise by using the language of $\mathbb A^1$-monoid actions as in Lemma \ref{retrlm}. We warn the reader that when used in this way, the symbol $t$ denotes a varying element of $\mathbb G_m$, and not the uniformizer in the rings $k[\![t]\!]$, $k(\!(t)\!)$, etc. 

Consider the Bialynicki-Birula decomposition of $X$ for the action defined by $-\mu_y$.  By our assumptions, and according to \cite[Thm.\,4.4]{BB} and \cite[Thm.\,5.8]{Hes}, we obtain a finite decomposition of $X$ by affine spaces defined over $k$
\begin{equation} \label{BB_X_decomp}
X = \coprod_i X_i, \qquad X_i := \{ x \in X \, | \, \underset{t \rightarrow 0}{\rm lim} \, t\cdot x = x_i\}\cong \mathbb A^{d_i}, 
\end{equation} 
where $t\cdot x = -\mu_y(t)(x)$ and $X_i$ is the ``attracting set'' for $x_i$ w.\,r.\,t.~the action defined by $-\mu_y$. 
We claim that if $f^{-1}(y) \cap X_i \neq \emptyset$, then $X_i \subseteq f^{-1}(y)$.  Let $x \in X_i$, so that 
\begin{equation} \label{limit}
x_i = \underset{t \rightarrow 0}{\rm lim}\, t\cdot x.
\end{equation}
If $x \in f^{-1}(y) \cap X_i$, then $x_i \in f^{-1}(y)$, for  $f^{-1}(y)$ is $\mathbb T$-invariant and closed.  Let $x \in X_i$ be arbitrary. Applying $f$ to (\ref{limit}), we find $\underset{t \rightarrow 0}{\rm lim}\, t \cdot f(x) = y$. Since $y$ is a repelling fixed-point for $-\mu_y$, this forces $f(x) = y$, so that $x \in f^{-1}(y)$.

It follows that the fiber $f^{-1}(y)$ is the union of certain cells in the decomposition (\ref{BB_X_decomp}) of $X$.
 \end{proof}

We now prove Theorem \ref{pavingtm}.(1).

We will apply Lemma \ref{BB_paving} to the morphism $p: X_\B(s_\bullet) \rightarrow X_\B(s_\star)$. The torus $\mathbb T$ is taken to be the product $\mathbb T := T \times \mathbb G_m$, where $T \subset G$ will act as usual and $\mathbb G_m$ will act through $c$, the dilation action discussed in $\S\ref{pftmff}$. 

Since $Y_\B(v) \subset v \overline{\mathcal U} x_e$, we see that $\mathcal G/\B$ is covered by the open $\mathbb T$-invariant subsets $v \overline{\mathcal U} x_e$ \,($v \in \widetilde{W}$). Further, as in the proof of Lemma \ref{ulemma}, for integers $i >\!\!> \!n \!>\!\!> 0$ we set $\mu = -2n\rho^\vee$ and define $\mu_{c^{-i}}: \mathbb G_m \rightarrow T \times \mathbb G_m, \,\, a \mapsto (\mu(a), a^{-i})$. Then $v \mu_{c^{-i}} v^{-1}$ contracts the open neighborhood $v\overline{\mathcal U} x_e$ onto the fixed point $v x_e = x_v$. 

In particular the only $\mathbb T$-fixed points in $\mathcal G/\B$ are the points $x_w$ for $w \in \widetilde{W}$. Moreover we claim that the $v \mu_{c^{-i}} v^{-1}$-fixed points in $\mathcal G/\B$ are also just the points $x_w$\, ($w \in \widetilde{W}$).  We easily reduce to the case $v =1$. A point in $(\mathcal G/\B)(\bar{k})$ can be written in the form $\bar{u} \cdot w \cdot x_e$ for  unique elements $w \in \widetilde{W}$ and $\bar{u} \in \overline{\mathcal U} \cap \, ^w\overline{\mathcal U}$ (see \cite[Lem.\,3.1]{GH} and (\ref{decomp2})).  Clearly $\bar{u} \cdot w \cdot x_e$ is fixed by $\mu_{c^{-i}}$ if and only if $\bar{u} \cdot w \cdot x_e = \underset{t \rightarrow 0}{\rm lim}\, t\cdot (\bar{u} \cdot w \cdot x_e)$ if and only if $\bar{u} \cdot w \cdot x_e = w \cdot x_e$. 

It follows that each Schubert variety $X_\B(w)$ and consequently each twisted product $X_\B(w_\bullet)$ has only finitely many $\mu_{c^{-i}}$-fixed points.

From these remarks it follows that any $\mathbb T$-fixed point $y = x_v \in X_\B(s_\star)$ has an invariant neighborhood which is contracted onto $y$ by a cocharacter $\mu_y := v \mu_{c^{-i}} v^{-1}$ for which $X_\B(s_\bullet)^{\mu_y}$ consists of finitely-many $k$-rational points. Thus all the hypotheses of Lemma \ref{BB_paving} are satisfied for the morphism $p: X_\B(s_\bullet) \rightarrow X_\B(s_\star)$, and we conclude that the fibers of $p$ over $\mathbb T$-fixed points are paved by affine spaces.

Finally we prove the triviality of the map $p$ over $\mathcal B$-orbits contained in its image. Assume $Y_\B(v) \subset X_\B(s_\star)$.  An element $\mathcal B' \in Y_\B(v)$ can be written in the form 
$$
\mathcal B' = {uv}\mathcal B
$$
for a unique element $u \in \mathcal U \cap \, ^v\overline{\mathcal U}$. We can then define an isomorphism
$$
p^{-1}(Y_\B(v)) ~ \overset{\sim}{\longrightarrow} ~ \, p^{-1}(v\mathcal B) \times Y_\B(v)
$$
by sending $(\mathcal B_1, \dots, \mathcal B_{r-1}, {uv}\mathcal B)$ to $(\, {u^{-1}}\mathcal B_1, \cdots, {u^{-1}}\mathcal B_{r-1}, {v}\mathcal B) \times {uv}\mathcal B$. 

This completes the proof of Theorem \ref{pavingtm}.(1).
\qed

\subsection{Proof of  the paving Theorem \ref{pavingtm}.(2)}\la{rtyuk}$\,$

In terms of diagram (\ref{mmh}), we need to pave by affine spaces the fibers of $q' \circ p' \circ \pi$. By $\B$-equivariance, we need consider only the fiber over $w\mathcal P/\mathcal P$ for $w \leq u_*$ in $\mathcal W/\mathcal W_{\PP}$. Let $w \in \mathcal W$ be a {\em minimal} element in its coset $w \mathcal W_\mathcal P$. Then
$$
q'^{-1}(w\mathcal P/\mathcal P) = \coprod_{w' \in \mathcal W_\mathcal P} Y_\mathcal B(ww').
$$
Each $Y_\mathcal B(ww' )$ is locally closed  in this fiber, and $Y_\mathcal B(w w'') \subset \overline{Y_\mathcal B(w w')}$ if and only if $w'' \leq w'$. By Theorem \ref{pavingtm}(1) applied to $p' \circ \pi$, we see that each $(p' \circ \pi)^{-1}(Y_\B(ww'))$ is paved by affine spaces. Theorem \ref{pavingtm}(2) follows.
\qed

\subsection{Proof of Theorem \ref{pavingtm}.(3)}\la{pzvr}$\;$

We need to prove that the variety $X_{\PP}(w_\bullet)$ is paved by affine spaces. The result can be proved by induction on $r$.  The case $r=1$ is just the statement that $X_{\PP}(w_1)$ is paved by affine spaces, which is clear. In fact, we even have the $\B$-invariant paving $X_{\PP} (w_1) = \coprod_{v}Y_{\B \PP}(v)$ where $v$ ranges over elements in $\W/\W_{\PP}$ such that $v \W_{\PP} \, \leq \, w_1 \mathcal W_{\PP}$.  The fact that $Y_{\B \PP}(v)$ is an affine space is shown in \eqref{Y_B_in_big_cell}.

The morphism  $X_\PP (w_\bullet) \to X_\PP (w_1)$ in Lemma \ref{zlt1} is: $\B$-equivariant
with fibers  isomorphic to 
   $X_{\PP} (w_2, \ldots, w_r)$; Zariski-locally trivial
over the base, and in fact  trivial over the intersection of $X_{\PP}(w_1)$ with any big cell.
Each $Y_{\B \PP}(v)$ is an affine space, so by induction, it suffices to prove this morphism is trivial over all of $Y_{\B \PP}(v)$. But by (\ref{Y_B_in_big_cell}), $Y_{\B \PP}(v)$ is contained in the big cell through $x_v$, and hence we get the desired triviality assertion.
\qed

\subsection{Proof of Corollary \ref{conv} via paving}\la{anches}$\;$

In light of  (\ref{luc}), Corollary \ref{conv} is the special case of Theorem \ref{dtm} with
$r'=r$ and $m=1$ and $\PP=\m{Q}$ (recall that in this case $\Q$-maximality is automatic). We offer a different proof
based on the paving  Theorem \ref{pavingtm}.(2) for  the fibers of the map $\phi := p \circ q \circ \pi = q' \circ p' \circ \pi$ arising from diagram (\ref{mmh}).

By Theorem \ref{tmb}, the complex $\m{IC}_{X_\PP (w_\bullet)}$  is a direct summand of $q_* \pi_*
{\oql}_{X_\B (s_{\bullet \bullet})}.$ For the same reason, the complex $p_* \m{IC}_{X_\PP (w_\bullet)}$  is a direct summand of $\phi_* {\oql}_{X_\B (s_{\bullet \bullet})}.$ It follows that it is enough to show that the latter is \good.  

By proper base change and Theorem \ref{tma}, we see that the paving of the fibers of $\phi$  Theorem \ref{pavingtm}.(2)
ensures that  $\phi_* {\oql}_{X_\B (s_{\bullet \bullet})}$ is \good.
\qed


\section{Remarks on the Kac-Moody setting and results over other fields $k$} \label{comp_lit_sec}

\subsection{Remarks on the Kac-Moody setting} \la{KM_rmks}

As noted in the introduction, if $G$ is a $k$-split simply connected semisimple group, then $\mathcal G = LG$ is a Kac-Moody group over $k$ but if $G$ is only reductive, then $LG$ is not Kac-Moody. We remark here that our techniques give results also when $\mathcal G$ is an arbitrary Kac-Moody group.  In this case, one has a refined Tits system $(\mathcal G, N, \mathcal U, \mathcal U^-, T, S)$ (see \cite[Def.\,5.2.1, Thm.\,6.2.8]{Kum}) and for any parabolic subgroup $\mathcal P \subset \mathcal G$, one has the Kac-Moody partial flag ind-variety $\mathcal G/\mathcal P$, Schubert varieties $\overline{\mathcal P w \mathcal P/\mathcal P}$, associated Bruhat decompositions $\mathcal G = \cup_{w \in \, _\mathcal P W_\mathcal P} \mathcal P w \mathcal P$ and Bott-Samelson morphisms $X_\mathcal P(w_\bullet) \rightarrow \mathcal G/\mathcal P$, as well as a theory of big cells and a Birkhoff decomposition (as in \cite[Thm.\,6.2.8]{Kum}). These objects satisfy the formal properties listed axiomatically in \cite[Chap.\,5]{Kum}. This is all described in detail in chapters 5-7 of \cite{Kum}, when the base field is $k = \mathbb C$. Over general base fields, one can invoke the standard references such as Tits \cite{Tits81a, Tits81b}, Slodowy \cite{Slo84}, Matthieu \cite{Mat88, Mat89}, and Littelmann \cite{Lit}, to get the same structures and properties over our finite field $k$.  Granting this, one can deduce formally the Kac-Moody analogues of our Theorem \ref{dtm} and Corollary \ref{conv}, using either the contraction or the affine paving method.

Results in the Kac-Moody setting have been proved earlier by Bezrukavnikov-Yun: in fact \cite[Prop.\,3.2.5]{BY} seems to be the first place the semisimplicity and Frobenius semisimplicity of $IC_{w_1} * IC_{w_2}$  was proved, for $IC$-complexes for $B$-orbits on full flag varieties of Kac-Moody groups. Their argument is different from ours. Note that \cite{BY} does not imply our full result for two reasons: 1) $LG$ is not a Kac-Moody group when $G$ is not simply-connected, and 2) we consider all {\em partial} affine flag varieties attached to $LG$ for connected reductive groups $G$.

Achar-Riche have developed in \cite{AR} an abstract framework which implies Frobenius semisimplicity results in various concrete situations. However, it appears to us that their method does not prove our Theorem \ref{dtm} or Corollary \ref{conv} in general. The main difficulty seems to be that, in most cases, our convolution morphisms  $p: X_\mathcal P(w_\bullet) \rightarrow X_\mathcal P(w_\star)$ are {\bf not} stratified morphisms of affable spaces (in the sense of \cite[9.13]{AR}), for any natural choices of affine even stratifications on the source and target; for more discussion we refer to our earlier arXiv posting arXiv:1602.00645v2.

\subsection{Results over other fields $k$}\la{otherk}$\;$

The results in \S\ref{r42} concerning  generalized convolution morphisms and 
the surjectivity criterion Theorem \ref{tmff}  hold, by the usual specialization arguments over an arbitrary algebraically closed field if we replace \good \, with even. Over the complex numbers,
they hold if we replace \good \, with even {\em and} Tate and we use M. Saito's theory of mixed Hodge modules to state them.
At present, we do not see how to establish the surjectivity assertions  in Theorems \ref{dtm} and \ref{tmff} 
without using weights (Frobenius, or M. Saito's). The paving results  hold over any field. Theorem \ref{tmb} holds over any algebraically closed field and so does Corollary \ref{corbb}, with the same provisions 
as above. The construction of $L^{--}P_{\bf f}$ and Theorem \ref{big_cell_thm} hold at least over any perfect field.

\end{document}